\documentclass[a4paper]{amsart}
\pdfoutput=1
\usepackage[utf8]{inputenc}
\usepackage[T1]{fontenc}
\usepackage{amssymb,mathtools,url}
\usepackage{mathrsfs}
\usepackage{dsfont}
\usepackage{tensor} 
\usepackage{lmodern}

\usepackage[all]{xy}
\usepackage{microtype}

\usepackage{enumitem}
\setlist[enumerate,1]{label=\textup{(\arabic*)}}

\usepackage[lite]{amsrefs}

\numberwithin{equation}{section}

\theoremstyle{plain}
\newtheorem{thm}[subsection]{Theorem}
\newtheorem{lem}[subsection]{Lemma}
\newtheorem{cor}[subsection]{Corollary}
\newtheorem{prop}[subsection]{Proposition}
\theoremstyle{definition}
\newtheorem{defn}[subsection]{Definition}

\theoremstyle{remark}
\newtheorem{rem}[subsection]{Remark}
\newtheorem{example}[subsection]{Example}

\newcommand{\ZZ}{\mathbb{Z}}
\newcommand{\FF}{\mathbb{F}}

\newcommand{\NN}{\mathbb{N}}
\newcommand{\TT}{\mathbb{T}}
\newcommand{\CC}{\mathbb{C}}

\DeclareMathOperator{\obj}{ob}

\newcommand*{\nb}{\nobreakdash}
\newcommand*{\Star}{\(^*\)\nobreakdash-}
\newcommand{\Cst}{\mathrm{C}^*}
\newcommand{\idealin}{\mathrel{\triangleleft}} 
\newcommand*{\Bound}{\mathbb B}
\newcommand*{\Comp}{\mathbb K}
\newcommand*{\defeq}{\mathrel{\vcentcolon=}}
\newcommand{\Hilm}[1][E]{\mathcal{#1}}
\newcommand{\Toep}{\mathcal{T}}
\newcommand{\CP}{\mathcal{O}}
\newcommand{\Corr}{\mathfrak{C}}
\newcommand{\Bim}{\mathrm{pr,*}}
\newcommand{\proper}{\mathrm{pr}}
\newcommand{\id}{\mathrm{id}}

\newcommand*\xbar[1]{%
   \hbox{%
     \vbox{%
       \hrule height 0.5pt 
       \kern0.5ex
       \hbox{%
         \kern-0.1em
         \ensuremath{#1}%
         \kern-0.1em
       }%
     }%
   }%
} 

\DeclarePairedDelimiter{\bra}{\langle}{\rvert}
\DeclarePairedDelimiter{\ket}{\lvert}{\rangle}
\DeclarePairedDelimiterX{\braket}[2]{\langle}{\rangle}{#1\,\delimsize\vert\,\mathopen{}#2}
\DeclarePairedDelimiterX{\BRAKET}[2]{\langle}{\rangle}{\!\delimsize\langle#1\,\delimsize\vert\,\mathopen{}#2\delimsize\rangle\!}

\DeclarePairedDelimiterX{\setgiven}[2]{\{}{\}}{#1\,{:}\,\mathopen{}#2}

\begin{document}

\title[Fell bundles over quasi-lattice ordered groups]{Fell bundles over quasi-lattice ordered groups and $\mathrm{C}^*$-algebras of compactly aligned product systems}

\thanks{The author was supported by CNPq (Brazil) through PhD fellowship 248938/2013-4 and CAPES/PrInt, grant number 88887.370650/2019-00.}

\author{Camila F. Sehnem}
\address{School of Mathematics and Statistics, Victoria University of Wellington, P.O. Box 600, Wellington 6140, New Zealand}
\email{camila.sehnem@vuw.ac.nz}

\keywords{relative Cuntz--Pimsner algebra; product system; semi-saturated Fell bundle; simplifiable product system of Hilbert bimodules}

\begin{abstract}
We define notions of semi-saturatedness and orthogonality for a Fell bundle over a quasi-lattice ordered group. We show that a compactly aligned product system of Hilbert bimodules can be naturally extended to a semi-saturated and orthogonal Fell bundle whenever it is simplifiable. Conversely, a semi-saturated and orthogonal Fell bundle is completely determined by the positive fibres and its cross-sectional $\Cst$\nb-algebra is isomorphic to a relative Cuntz--Pimsner algebra of a simplifiable product system of Hilbert bimodules. We show that this correspondence is part of an equivalence between bicategories and use this to generalise several results of Meyer and the author in the context of single correspondences. We apply functoriality for relative Cuntz--Pimsner algebras to study Morita equivalence between $\Cst$\nb-algebras attached to compactly aligned product systems over Morita equivalent $\Cst$\nb-algebras. 
\end{abstract}

\maketitle

\section{Introduction}

 Given a discrete group~$G$ and a subsemigroup~$P$ of~$G$, in some situations a partial representation of~$G$ on a $\Cst$\nb-algebra may be reconstructed from its restriction to~$P$. When $G=P^{-1}P$, a representation of~$P$ by isometries can be extended to a partial representation of~$G$ if and only if the range projections associated to the isometries commute with each other \cite[Theorem~31.16]{Exel:Partial_dynamical}. Such an extension is unique and so isometric representations of~$P$ with commuting range projections are in one-to-one correspondence with partial representations of~$G$ which restrict to isometric representations of~$P$. A similar correspondence holds, for example, in the context of circle-valued cocycles under the same assumption $G=P^{-1}P$~\cite{Laca,10.2307/2160888}.

In this paper, we are primarily interested in extensions of semigroup actions by Hilbert bimodules. We assume that~$P$ is a subsemigroup of~$G$ so that~$(G,P)$ is a quasi-lattice ordered group in the sense of Nica~\cite{Nica:Wiener--hopf_operators}. In this case any element in~$PP^{-1}$ has a reduced expression of the form $pq^{-1}$, with $p,q\in P$. We use this to give sufficient conditions for a compactly aligned product system over~$P$ of Hilbert bimodules to extend uniquely to a Fell bundle over the group~$G$, and to identify which Fell bundles over~$G$ can be reconstructed in this way by its positive fibres. We believe this may be profitable in several ways. For instance, $P$ may embed into another group~$H$ with~$H$ amenable even for a non-amenable~$G$~\cite{Hochster}. This may be used to establish amenability of the Fell bundle, that is, an isomorphism between full and reduced cross-sectional $\Cst$\nb-algebras through the regular representation.

A Fell bundle $(B_g)_{g\in G}$ over a discrete group is called saturated if $B_g\cdot B_h=B_{gh}$ for all~$g,h\in G$. In this case, each fibre may be viewed as an imprimitivity bimodule over the unit fibre algebra~$B_e$ with the structure coming from the multiplication and involution operations on~$(B_g)_{g\in G}$. In the general case of a non-saturated Fell bundle, each fibre carries a structure of a Hilbert $B_e$\nb-bimodule.

 Exel introduced in \cite{Exel:Circle_actions} a notion of semi-saturatedness for actions of the unit circle~$\TT$. An action of~$\TT$ on a $\Cst$\nb-algebra~$B$ is said to be \emph{semi-saturated} if $B$ is generated as a $\Cst$\nb-algebra by the spectral subspace~$B_1$ and the fixed-point algebra~$B_0$. He proved that this is the case for the canonical action of~$\TT$ on the crossed product of a $\Cst$\nb-algebra by a partial automorphism. A Fell bundle $(B_n)_{n\in\ZZ}$ is then called semi-saturated if $B_m\cdot B_{n}=B_{m+n}$ for all~$m,n\geq 0$ (equivalently, $m,n\leq 0$). Abadie, Eilers and Exel proved that the cross-sectional $\Cst$\nb-algebra of a semi-saturated Fell bundle~$(B_n)_{n\in \ZZ}$ is canonically isomorphic to the crossed product of the unit fibre~$B_0$ by the Hilbert $B_0$\nb-bimodule~$B_1$. They also provided examples to illustrate that the crossed product of a $\Cst$\nb-algebra by a Hilbert bimodule need not come from a partial automorphism. The crossed product by a Hilbert bimodule is a special case of a Cuntz--Pimsner algebra~\cite{Pimsner:Generalizing_Cuntz-Krieger, Katsura:Cstar_correspondences, Muhly-Solel:Tensor}.

Despite the fact that a free group~$\FF$ on more than one generator is not amenable, there is a non-trivial class of amenable Fell bundles over~$\FF$. A Fell bundle $(B_g)_{g\in \FF}$ over the free group~$\FF$ on the set of generators~$S$ is said to be \emph{semi-saturated} in~\cite{Exel:Partial_amenable_free} if $B_g\cdot B_h=B_{gh}$ whenever the product $g\cdot h$ involves no cancellation. It is called \emph{orthogonal} if $B_s^*B_t=\{0\}$ for all distinct generators $s,t\in S$. A Fell bundle over~$\FF$ with separable fibres that is semi-saturated and orthogonal is then amenable \cite[Theorem~6.3]{Exel:Partial_amenable_free}. This happens because~$\FF$ has a \emph{quasi-lattice ordered group} structure arising from the free unital subsemigroup~$\FF^+$ generated by~$S$, and $(\FF,\FF^+)$ is amenable in the sense of Nica~\cite{Nica:Wiener--hopf_operators}. A semi-saturated and orthogonal partial representation of~$\FF$ as introduced by Exel is completely determined by its restriction to~$\FF^+$.

In this paper, we introduce concepts of semi-saturatedness and orthogonality for Fell bundles over quasi-lattice orders. Recall that a pair~$(G,P)$ with $P\cap P^{-1}=\{e\}$ is \emph{quasi-lattice ordered} if whenever a pair of elements $g,h\in G$ has a common upper bound in~$P$, then it also has a least upper bound $g\vee h$ in~$P$. We say that~$(B_g)_{g\in G}$ is \emph{orthogonal} if~$B_g=\{0\}$ for all~$g\in G$ with no upper bound in~$P$. That is, $g\vee e=\infty$. Semi-saturatedness is defined so that, when combined with orthogonality, the restriction of~$(B_g)_{g\in G}$ to the positive fibres gives a product system of Hilbert bimodules~$\mathcal{B}=(B_p)_{p\in P}$ which generates $(B_g)_{g\in G}$ in an appropriate sense (see Definition~\ref{def:semi-saturatedness}). 

Our main motivation to define these notions of semi-saturatedness and orthogonality for Fell bundles over quasi-lattice orders was the canonical topological~$G$\nb-grading of a Nica--Toeplitz algebra of a compactly aligned product system over~$P$, and its quotients by gauge-invariant ideals (see Section~\ref{sec:preliminaries}). In the context of single correspondences, the Fell bundle~$(\CP^n_{J,\Hilm})_{n\in\ZZ}$ over~$\ZZ$ associated to a relative Cuntz--Pimsner algebra is always semi-saturated. And $(\CP^0_{J,\Hilm},\CP^1_{J,\Hilm})=(A,\Hilm)$ if and only if~$\Hilm$ is a Hilbert bimodule over~$A$ and~$J=\BRAKET{\Hilm}{\Hilm}$ is Katsura's ideal, that is, $J$ is the largest ideal acting faithfully and by compact operators on~$\Hilm$ \cite[Proposition~5.18]{Katsura:Cstar_correspondences}. 

 It turns out that the Fell bundle $(\mathcal{N}\Toep^g_{\Hilm})_{g\in G}$ associated to a Nica--Toeplitz algebra is semi-saturated and orthogonal. This is also the case for a relative Cuntz--Pimsner algebra of~$\Hilm$ when we define it as a quotient of~$\mathcal{N}\Toep_{\Hilm}$. Following ideas for~$(G,P)=(\ZZ,\NN)$, we find precisely the class of compactly aligned product systems of Hilbert bimodules that can be extended to semi-saturated and orthogonal Fell bundles over~$G$ as above. We call a product system in this class \emph{simplifiable}. In Theorem~\ref{thm:uniqueness_and_extension}, we obtain an extension of a simplifiable product system of Hilbert bimodules to a semi-saturated and orthogonal Fell bundle over~$G$. This also arises from the canonical $G$\nb-grading of the relative Cuntz--Pimsner algebra determined by the family of Katsura's ideals for the product system, and we prove that it is unique up to isomorphism.

The restriction of a semi-saturated and orthogonal Fell bundle~$(B_g)_{g\in G}$ to the positive fibres gives a product system of Hilbert bimodules. We show in Proposition~\ref{prop:associated_product_system} that this product system is simplifiable, and its relative Cuntz--Pimsner algebra for the family of Katsura's ideals is canonically isomorphic to the cross-sectional $\Cst$\nb-algebra of~$(B_g)_{g\in G}$. Hence representations of~$(B_g)_{g\in G}$ are in bijection with Cuntz--Pimsner covariant representations of the product system~$(B_p)_{p\in P}$ on the family of Katsura's ideals. So we say that a semi-saturated and orthogonal Fell bundle over~$G$ is \emph{extended} from~$P$.

As the work of Exel \cite{Exel:Partial_amenable_free} already suggested, amenability of a Fell bundle extended from the positive cone of a quasi-lattice order seems to be related to amenability of the underlying quasi-lattice. We prove that a Fell bundle extended from~$\FF^+$ can be described as a Cuntz--Pimsner algebra of a correspondence over the unit fibre~$B_e$ and hence it is amenable (see Lemma~\ref{lem:cp_picture} and Proposition~\ref{prop:amenability_free_groups}). In particular, the main result of~\cite{Exel:Partial_amenable_free} remains true if we remove the separability assumption. This allows us to establish nuclearity of the cross-sectional $\Cst$\nb-algebra as well when the unit fibre algebra is nuclear. As for free semigroups, a Fell bundle extended from the Baumslag--Solitar semigroup $\mathrm{BS}(c,d)^+$ with $c,d\geq 1$ is isomorphic to a Cuntz--Pimsner algebra of a single correspondence and so it is also amenable. In general we do not know whether a Fell bundle extended from the positive cone of an amenable quasi-lattice order is amenable too.

In Section~\ref{sec:bicategorical-part} we consider bicategories of compactly aligned product systems.  In Theorem~\ref{thm:induced_equivalence}, we build an equivalence between a bicategory of simplifiable product systems of Hilbert bimodules over~$P$ and a bicategory of Fell bundles over~$G$ extended from~$P$ that sends a product system to its unique extension from Theorem~\ref{thm:uniqueness_and_extension}. We observe that, as opposed to Fowler's original definition of a Cuntz--Pimsner algebra, we take the relative Cuntz--Pimsner algebra of a product system~$\Hilm=(\Hilm_p)_{p\in P}$ as a quotient of the Nica--Toeplitz algebra of~$\Hilm$. The relative Cuntz--Pimsner algebra $\CP_{\mathcal{J},\Hilm}$ is determined by a family of ideals $\mathcal{J}=\{J_p\}_{p\in P}$ in the coefficient algebra~$A$, where each ideal~$J_p$ acts by compact operators on~$\Hilm_p$, and the canonical representation of~$\Hilm$ in~$\CP_{\mathcal{J},\Hilm}$ is also Nica covariant. This allows us to attach a simplifiable product system of Hilbert bimodules~$(\CP^p_{\mathcal{J},\Hilm})_{p\in P}$ with coefficient algebra~$\CP^e_{\mathcal{J},\Hilm}$ to each triple $(A,\Hilm,\mathcal{J})$. 

A triple $(A,\Hilm,\mathcal{J})$ as above is an object of our bicategory $\Corr^P_{\proper}$. A morphism from $(A,\Hilm,\mathcal{J}_A)$ to $(B,\mathcal{G},\mathcal{J}_B)$ is a proper covariant correspondence $(\Hilm[F],V)$ (see Definition~\ref{defn:cov-correspondence}). We generalise several results of Meyer and the author~\cite{Meyer-Sehnem:Bicategorical_Pimsner} to compactly aligned product systems over positive cones of quasi-lattice orders using the equivalence between Fell bundles extended from~$P$ and simplifiable product systems of Hilbert bimodules over~$P$. In particular, we show that the construction of a relative Cuntz--Pimsner algebra is part of a reflector from a bicategory of compactly aligned product systems to a bicategory of simplifiable product systems of Hilbert bimodules. That is, the reflector sends an object $(A,\Hilm,\mathcal{J})$ to $(\CP^p_{\mathcal{J},\Hilm})_{p\in P}$. This is \cite[Corollary~4.7]{Meyer-Sehnem:Bicategorical_Pimsner} in the context of single correspondences. 

In Subsection~\ref{subsec: Morita}, we study Morita equivalence for relative Cuntz--Pimsner algebras. The construction of a relative Cuntz--Pimsner algebra induces a functor from $\Corr^P_{\proper}$ into a bicategory of $\Cst$\nb-algebras with correspondences as morphisms. Since an arrow between objects $(A,\Hilm,\mathcal{J}_A)$ and $(B,\mathcal{G},\mathcal{J}_B)$ comes from a correspondence $\Hilm[F]\colon A\leadsto B$, invertible morphisms in $\Corr^P_{\proper}$ imply Morita equivalence between the underlying relative Cuntz--Pimsner algebras. This is~\cite{Abadie-Eilers-Exel:Morita_bimodules, Muhly-Solel:Morita_equivalence_of_tensor_algebras} for~$(G,P)=(\ZZ,\NN)$. When $\tensor*[_\alpha]{A}{}$ and $\tensor*[_\beta]{B}{}$ are the product systems of Hilbert bimodules built out of actions~$\alpha, \beta$ of~$P$ on~$A$ and $B$ by injective and extendible endomorphisms with hereditary range, we generalise~\cite[Proposition 2.4]{Muhly-Solel:Morita_equivalence_of_tensor_algebras} of Muhly and Solel and characterise the invertible proper covariant correspondences in $\Corr^P_{\Bim}$ between $\tensor*[_\alpha]{A}{}$ and $\tensor*[_\beta]{B}{}$ (see Proposition~\ref{prop:charact_extendible}). Our result is an analogue of Morita equivalence for actions of groups by Combes~\cite{10.1112/plms/s3-49.2.289} and Curto, Muhly and Williams~\cite{CMW}.

\subsection*{Acknowledgements} The content of this article is part of my PhD dissertation, written at the University of Göttingen, under the supervision of Ralf Meyer. I thank Ralf Meyer for his support during my PhD studies. I also would like to thank Aidan Sims for an observation that led me to include Corollary~\ref{cor:cp-cross-section}. A part of this article was completed during a post-doctoral position at Universidade Federal de Santa Catarina.

\section{Relative Cuntz--Pimsner algebras of compactly aligned product systems}\label{sec:preliminaries}

 In this section we recall some basic concepts and known results on product systems and their $\Cst$-algebras. We will be interested in compactly aligned product systems over semigroups that are positive cones of quasi-lattice orders.

\subsection{Notation and basic notions} Let $A$ and $B$ be $\Cst$\nb-algebras. A \emph{correspondence} $\Hilm\colon A\leadsto B$ is a Hilbert $B$\nb-module~$\Hilm$ with a nondegenerate left action of~$A$ implemented by a \Star homomorphism $\varphi\colon A\to\Bound(\Hilm)$. We say that $\Hilm$ is a \emph{Hilbert $A,B$\nb-bimodule} if the left action of~$A$ comes from a left Hilbert $A$\nb-module structure $\BRAKET{\cdot}{\cdot}_A$ on~$\Hilm$ that is compatible with $\braket{\cdot}{\cdot}_B$. That is, $\BRAKET{\xi}{\eta}\zeta=\xi\braket{\eta}{\zeta}$ for all $\xi,\eta,\zeta\in\Hilm$.

 Let~$P$ be a semigroup with identity~$e$. A \emph{product system} over~$P$ of $A$\nb-correspon\-dences consists of:
\begin{enumerate}
\item[(i)] a correspondence $\Hilm_p\colon A\leadsto A$ for each~$p\in P$, where $\Hilm_e=A$  is the identity correspondence over~$A$;
\item[(ii)] correspondence isomorphisms $\mu_{p,q}\colon\Hilm_p\otimes_A\Hilm_q\overset{\cong}{\rightarrow}\Hilm_{pq}$, also called \emph{multiplication maps}, for all $p,q\in P\setminus\{e\}$;
\end{enumerate}

We let~$\varphi_p\colon A\to\Bound(\Hilm_p)$ be the multiplication map~$\mu_{e,p}$ and let~$\mu_{p,e}$ implement the right action of~$A$ on~$\Hilm_p$, respectively. The multiplication maps must be associative. That is, the following diagram commutes for all $p, q, r\in P$:
  \[
  \xymatrix{
    (\Hilm_p\otimes_A\Hilm_q)\otimes_A\Hilm_r  \ar@{->}[d]^{\mu_{p,q}\otimes1}
   \ar@{<->}[rr]& &     \Hilm_p\otimes_A(\Hilm_q\otimes_A\Hilm_r)
   \ar@{->}[rr]^{1\otimes\mu_{q,r}}&&
    \Hilm_p\otimes_A\Hilm_{qr} \ar@{->}[d]^{\mu_{p,qr}} \\
    \Hilm_{pq}\otimes_A\Hilm_r   \ar@{->}[rrrr]^{\mu_{pq,r}}&& &&
    \Hilm_{pqr},
  }
  \] where the first isomorphism on the top row of the diagram is simply the associativity isomorphism for internal tensor products of Hilbert modules. We will say that a product system $\Hilm=(\Hilm_p)_{p\in P}$ is \emph{faithful} if~$\varphi_p$ is injective for all~$p\in P$. It is \emph{proper} if~$\varphi_p(A)\subseteq \Comp(\Hilm_p)$ for all~$p$ in~$P$. If each~$\Hilm_p$ is a Hilbert $A$\nb-bimodule, we will speak of a \emph{product system of Hilbert bimodules}.

 A \emph{representation} of a product system~$\Hilm=(\Hilm_p)_{p\in P}$ in a $\Cst$\nb-algebra~$B$ consists of linear maps~$\psi_p\colon\Hilm_p\rightarrow B$, for all~$p\in P\setminus\{e\},$ and a \Star homomorphism $\psi_e\colon A\rightarrow B$, satisfying the following two axioms:
  \begin{enumerate}
  \item[(T1)] $\psi_p(\xi)\psi_q(\eta)=\psi_{pq}(\xi\eta)$ for all $p,q\in P$, $\xi\in\Hilm_p$ and  $\eta\in\Hilm_q$;
  \item[(T2)] $\psi_p(\xi)^*\psi_p(\eta)=\psi_e(\braket{\xi}{\eta})$ for all $p\in P$ and $\xi, \eta\in\Hilm_p$.  
  \end{enumerate}
    
 If $\psi_e$ is faithful, we say that~$\psi$ is \emph{injective}. In this case, relation (T2) implies that~$\|\psi_p(\xi)\|=\|\xi\|$ for all~$\xi\in\Hilm_p$ and~$p\in P$.

\subsection{Compactly aligned product systems and Nica--Toeplitz algebras} Let us restrict our attention to semigroups arising from quasi-lattice orders in the sense of~\cite{Nica:Wiener--hopf_operators}: let~$G$ be a group and let~$P$ be a subsemigroup of~$G$ with~$P\cap P^{-1}=\{e\}$. We say that~$(G,P)$ is a \emph{quasi-lattice ordered group} if any two elements~$g_1, g_2$ of~$G$ with a common upper bound in~$P$ with respect to the partial order $g_1\leq g_2\Leftrightarrow g_1^{-1}g_2\in P$ also have a least upper bound~$g_1\vee g_2$ in~$P$. We write $g_1\vee g_2=\infty$ if~$g_1$ and~$g_2$ have no common upper bound in~$P$. Following~\cite{crisp_laca_2002}, we call~$P$ the \emph{positive cone} of~$(G,P)$, observing that~$P=\left\{g\in G\middle|\, g\geq e\right\}$.

Let~$(G,P)$ be a quasi-lattice ordered group and let~$\Hilm=(\Hilm_p)_{p\in P}$ be a product system over~$P$. For~$p\in P$ and $\xi, \eta\in \Hilm_p$, we denote by $\ket{\xi}\bra{\eta}$ the generalised rank\nobreakdash-$1$ operator on~$\Hilm_p$ that sends $\zeta\in\Hilm_p$ to $\xi\braket{\eta}{\zeta}\in\Hilm_p$. Let~$\psi=\{\psi_p\}_{p\in P}$ be a representation of~$\Hilm$ in a $\Cst$\nb-algebra~$B$. For each~$p\in P$, we will denote by~$\psi^{(p)}$ the \Star homomorphism from~$\Comp(\Hilm_p)$ to~$B$ obtained as in~\cite{Pimsner:Generalizing_Cuntz-Krieger}. This is defined on a generator~$\ket{\xi}\bra{\eta}$ by $$\psi^{(p)}\big(\ket{\xi}\bra{\eta}\big)\coloneqq\psi_p(\xi)\psi_p(\eta)^*.$$
We may use the multiplication maps on~$\Hilm$ to define \Star homomorphisms  $\iota_p^{pq}\colon\Bound(\Hilm_p)\allowbreak\to\Bound(\Hilm_{pq})$. Explicitly, $\iota_p^{pq}$ sends~$T\in\Bound(\Hilm_p)$ to~$\mu_{p,q}\circ( T\otimes\id_{\Hilm_q})\circ\mu_{p,q}^{-1}$. We say that~$\Hilm=(\Hilm_p)_{p\in P}$ is \emph{compactly aligned} if, for all~$p, q\in P $ with~$p\vee q<\infty$, we have $$\iota_p^{p\vee q}(T)\iota_q^{p\vee q}(S)\in\Comp(\Hilm_{p\vee q}),\qquad\text{for all } T\in\Comp(\Hilm_p)\text{ and } S\in\Comp(\Hilm_q).$$ Here $\iota_p^{p\vee q}$ and $\iota_q^{p\vee q}$ denote the \Star homomorphisms $\iota_p^{p\cdot p^{-1}(p\vee q)}\colon\Bound(\Hilm_p)\to\Bound(\Hilm_{p\vee q})$ and $\iota_q^{q\cdot q^{-1}(p\vee q)}\colon\Bound(\Hilm_q)\to\Bound(\Hilm_{p\vee q}),$ respectively.

If~$\Hilm$ is compactly aligned, a representation~$\psi=\{\psi_p\}_{p\in P}$ of~$\Hilm$ in a $\Cst$\nb-algebra~$B$ is \emph{Nica covariant} if, for all~$p,q\in P$, $T\in \Comp(\Hilm_p)$ and~$S\in\Comp(\Hilm_q)$, we have $$\psi^{(p)}(T)\psi^{(q)}(S)=\begin{cases} \psi^{(p\vee q)}\big(\iota_p^{p\vee q}(T)\iota_q^{p\vee q}(S)\big)  &\text{if } p\vee q<\infty,\\
0 & \text{otherwise.}
\end{cases}$$

 There is a canonical Nica covariant representation associated to a compactly aligned product system: let~$\Hilm^+$ be the right Hilbert $A$\nb-module given by the direct sum of all $\Hilm_p$'s. That is, $$\Hilm^+=\bigoplus_{\substack{p\in P}}\Hilm_p.$$  Define a representation of~$\Hilm$ in~$\Bound(\Hilm^+)$ as follows. Given~$\xi\in \Hilm_p$ and $\eta^+=\bigoplus_{\substack{s\in P}}\eta_s$, set $$\psi^+_p(\xi)(\eta^+)_s=\begin{cases}\mu_{p,p^{-1}s}(\xi\otimes\eta_s)&\text{if } s\in pP,\\
0  &\text{otherwise}.
\end{cases}$$
We view~$\Hilm_{ps}$ as the correspondence $\Hilm_p\otimes_A\Hilm_s$ through the correspondence isomorphism~$\mu_{p,s}^{-1}$. In this way, $\psi_p^+(\xi)^*(\eta)_s$ is the image of~$\eta_{ps}$ in~$\Hilm_s$ under the operator defined on elements of the form $\mu_{p,s}(\zeta_p\otimes\zeta_s)$ by the formula $$\psi^+_p(\xi)^*(\mu_{p,s}(\zeta_p\otimes\zeta_s))=\varphi_s(\braket{\xi}{\zeta_p})\zeta_s.$$ So~$\psi_p^+(\xi)^*$ is the adjoint of~$\psi_p^+(\xi)$. This gives rise to a Nica covariant representation $\psi^+=\{\psi^+_p\}_{p\in P}$ of~$\Hilm$ in~$\Bound(\Hilm^+)$ called the \emph{Fock representation} of~$\Hilm$. This representation is injective.

\begin{prop}[\cite{Fowler:Product_systems}*{Theorem 6.3}]\label{prop:defn_Nica_Toep} Let~$(G,P)$ be a quasi-lattice ordered group and let~$\Hilm$ be a compactly aligned product system over~$P$. Then there is a $\Cst$\nb-algebra~$\mathcal{N}\Toep_{\Hilm}$ and a Nica covariant representation~$\bar{\pi}=\{\bar{\pi}_p\}_{p\in P}$ of~$\Hilm$ in~$\mathcal{N}\Toep_{\Hilm}$ so that $\bar{\pi}(\Hilm)$ generates~$\mathcal{N}\Toep_{\Hilm}$ as a $\Cst$\nb-algebra and, given a Nica covariant representation~$\psi=\{\psi_p\}_{p\in P}$ of~$\Hilm$ in a $\Cst$\nb-algebra~$B$, there is a unique \Star homomorphism~$\bar{\psi}\colon\Toep_{\Hilm}\to B$ such that~$\bar{\psi}\circ\bar{\pi}_p=\psi_p$ for all~$p\in P$. Moreover, $\bar{\pi}$ is injective and the pair~$(\mathcal{N}\Toep_{\Hilm},\bar{\pi})$ is unique up to canonical isomorphism. 
\end{prop}

 We call~$\Hilm[N]\Toep_{\Hilm}$ the \emph{Nica--Toeplitz algebra} of~$\Hilm$.  

\subsection{Relative Cuntz--Pimsner algebras} Let~$\Hilm=(\Hilm_p)_{p\in P}$ be a product system. For each~$p\in P$, let~$J_p\idealin A$ be an ideal that acts by compact operators on~$\Hilm_p$ and set~$\mathcal{J}=\{J_p\}_{p\in P}$. We say that a representation~$\psi=\{\psi_p\}_{p\in P}$ is \emph{Cuntz--Pimsner covariant} on~$\mathcal{J}$ if, for all~$p\in P$ and  all~$a$ in~$J_p$, $$\psi^{(p)}(\varphi_p(a))=\psi_e(a).$$

\begin{prop} Let~$(G,P)$ be a quasi-lattice ordered group and let~$\Hilm$ be a compactly aligned product system over~$P$. Let~$\mathcal{J}=\{J_p\}_{p\in P}$ be a family of ideals in~$A$ with~$\varphi_p(J_p)\subseteq\Comp(\Hilm_p)$ for all~$p\in P$. Then there is a $\Cst$\nb-algebra~$\CP_{\Hilm[J],\Hilm}$ and a Nica covariant representation~$j=\{j_p\}_{p\in P}$ of~$\Hilm$ in~$\CP_{\Hilm[J],\Hilm}$ that is also Cuntz--Pimsner covariant on~$\Hilm[J]$ and such that
\begin{enumerate}
\item[\textup{(i)}]~$\CP_{\Hilm[J],\Hilm}$ is generated by~$j(\Hilm)$ as a $\Cst$\nb-algebra;

\item[\textup{(ii)}] given a Nica covariant representation~$\psi=\{\psi_p\}_{p\in P}$ of~$\Hilm$ in a $\Cst$\nb-algebra $B$ that is Cuntz--Pimsner covariant on~$\Hilm[J]$, there is a unique \Star homomor\-phism $\bar{\psi}_{\Hilm[J]}\colon\CP_{\Hilm[J],\Hilm}\to B$ such that~$\bar{\psi}_{\Hilm[J]}\circ j_p=\psi_p$ for all~$p\in P$.
\end{enumerate}
 \noindent Moreover, the pair~$(\CP_{\Hilm[J],\Hilm},j)$ is unique up to canonical isomorphism. 

\begin{proof} Let~$\mathcal{N}\Toep_{\Hilm}$ be the Nica--Toeplitz algebra of~$\Hilm$. Let~$\CP_{\Hilm[J],\Hilm}$ be the quotient of $\mathcal{N}\Toep_{\Hilm}$ by the ideal generated by $$\{ \bar{\pi}_e(a)-\bar{\pi}^{(p)}(\varphi_p(a))\mid a\in J_p, p\in P\}.$$ For each~$p\in P$, we let $j_p\colon\Hilm_p\to\CP_{\Hilm[J],\Hilm}$ be the composition of $\bar{\pi}_p$ with the quotient map and we set~$j=\{j_p\}_{p\in P}$. The pair $(\CP_{\Hilm[J],\Hilm},j)$ satisfies the required properties.
\end{proof}
  \end{prop}

\begin{defn} Given~$\Hilm$ and~$\mathcal{J}$ as above, we call~$\CP_{\Hilm[J],\Hilm}$ the \emph{relative Cuntz--Pimsner algebra} determined by~$\mathcal{J}$.
\end{defn}

We emphasize two particular cases. If~$J_p=\{0\}$ for all~$p\in P$, then~$\CP_{\Hilm[J],\Hilm}=\mathcal{NT}_{\Hilm}$. As a consequence of \cite[Corollary~3.7]{Pimsner:Generalizing_Cuntz-Krieger}, if~$(G,P)=(\ZZ,\NN)$, $\Hilm$ is a product system of Hilbert bimodules and~$J_p=\BRAKET{\Hilm_p}{\Hilm_p}$ for all $p$ in~$P$, then~$\CP_{\Hilm[J],\Hilm}$ is the $\Cst$\nb-algebra studied by Katsura in~\cite{Katsura:Cstar_correspondences}. He proved that the canonical \Star homomorphism from~$A$ to~$\CP_{\Hilm[J],\Hilm}$ is an isomorphism onto the fixed-point algebra of~$\CP_{\Hilm[J],\Hilm}$ with respect to the gauge action of~$\TT$. In this case, $\Hilm$ extends to a semi-saturated Fell bundle over~$\ZZ$ (see~\cite{Abadie-Eilers-Exel:Morita_bimodules}). We will generalise this to a certain class of compactly aligned product systems of Hilbert bimodules over semigroups arising from quasi-lattice orders.

\begin{rem}\label{rem:fowlers_construction} Fowler defined the Cuntz--Pimsner algebra of a product system~$\Hilm$ to be the universal $\Cst$\nb-algebra for representations of~$\Hilm$ that are Cuntz--Pimsner covariant on $\Hilm[J]=\{J_p\}_{p\in P}$, where $J_p=\varphi_p^{-1}(\Comp(\Hilm_p))$ for all~$p\in P$ (see~\cite{Fowler:Product_systems}). Here we consider the class of compactly aligned product systems and define the relative Cuntz--Pimsner algebra with respect to a family of ideals as a quotient of the Nica--Toeplitz algebra of~$\Hilm$. This provides the construction of relative Cuntz--Pimsner algebras with a special feature and will allow us to generalise most of the results obtained in \cite{Meyer-Sehnem:Bicategorical_Pimsner} to quasi-lattice ordered groups. Our approach applies to Fowler's Cuntz--Pimsner algebras of \emph{proper} product systems~$\Hilm=(\Hilm_p)_{p\in P}$ if~$(G,P)$ is a quasi-lattice ordered group and~$P$ is directed. This is so because, in this case, a Cuntz--Pimsner covariant representation of~$\Hilm$ in the sense of Fowler is also Nica covariant~\cite[Proposition 5.4]{Fowler:Product_systems}.
\end{rem}

The next is the analogue of \cite[Proposition~2.15]{Meyer-Sehnem:Bicategorical_Pimsner} in the context of product systems.

\begin{prop}\label{prop:equivalent_covariant_cond}
  A representation~$\psi=\{\psi_p\}_{p\in P}$ of~$\Hilm=(\Hilm_p)_{p\in P}$
  in~\(B\)
  is Cuntz--Pimsner covariant on~\(\Hilm[J]=\{J_p\}_{p\in P}\)
  if and only if \(\psi_e(J_p) \subseteq \psi_p(\Hilm_p)\cdot B\) for all~$p\in P$.
  \label{pro:covariant_through_subspaces}
\end{prop}

\subsection{Coaction on relative Cuntz--Pimsner algebras} Let~$(G,P)$ be a quasi-lattice order and let~$\Hilm$ be a compactly aligned product system over~$P$. The representation of~$\Hilm$ in~$\CP_{\Hilm[J],\Hilm}\otimes \Cst(G)$ which sends~$\xi\in\Hilm_p$ to~$\xi\otimes u_p$ is Nica covariant and also Cuntz--Pimsner covariant on~$\Hilm[J]$. So this yields a \Star homomorphism $\delta\colon\CP_{\Hilm[J],\Hilm}\to\CP_{\Hilm[J],\Hilm}\otimes \Cst(G)$.

\begin{prop}\label{prop:CP_grading_description} The \Star homomorphism $\delta\colon\CP_{\Hilm[J],\Hilm}\to\CP_{\Hilm[J],\Hilm}\otimes \Cst(G)$ gives a full nondegenerate coaction of~$G$ on~$\CP_{\Hilm[J],\Hilm}$. Moreover, the spectral subspace~$\CP^g_{\Hilm[J],\Hilm}$ for~$\delta$ at~$g\in G$ with~$g\vee e<\infty$ is the closure of sums of elements of the form $$j_p(\xi)j_q(\eta)^*$$ with~$\xi\in\Hilm_p$ and~$\eta\in\Hilm_q,$ where $pq^{-1}=g$ and $p,q\in P$. If~$g\vee e=\infty$, then~$\CP^g_{\Hilm[J],\Hilm}$ is the trivial subspace.
\end{prop} 
\begin{proof} That $\delta$ is a full nondegenerate coaction follows as in~\cite[Proposition~3.5]{Carlsen-Larsen-Sims-Vittadello:Co-universal}. For the last part of the statement, notice that the Nica covariance condition entails~$j_p(\Hilm_p)^*j_q(\Hilm_q)=\{0\}$ whenever~$p\vee q=\infty$ as $j_p(\Hilm_p)=j_p(\Hilm_p)j_p(\Hilm_p)^*j_p(\Hilm_p)$ for all~$p\in P$. In case~$p\vee q<\infty$, we have 
$$j_p(\Hilm_p)^*j_q(\Hilm_q)\subseteq\overline{\mathrm{span}}\left\{j_{p^{-1}(p\vee q)}(\xi)j_{q^{-1}(p\vee q)}(\eta)^*\middle|\, \xi\in\Hilm_{p^{-1}(p\vee q)}, \eta\in\Hilm_{q^{-1}(p\vee q)}\right\}.$$ So take~$g\in G$ with~$g\vee e=\infty$. In particular, $g$ has no presentation of the form~$pq^{-1}$ with~$p,q$ in~$P$. Thus, by successive applications of the above simplification for elements of the form~$j_p(\xi_p)^*j_q(\xi_q)$, it follows that $$j_{p_1}(\xi_{p_1})j_{p_2}(\xi_{p_2})^*\ldots j_{p_{2n-1}}(\xi_{p_{2n-1}})j_{p_{2n}}(\xi_{p_{2n}})^*=0$$ whenever~$p_1p_2^{-1}\ldots p_{2n-1}p_{2n}^{-1}=g$ and~$\xi_{p_i}\in \Hilm_{p_i}$ for all $i\in\{1,2,\ldots,2n\}$. As a consequence, 
$\CP^g_{\Hilm[J],\Hilm}=\{0\}$ because it is spanned by elements of that form. Now a similar reasoning shows that if~$g\in G$ satisfies~$g\vee e<\infty$ and~$p_1p_2^{-1}\ldots p_{2n-1}p_{2n}^{-1}=g$, then $j_{p_1}(\xi_{p_1})j_{p_2}(\xi_{p_2})^*\ldots j_{p_{2n-1}}(\xi_{p_{2n-1}})j_{p_{2n}}(\xi_{p_{2n}
})^*$ lies in the closed subspace spanned by $$\left\{j_p(\xi)j_q(\eta)^*\middle|\,  pq^{-1}=g,\, \xi\in\Hilm_p\text{ and }\eta\in\Hilm_q\right\}.$$ This completes the proof.
\end{proof}

\begin{cor}\label{cor:positive_fibres_isomorphism} Let~$\Hilm=(\Hilm_p)_{p\in P}$ be a compactly aligned product system and~$\Hilm[J]$ as above. Then, for all~$p\in P$, we have an isomorphism $$\CP^p_{\Hilm[J],\Hilm}\cong \Hilm_p\otimes_A\CP^e_{\Hilm[J],\Hilm}$$ of correspondences $A\leadsto \CP^e_{\Hilm[J],\Hilm}$. Moreover, 
$(\CP^p_{\Hilm[J],\Hilm})_{p\in P}$ is a product system of Hilbert $\CP^e_{\Hilm[J],\Hilm}$\nb-bimodules.
\end{cor}
\begin{proof} By Proposition~\ref{prop:CP_grading_description}, $\CP^p_{\Hilm[J],\Hilm}$ is generated by elements of the form~$j_r(\xi)j_s(\eta)^*$, with~$\xi\in\Hilm_r$, $\eta\in\Hilm_s$ and~$rs^{-1}=p$. In particular, $r=ps$ and we can use the isomorphism~$\mu_{p,s}^{-1}$ to show that~$j_r(\xi)j_s(\eta)^*$ lies in~$j_p(\Hilm_p)j_s(\Hilm_s)j_s(\Hilm_s)^*$, which in turn is contained in~$j_p(\Hilm_p)\CP^e_{\Hilm[J],\Hilm}$. The inclusion $j_p(\Hilm_p)\CP^e_{\Hilm[J],\Hilm}\subseteq \CP^p_{\Hilm[J],\Hilm}$ is trivial. So~$\CP^p_{\Hilm[J],\Hilm}=j_p(\Hilm_p)\CP^e_{\Hilm[J],\Hilm}$. Hence~$\Hilm_p\otimes_A\CP^e_{\Hilm[J],\Hilm}\to\CP^p_{\Hilm[J],\Hilm}$, $\xi\otimes\eta\mapsto j_p(\xi)\eta$ gives an isomorphism of correspondences~$A\leadsto \CP^e_{\Hilm[J],\Hilm}$.

For each~$p\in P$, $\CP^p_{\Hilm[J],\Hilm}$ is a Hilbert $\CP^e_{\Hilm[J],\Hilm}$\nb-bimodule with the structure obtained from the multiplication and involution operations on~$\CP_{\Hilm[J],\Hilm}$. In particular, $\CP^e_{\Hilm[J],\Hilm}\CP^p_{\Hilm[J],\Hilm}=\CP^p_{\Hilm[J],\Hilm}$. Hence, if~$p,q\in P$, we have a correspondence isomorphism \begin{align*}\CP^p_{\Hilm[J],\Hilm}\otimes_{\CP^e_{\Hilm[J],\Hilm}}\CP^q_{\Hilm[J],\Hilm}&\cong\CP^p_{\Hilm[J],\Hilm}\CP^q_{\Hilm[J],\Hilm}= (j_p(\Hilm_p)\CP^e_{\Hilm[J],\Hilm})\CP^q_{\Hilm[J],\Hilm}=j_p(\Hilm_p)(\CP^e_{\Hilm[J],\Hilm}\CP^q_{\Hilm[J],\Hilm})\\&=j_p(\Hilm_p)\CP^q_{\Hilm[J],\Hilm}=j_p(\Hilm_p)j_q(\Hilm_q)\CP^e_{\Hilm[J],\Hilm}=j_{pq}(\Hilm_{pq})\CP^e_{\Hilm[J],\Hilm}=\CP^{pq}_{\Hilm,\Hilm[J]}.\end{align*} These multiplication maps are associative because they coincide with the multiplication on~$\CP_{\Hilm[J],\Hilm}$.
\end{proof}

 \section{Fell bundles over quasi-lattice ordered groups}~\label{sec:Fell_bundles_positive}

This section contains our main findings. We show that a simplifiable product system of Hilbert bimodules~$\Hilm=(\Hilm_p)_{p\in P}$ admits a unique extension~$\hat{\Hilm}=(\hat{\Hilm}_g)_{g\in G}$ to a semi-saturated and orthogonal Fell bundle over~$G$. A Fell bundle over~$G$ that is semi-saturated and orthogonal is then completely determined by its positive fibres. Albandik and Meyer obtained a similar correspondence between proper product systems over an Ore monoid and saturated Fell bundles over its enveloping group \cite[Proposition~3.17]{Albandik-Meyer:Product}.

\subsection{From product systems of Hilbert bimodules to Fell bundles} In order to define semi-saturatedness for a Fell bundle over~$G$, we recall that an element of~$G$ with an upper bound in the positive cone has a certain reduced form in terms of elements of~$P$.

\begin{lem} Let~$(G,P)$ be a quasi-lattice ordered group and let~$g\in G$ with~$g\vee e<\infty$. Then $g^{-1}\vee e<\infty$ and $g=(g\vee e)(g^{-1}\vee e)^{-1}$. 
\end{lem}
\begin{proof} Let $q\in P$ be such that $g^{-1}(g\vee e)=q$. Then $g\vee e=gq=(g^{-1})^{-1}q$. This shows that $g^{-1}\vee e<\infty$ and $g^{-1}\vee e\leq q$. But $g(g^{-1}\vee e)$ belongs to~$P$ and $g\leq g (g^{-1}\vee e)$. So $g\vee e\leq g (g^{-1}\vee e)$. Since the partial order $g_1\leq g_2\Leftrightarrow g_1^{-1}g_2\in P$ is invariant under left-translation by elements of~$G$, it follows that $q=g^{-1}(g\vee e)\leq g^{-1}\vee e$. So $q=g^{-1}\vee e$ and, therefore, $g=(g\vee e)(g^{-1}\vee e)^{-1}$.
\end{proof}

\begin{defn}\label{def:semi-saturatedness} Let~$(G,P)$ be a quasi-lattice ordered group and let~$(B_g)_{g\in G}$ be a Fell bundle over~$G$. We will say that~$(B_g)_{g\in G}$ is \emph{semi-saturated} with respect to the quasi-lattice ordered group structure of~$(G,P)$ if it satisfies the following conditions: 
\begin{enumerate}
\item[(S1)] $B_pB_q=B_{pq}$ for all $p,q\in P$;

\item[(S2)] $B_g=B_{(g\vee e)}B^*_{(g^{-1}\vee e)}$ for all~$g\in G$ with~$g\vee e<\infty$; 
\end{enumerate} 
\end{defn}

\begin{defn}\label{defn:orthogonal} A Fell bundle over~$G$ will be called \emph{orthogonal} with respect to~$(G,P)$ if~$B_g=\{0\}$ whenever~$g\vee e=\infty$.
\end{defn}

 Let~$\mathbb{F}$ be the free group on a set of generators~$S$. A Fell bundle over~$\mathbb{F}$ is semi-saturated in the sense of Exel if~$B_gB_h=B_{gh}$ for all~$g,h\in\mathbb{F}$ such that the multiplication~$g\cdot h$ involves no cancellation. It is called orthogonal if~$B_s^*B_t=\{0\}$ whenever~$s$ and~$t$ are distinct generators of~$\mathbb{F}$ (see~\cite{Exel:Partial_amenable_free} for further details). Let~$\mathbb{F}^+$ be the unital subsemigroup of~$\mathbb{F}$ generated by~$S$. Recall from \cite{Nica:Wiener--hopf_operators} that~$(\mathbb{F},\mathbb{F}^+)$ is a quasi-lattice ordered group and that an element~$g\in\FF$ satisfies~$g\vee e<\infty$ if and only if its reduced form is~$pq^{-1}$, with~$p,q$ in~$\FF^+$. In this case, $g\vee e=p$ and $g^{-1}\vee e=q$. The following result compares our definitions of semi-saturatedness and orthogonality for Fell bundles over~$\FF$ with those introduced by Exel.

\begin{prop} A Fell bundle~$(B_g)_{g\in \mathbb{F}}$ is semi-saturated and orthogonal with respect to~$(\mathbb{F},\mathbb{F}^+)$ if and only if it is both semi-saturated and orthogonal as defined in~\cite{Exel:Partial_amenable_free}. 
\end{prop}
\begin{proof} Suppose that~$(B_g)_{g\in \mathbb{F}}$ is semi-saturated and orthogonal with respect to~$(\mathbb{F},\mathbb{F}^+)$. Then orthogonality implies that~$(B_g)_{g\in\mathbb{F}}$ is orthogonal as defined by Exel, since $(p^{-1}q)\vee e=\infty$ if~$p$ and~$q$ are distinct generators of~$\FF$. In order to prove that~$(B_g)_{g\in\mathbb{F}}$ is also semi-saturated according to~\cite{Exel:Partial_amenable_free}, let~$g,h\in\FF$ be such that the product~$g\cdot h$ involves no cancellation. If~$gh\vee e=\infty$, then~$B_{gh}=\{0\}=B_gB_h$. Assume that~$(gh)\vee e<\infty$. First, this implies that either~$g$ belongs to~$\FF^+$ and~$h\vee e<\infty$ or~$g\vee e<\infty$ and~$h\in(\FF^+)^{-1}$ because~$gh$ has reduced form~$pq^{-1}$ with~$p,q\in P$ and the product~$g\cdot h$ involves no cancellation. In case~$g\in \FF^+$, we then have~$g(h\vee e)=gh\vee e$ and~$(gh)^{-1}\vee e=(h^{-1}g^{-1})\vee e=h^{-1}\vee e$. So axioms (S1) and (S2) give us $$B_gB_h=B_gB_{h\vee e}B_{h^{-1}\vee e}^*=B_{g(h\vee e)}B_{h^{-1}\vee e}^*=B_{(gh)\vee e}B_{(gh)^{-1}\vee e}^*=B_{gh}.$$ Now if~$h\in(\FF^+)^{-1}$, it follows from the previous case that $$B_gB_h=(B_{h^{-1}}B_{g^{-1}})^*=B_{h^{-1}g^{-1}}^*=B_{gh}.$$ This shows that~$(B_g)_{g\in\mathbb{F}}$  is semi-saturated as defined in~\cite{Exel:Partial_amenable_free}.

Now suppose that~$(B_g)_{g\in \mathbb{F}}$ is a Fell bundle that is semi-saturated and orthogonal according to~\cite{Exel:Partial_amenable_free}. Clearly, $(B_g)_{g\in \mathbb{F}}$ satisfies (S1). Any element of~$\FF$ has a reduced form, so that orthogonality as in Definition~\ref{defn:orthogonal} follows by combining semi-saturatedness and orthogonality of~$(B_g)_{g\in \mathbb{F}}$. Given~$g\in\FF$ with~$g\vee e<\infty$, the product~$(g\vee e) (g^{-1}\vee e)^{-1}$ involves no cancellation. Therefore, semi-saturatedness gives us 
$$B_g=B_{g\vee e}B_{(g^{-1}\vee e)^{-1}}=B_{g\vee e}B^*_{(g^{-1}\vee e)}.$$ This completes the proof of the statement.
\end{proof}

Our main examples of Fell bundles that are semi-saturated and orthogonal come from the grading of relative Cuntz--Pimsner algebras associated to compactly aligned product systems obtained in Proposition~\ref{prop:CP_grading_description}. In fact, we will prove that any Fell bundle that is semi-saturated and orthogonal is isomorphic to one of this form.

\begin{example}\label{ex:semisaturatedness_for_CP} Let~$\Hilm=(\Hilm_p)_{p\in P}$ be a compactly aligned product system and let~$\Hilm[J]=\{J_p\}_{p\in P}$ be a family of ideals in~$A$ with $J_p\subseteq\varphi_p^{-1}(\Comp(\Hilm_p))$ for all~$p\in P$. Then $(\CP^g_{\Hilm[J],\Hilm})_{g\in G}$ is orthogonal because~$\CP^g_{\Hilm[J],\Hilm}=\{0\}$ whenever~$g\vee e=\infty$. To see that it is also semi-saturated, observe that if~$p,q\in P$ satisfy~$pq^{-1}=g$, then there is~$r\in P$ with~$p=(g\vee e) r$ and $q=(g^{-1}\vee e)r$. Indeed, since $g\vee e$ and $g^{-1}\vee e$ are the least upper bounds for~$g$ and~$g^{-1}$ in~$P$, respectively, there are~$r,s\in P$ such that~$p=(g\vee e) r$ and~$q=(g^{-1}\vee e)s$. The equality~$g=(g\vee e)(g^{-1}\vee e)^{-1}=(g\vee e) rs^{-1}(g^{-1}\vee e)^{-1}$ entails~$r=s$.

Thus, given~$g$ in~$G$ with $g\vee e<\infty$, write $g=(g\vee e)(g^{-1}\vee e)^{-1}$. By Proposition~\ref{prop:CP_grading_description}, $\CP^g_{\Hilm[J],\Hilm}$ is spanned by the elements of the form~$j_p(\xi)j_q(\eta)^*$, with $\xi\in \Hilm_p$, $\eta\in \Hilm_q$ and~$pq^{-1}=g$. Given such an element~$j_p(\xi)j_q(\eta)^*$, let~$r\in P$ be such that~$p=(g\vee e)r$ and~$q=(g^{-1}\vee e)r$. We then employ the isomorphisms~$\mu_{g\vee e, r}^{-1}$ and~$\mu_{g^{-1}\vee e, r}^{-1}$ to conclude that $$j_p(\xi)j_q(\eta)^*\in j_{g\vee e}(\Hilm_{g\vee e})j_r(\Hilm_r)j_r(\Hilm_r)^*j_{g^{-1}\vee e}(\Hilm_{g^{-1}\vee e})^*\subseteq \CP^{g\vee e}_{\Hilm[J],\Hilm} (\CP^{g^{-1}\vee e}_{\Hilm[J],\Hilm})^*.$$ Therefore, $(\CP^g_{\Hilm[J],\Hilm})_{g\in G}$ satisfies (S2). Now axiom (S1) follows from Corollary~\ref{cor:positive_fibres_isomorphism}. Thus $(\CP^g_{\Hilm[J],\Hilm})_{g\in G}$ is also semi-saturated.
\end{example}

\begin{defn}\label{def:simplifiable_product_bimodules} A product system of \emph{Hilbert bimodules} $\Hilm=(\Hilm_p)_{p\in P}$ will be called \emph{simplifiable} if for all~$p,q\in P$ one has
\begin{enumerate}
\item[(i)]$\BRAKET{\Hilm_p}{\Hilm_p}\BRAKET{\Hilm_q}{\Hilm_q}\subseteq\BRAKET{\Hilm_{p\vee q}}{\Hilm_{p\vee q}}\quad\text { if } p\vee q<\infty$;
\item[(ii)]$\BRAKET{\Hilm_p}{\Hilm_p}\BRAKET{\Hilm_q}{\Hilm_q}=\{0\}$\quad if $p\vee q=\infty$;
\end{enumerate}here $\BRAKET{\cdot}{\cdot}$ denotes the left $A$\nb-valued inner product.
\end{defn}
\begin{rem} A simplifiable product system of Hilbert bimodules is compactly aligned. The converse is not true in general. For instance, take a nontrivial Hilbert bimodule~$\Hilm$ over a $\Cst$\nb-algebra~$A$ satisfying~$\Hilm\otimes_A\Hilm=\{0\}$. This produces a product system over~$\NN\times\NN$ such that $\Hilm_{(1,0)}=\Hilm_{(0,1)}=\Hilm$. It is compactly aligned because~$\Hilm_{(1,1)}=\{0\}$, but $\BRAKET{\Hilm_{(1,0)}}{\Hilm_{(1,0)}}=\BRAKET{\Hilm_{(0,1)}}{\Hilm_{(0,1)}}\neq\{0\}$.
\end{rem}  

A canonical example of a simplifiable product system of Hilbert bimodules comes from the underlying irreversible dynamical system of a semigroup $\Cst$\nb-algebra.
\begin{example}\label{ex:semigroup-product-system} Let $(G,P)$ be quasi-lattice ordered. Let $\chi_{qP} \in\ell^{\infty}(P)$ denote the characteristic function on the right ideal of~$P$ given by $\{qr\mid r\in P\}$. Let $A$ be the $\Cst$\nb-subalgebra generated by $\{\chi_{qP}\mid  q\in P\}$. There is an action $\beta=\{\beta_p\}_{p\in P}$ of~$P$ on~$A$ by injective endomorphisms with hereditary range. The endomorphism $\beta_p$ is defined on a generator by $\beta_p(\chi_{qP})\coloneqq\chi_{pqP}$. So $\beta_p(A)$ is the corner determined by the projection~$\chi_{pP}$. This gives a product system of Hilbert bimodules $A_{\beta}=(A_{\beta_p})_{p\in P}$ as follows. We set $A_{\beta_p}\coloneqq A\chi_{pP}$ and define the right $A$\nb-valued inner product by $$\braket{a\chi_{pP}}{b\chi_{pP}}\coloneqq \beta_{p}^{-1}(\chi_{pP}a^*\cdot b\chi_{pP}).$$ The right action of~$A$ on~$A_{\beta_p}$ is implemented by~$\beta_p$ and the homomorphism $\varphi_p\colon A\to \Bound(A_{\beta_p})$ is given by left multiplication. The multiplication map $\mu_{p,q}\colon A_{\beta_p}\otimes A_{\beta_q}\to A_{\beta_{pq}}$ sends an elementary tensor $a\chi_{pP}\otimes b\chi_{qP}$ to $a\beta_p(b)\chi_{pqP}$. See, for example, \cite[Section~4.3]{SEHNEM2019558}.

For each~$p\in P$, the correspondence $A_{\beta_p}$ carries a structure of Hilbert $A$\nb-bimodule, with left $A$\nb-valued inner product given by $\BRAKET{a\chi_{pP}}{b\chi_{pP}}=a\chi_{pP}b^*$. Thus $A_{\beta}=(A_{\beta_p})_{p\in P}$ is a product system of Hilbert bimodules that is simplifiable because the projections $\{\chi_{qP}\mid q\in P\}$ satisfy the relations $$\chi_{pP}\cdot \chi_{qP}=\begin{cases} \chi_{(p\vee q)P} &\text{if } p\vee q<\infty,\\
0 & \text{otherwise.}
\end{cases}$$ In this example, the relative Cuntz--Pimsner algebra for the family of Katsura's ideals of $A_{\beta}$ coincides with Nica's semigroup $\Cst$\nb-algebra $\Cst(G,P)$ \cite{Nica:Wiener--hopf_operators}. See Propositions 4.8 and 4.11 of \cite{SEHNEM2019558}.

\end{example}

\begin{prop}\label{prop:cuntz_nica} Let~$\Hilm=(\Hilm_p)_{p\in P}$ be a simplifiable product system of Hilbert bimodules. For each~$p\in P$, let $I_p\coloneqq\BRAKET{\Hilm_p}{\Hilm_p}$ and set $\mathcal{I}=\{I_p\}_{p\in P}$. If~$\psi=\{\psi_p\}_{p\in P}$ is a representation of~$\Hilm$ in a $\Cst$\nb-algebra~$B$ that is Cuntz--Pimsner covariant on~$\mathcal{I}$, then it is also Nica covariant.
\end{prop}
\begin{proof} Let~$p,q\in P$, $T\in \Comp(\Hilm_p)$ and $S\in\Comp(\Hilm_q)$. Let~$a\in I_p$ and $b\in I_q$ be such that $\varphi_p(a)=T$ and~$\varphi_q(b)=S$. Cuntz--Pimsner covariance on~$\Hilm[I]$ gives us $$\psi^{(p)}(T)\psi^{(q)}(S)=\psi_e(a)\psi_e(b)=\psi_e(ab).$$ So by condition (ii) of Definition~\ref{def:simplifiable_product_bimodules}, $\psi^{(p)}(T)\psi^{(q)}(S)=0$ if~$p\vee q=\infty$. In case~$p\vee q<\infty$, it follows that $\iota^{p\vee q}_p(T)\iota_q^{p\vee q}(S)=\varphi_{p\vee q}(ab).$ Applying the Cuntz--Pimsner covariance condition to~$ab\in I_{p\vee q}$, we obtain $$\psi^{(p)}(T)\psi^{(q)}(S)=\psi_e(ab)=\psi^{(p\vee q)}(\varphi_{p\vee q}(ab))=\psi^{(p\vee q)}\big(\iota^{p\vee q}_p(T)\iota_q^{p\vee q}(S)\big).$$ Therefore, $\psi$ is Nica covariant.
\end{proof}

\begin{lem}\label{pro:relative_embedding} Let~$\Hilm=(\Hilm_p)_{p\in P}$ be a simplifiable product system of Hilbert bimodules. For each~$p\in P$, set~$I_p\coloneqq\BRAKET{\Hilm_p}{\Hilm_p}$ and~$\mathcal{I}=\{I_p\}_{p\in P}$. Then the canonical \Star homomorphism from~$A$ to the relative Cuntz--Pimsner algebra~$\CP_{\mathcal{I},\Hilm}$  is an isomorphism onto the gauge-fixed point algebra~$\CP^e_{\mathcal{I},\Hilm}$. Moreover, $\CP^p_{\Hilm[I],\Hilm}\cong \Hilm_p$ for all~$p\in P$.
\end{lem}
\begin{proof} Conditions (i) and (ii) of Definition~\ref{def:simplifiable_product_bimodules} imply that the family of ideals~$\mathcal{I}=\{I_p\}_{p\in P}$ satisfies the hypothesis of~\cite[Proposition~4.8]{SEHNEM2019558}. So~$\CP_{\mathcal{I},\Hilm}$ is canonically isomorphic to the covariance algebra~$A\times_{\Hilm}P$ of~$\Hilm$. In particular, the universal \Star homomorphism $j_e\colon A\to \CP_{\mathcal{I},\Hilm}$ is injective by \cite[Theorem~3.10-(C3)]{SEHNEM2019558}.

Since~$\Hilm$ is simplifiable, its representation in~$\CP_{\Hilm[I],\Hilm}$ is Nica covariant by Proposition~\ref{prop:cuntz_nica}. Thus, $\CP^e_{\Hilm[I],\Hilm}$ is the closed linear span of the set $$\{j_p(\xi)j_p(\eta)^*\vert  \xi,\eta\in\Hilm_p, p\in P\}.$$ Hence the Cuntz--Pimsner covariance condition implies that~$j_e\colon A\rightarrow\CP^e_{\Hilm[I],\Hilm}$ is an isomorphism. 

It follows that~$j_p\colon\Hilm_p\rightarrow\CP^p_{\Hilm[I],\Hilm}$ is injective for all~$p\in P$. Again because the representation of~$\Hilm$ in~$\CP_{\Hilm[I],\Hilm}$ is Nica covariant, $\CP^p_{\Hilm[I],\Hilm}$ is generated by elements of the form~$j_q(\xi)j_r(\eta)^*$ with~$qr^{-1}=p$. Using that~$\mu_{p,r}$ is a correspondence isomorphism, we deduce from Cuntz--Pimsner covariance that~$j_p$ is also surjective, as asserted. 
\end{proof}

\begin{thm}\label{thm:uniqueness_and_extension} Let~$(G,P)$ be a quasi-lattice ordered group and~$\Hilm=(\Hilm_p)_{p\in P}$ a simplifiable product system of Hilbert bimodules. There is a semi-saturated and orthogonal Fell bundle~$\hat{\Hilm}=(\hat{\Hilm}_g)_{g\in G}$ extending the structure of product system of~$\Hilm$, in the sense that 
\begin{itemize}
\item[\textup{(i)}] there are isomorphisms~$j_p\colon\Hilm_p\cong\hat{\Hilm}_p$ of complex vector spaces such that $j_e\colon A\to\hat{\Hilm}_e$ is a \Star isomorphism and $j_p(\xi)j_q(\eta)=j_{pq}(\mu_{p,q}(\xi\otimes\eta))$ for all $p,q\in P$; 
\item[\textup{(ii)}] $j_p(\xi)^*j_p(\eta)=j_e(\braket{\xi}{\eta})$ for all $\xi,\eta\in\Hilm_p$ and~$p\in P$, where $^*\colon\hat{\Hilm}_{p}\to\hat{\Hilm}_{p^{-1}}$ is the involution operation on~$\hat{\Hilm}$.
\end{itemize}Moreover, $\hat{\Hilm}$ is unique up to canonical isomorphism of Fell bundles.
\end{thm}
\begin{proof} Let~$(\CP^g_{\Hilm[I],\Hilm})_{g\in G}$ be the Fell bundle associated to the canonical coaction of~$G$ on~$\CP_{\Hilm[I],\Hilm}$. This is semi-saturated and orthogonal by Example~\ref{ex:semisaturatedness_for_CP}. For each~$g\in G$, we set~$\hat{\Hilm}_g\coloneqq\CP^g_{\Hilm[I],\Hilm}$. So, for all~$p\in P$, $j_p\colon \Hilm_p\to \hat{\Hilm}_p$ is an isomorphism of complex vector spaces by Lemma~\ref{pro:relative_embedding} and we have~$j_p(\xi)^*j_p(\eta)=j_e(\braket{\xi}{\eta})$ for all $\xi,\eta\in\Hilm_p$. Hence~$\hat{\Hilm}=(\hat{\Hilm}_g)_{g\in G}$ is a semi-saturated and orthogonal Fell bundle over~$G$ that extends the structure of~$\Hilm=(\Hilm_p)_{p\in P}$.

In order to prove the uniqueness property, let~$\hat{\Hilm}'=(\hat{\Hilm}'_g)_{g\in G}$ be another Fell bundle that is semi-saturated and orthogonal and extends the structure of product system of~$\Hilm$. Let~$j'=\{j'_p\}_{p\in P}$ be the family of isomorphisms~$\Hilm_p\cong\hat{\Hilm}_p'$. Then~$j'=\{j'_p\}_{p\in P}$ is a representation of~$\Hilm$ in the cross-sectional $\Cst$\nb-algebra~$\Cst((\hat{\Hilm}'_g)_{g\in G})$. We will now see that~$j'=\{j'_p\}_{p\in P}$ is Cuntz--Pimsner covariant on~$\Hilm[I]=\{I_p\}_{p\in P}$. Indeed, the ideal in~$A$ determined by $j'_p(\Hilm_p)j'_p(\Hilm_p)^*$ is contained in $(\ker\varphi_p)^\perp$ because~$j'$ preserves multiplication and $j_e'$ is an isomorphism from~$A$ onto the fixed-point algebra of~$\hat{\Hilm}'$. And for all~$\xi,\eta$ and~$\zeta\in \Hilm_p,$ we have that $$j'_e(\BRAKET{\xi}{\eta})j'_p(\zeta)=j'_p(\xi\braket{\eta}{\zeta})=j'_p(\xi)j'_e(\braket{\eta}{\zeta})=j'_p(\xi)j_p'(\eta)^*j_p'(\zeta).$$ This implies that~$j_e'(\BRAKET{\xi}{\eta})=j'_p(\xi)j_p'(\eta)^*$.

We obtain a \Star homomorphism $\bar{j'}\colon \CP_{\Hilm[I],\Hilm}\to \Cst((\hat{\Hilm}'_g)_{g\in G})$ by universal property. Such a \Star homomorphism induces an isomorphism of Fell bundles~$\hat{\Hilm}\cong\hat{\Hilm}'$ because~$\hat{\Hilm}'$ is also semi-saturated and orthogonal and, for all~$p\in P$, $j'_p\colon\Hilm_p\to\hat{\Hilm}_p'$ is an isomorphism. This completes the proof of the theorem.    
\end{proof}

\begin{cor}\label{cor:cp-cross-section} Let $(G,P)$ be a quasi-lattice ordered group, with~$G$ amenable. Let $\Hilm=(\Hilm_p)_{p\in P}$ be a simplifiable product system of Hilbert bimodules and let $\hat{\Hilm}=(\hat{\Hilm}_g)_{g\in G}$ be the Fell bundle extending the structure of product system of~$\Hilm$ as in Theorem~\textup{\ref{thm:uniqueness_and_extension}}. Then the regular representation of~$\hat{\Hilm}$ induces an injective \Star homomorphism of $$\CP_{\Hilm[I],\Hilm}\to\Bound(\oplus_{\substack{g\in G}}\hat{\Hilm}_g).$$ The compression by the projection $Q_+\colon\oplus_{\substack{g\in G}}\hat{\Hilm}_g\to\oplus_{\substack{p\in P}}\Hilm_p$ gives a completely positive map $\CP_{\Hilm[I], \Hilm}\to \mathcal{N}\Toep_{\Hilm}$ which is a cross-section to the quotient map $q\colon \mathcal{N}\Toep_{\Hilm}\to \CP_{\Hilm[I], \Hilm}$.
\end{cor}
\begin{proof} By definition, $\hat{\Hilm}_g=\CP^g_{\Hilm[I],\Hilm}$ for all~$g\in G$ and hence $\Cst((\hat{\Hilm}_g)_{g\in G})$ is canonically isomorphic to $\CP_{\Hilm[I], \Hilm}$. From this we deduce that the regular representation of~$\hat{\Hilm}$ induces a \Star homomorphism $\CP_{\Hilm[I],\Hilm}\to\Bound(\bigoplus_{\substack{g\in G}}\hat{\Hilm}_g)$. This is injective because~$G$ is amenable (see~\cite[Theorem~20.7]{Exel:Partial_dynamical}). The image of $\CP_{\Hilm[I],\Hilm}$ under the compression $b\mapsto Q_+bQ_+$ is precisely the $\Cst$\nb-subalgebra of $\Bound(\Hilm^+)$ generated by the Fock representation of~$\Hilm$. This $\Cst$\nb-algebra is isomorphic to $\mathcal{N}\Toep_{\Hilm}$ because~$G$ is amenable (see \cite[Theorem~7.2]{Fowler:Product_systems} and \cite[Proposition~4.4]{SEHNEM2019558}). We then get a cross-section $\CP_{\Hilm[I],\Hilm}\to\mathcal{N}\Toep_{\Hilm}$ to the quotient map~$q$ as asserted.
\end{proof}

A product system of Hilbert bimodules~$\Hilm=(\Hilm_p)_{p\in P}$ gives rise to a product system~$\Hilm^*$ over~$P^{\mathrm{op}}$ by setting $\Hilm^*\coloneqq (\Hilm_p^*)_{p\in P},$ where $\Hilm_p^*$ is the Hilbert bimodule adjoint to~$\Hilm_p$. The multiplication map $\Hilm_p^*\otimes_A\Hilm_q^*\cong\Hilm^*_{qp}$ is given by the isomorphism $\Hilm^*_p\otimes_A\Hilm_q^*\cong (\Hilm_q\otimes_A\Hilm_p)^*$, $\xi^*\otimes\eta^*\mapsto(\eta\otimes\xi)^*$, followed by the multiplication map~$\mu_{q,p}$. We identify~$A$ with its adjoint Hilbert bimodule~$A^*$ through the isomorphism~$a\mapsto \widetilde{a^*}$ implemented by the involution operation on~$A$, where~$\widetilde{a^*}$ is the image of~$a^*$ in~$A^*$ under the canonical conjugate-linear map. We regard $\Hilm^*$ as a product system of correspondences over~$A$ using this identification. So $\Hilm^*_p$ has the canonical structure of Hilbert $A$\nb-bimodule adjoint of $\Hilm_p$ for each~$p\in P\setminus\{e\}$. Notice that $\Hilm^{**}=\Hilm$. Before providing concrete examples of relative Cuntz--Pimsner algebras for simplifiable product systems of Hilbert bimodules, we need the following lemma.

\begin{lem}\label{lem:adjoint_product_system}  Let~$\Hilm=(\Hilm_p)_{p\in P}$ be a product system of Hilbert bimodules. For each~$p\in P$, let~$I_{\Hilm_p}\coloneqq\BRAKET{\Hilm_p}{\Hilm_p}$ and set~$\Hilm[I]_{\Hilm}=\{I_{\Hilm_p}\}_{p\in P}$. A representation~$\psi=\{\psi_p\}_{p\in P}$ of~$\Hilm$ in a $\Cst$\nb-algebra~$B$ that is Cuntz--Pimsner covariant on~$\Hilm[I]_{\Hilm}$ naturally induces a representation of~$\Hilm^*=(\Hilm_p^*)_{p\in P}$ that is Cuntz--Pimsner covariant on~$\Hilm[I]_{\Hilm^*}$, where~$\Hilm[I]_{\Hilm^*}=\{I_{\Hilm_p^*}\}_{p\in P}$ and~$I_{\Hilm_p^*}=\BRAKET{\Hilm^*_p}{\Hilm_p^*}=\braket{\Hilm_p}{\Hilm_p}$. As a consequence, representations of~$\Hilm$ that are Cuntz--Pimsner covariant on~$\Hilm[I]_{\Hilm}$ are in one-to-one correspondence with representations of~$\Hilm^*$ that are Cuntz--Pimsner covariant on~$\Hilm[I]_{\Hilm^*}$.
\end{lem}
\begin{proof} For~$p=e$, put~$\psi_e^*\coloneqq\psi_e$. Given~$p\in P\setminus\{e\}$, define~$\psi^*_p\colon\Hilm_p^*\to B$ by $\psi^*_p(\xi^*)\coloneqq\psi_p(\xi)^*$ and set~$\psi^*=\{\psi^*_p\}_{p\in P}$. Then, for all~$\xi\in\Hilm_p$ and~$\eta\in\Hilm_q$, $$\psi^*_p(\xi^*)\psi^*_q(\eta^*)=\psi_p(\xi)^*\psi_q(\eta)^*=\psi_{qp}\left(\mu_{q,p}(\eta\otimes\xi)\right)^*=\psi_{qp}^*\left(\mu_{q,p}(\eta\otimes\xi)^*\right).$$ Since~$\psi$ is Cuntz--Pimsner covariant on~$\Hilm[I]_{\Hilm}$, it follows that $$\psi^*_p(\xi^*)^*\psi_p^*(\eta^*)=\psi_p(\xi)\psi_p(\eta)^*=\psi^{(p)}(\ket{\xi}\bra{\eta})=\psi_e(\BRAKET{\xi}{\eta})=\psi^*_e(\BRAKET{\xi}{\eta})$$ for all~$\xi,\eta\in\Hilm_p$. That~$\psi^*$ is Cuntz--Pimsner covariant on~$\Hilm[I]_{\Hilm^*}$ follows from the fact that~$\psi$ is a representation of~$\Hilm$. So the last statement is obtained from the identity~$\Hilm=\Hilm^{**}$.
\end{proof}

Let us present an important class of examples of product systems of Hilbert bimodules, with Proposition~\ref{prop:charact_extendible} in mind. This generalises Example~\ref{ex:semigroup-product-system}.

  \begin{example}[Crossed products by semigroups of endomorphisms]\label{ex:semigroup_of_endomorphisms} An endomorphism $\alpha\colon A\to A$ of a $\Cst$\nb-algebra~$A$ is said to be~\emph{extendible} if it extends to a strictly continuous endomorphism of the multiplier algebra~$M(A)$ (see~\cite{Adji-S}). This happens if and only if there is a projection~$Q\in M(A)$ so that~$\alpha(u_{\lambda})$ converges to~$Q$ in the strict topology of~$M(A)$, where~$(u_{\lambda})_{\lambda\in\Lambda}$ is an approximate unit for~$A$. In particular, we have $Q=\alpha(1)$, where here we still denote by~$\alpha$ the induced endomorphism of~$M(A)$. Let $\alpha\colon P\to \mathrm{End}(A)$ be an action by extendible endomorphisms with~$\alpha_e=\id_A$.\footnote{If $\alpha_e$ is injective, the equality $\alpha_e=\alpha_e\circ\alpha_e$ entails $\alpha_e=\id_A$.} For each~$p\in P$, let~$\tensor*[_{\alpha_p}]{A}{}\coloneqq\alpha_p(1)A$ be equipped with the structure of right Hilbert $A$\nb-module coming from the multiplication and involution operations on~$A$. That is, $\alpha_p(1)a\cdot b\coloneqq\alpha_p(1)ab$ and $\braket{\alpha_p(1)a}{\alpha_p(1)b}\coloneqq a^*\alpha_p(1)b$ for all~$a,b\in A$. We let~$\varphi_p\colon A\to\Bound(\tensor*[_{\alpha_p}]{A}{})$ be the \Star homomorphism implemented by~$\alpha_p$. So $\varphi_p(b)(\alpha_p(1)a)=\alpha_p(b)a$. This turns~$\tensor*[_{\alpha_p}]{A}{}$ into a correspondence over~$A$.
  
  We let~$\mu_{p,q}\colon\tensor*[_{\alpha_p}]{A}{}\otimes_A\tensor*[_{\alpha_q}]{A}{}\to\tensor*[_{\alpha_{qp}}]{A}{}$ be defined on elementary tensors by $$\alpha_p(1)a\otimes_A\alpha_q(1)b\mapsto \alpha_{qp}(1)\alpha_q(a)b.$$ This intertwines the left and right actions of~$A$ and preserves the $A$\nb-valued inner product. It is surjective because $$\alpha_{qp}(1)a=\lim_{\substack{\lambda}}\alpha_{qp}(u_{\lambda})a=\lim_{\substack{\lambda}}\alpha_q(\alpha_p(u_{\lambda}))a=\lim_{\substack{\lambda}}\alpha_q(\alpha_p(1)\alpha_p(u_{\lambda}))a.$$ Since~$p\mapsto \alpha_p$ is an action by endomorphisms, the multiplication maps are associative. Thus, $\alpha\colon P\to \text{End}(A)$ gives rise to a product system $\tensor*[_{\alpha}]{A}{}=(\tensor*[_{\alpha_p}]{A}{})_{p\in P}$ over~$P^{\mathrm{op}}$, where~$P^{\mathrm{op}}$ is the opposite semigroup of~$P$. Moreover, $\tensor*[_{\alpha}]{A}{}$ is proper, since~$\Comp(\tensor*[_{\alpha_p}]{A}{})\cong \alpha_p(1)A\alpha_p(1)$ and $\alpha_p(a)=\alpha_p(1)\alpha_p(a)\alpha_p(1)$ for all~$a\in A$ and~$p\in P$.

Now suppose further that~$\alpha\colon P\to \mathrm{End}(A)$ is an action by extendible endomorphisms as above with the property that, for all~$p\in P$, $\alpha_p$ is an injective endomorphism with hereditary range. In this case, $\tensor*[_{\alpha}]{A}{}$ is faithful and may be enriched to a product system of Hilbert bimodules over~$P^{\mathrm{op}}$. The left $A$\nb-valued inner product is given by $$\BRAKET{\alpha_p(1)a}{\alpha_p(1)b}=\alpha_p^{-1}(\alpha_p(1)ab^*\alpha_p(1))$$ for all~$a,b\in A$ and~$p\in P$. In particular, this yields a product system $A_{\alpha}=(A_{\alpha_p})_{p\in P}$ over~$P$, where~$A_{\alpha_p}=A\alpha_p(1)$ with the Hilbert $A$\nb-bimodule structure obtained from~$\tensor*[_{\alpha_p}]{A}{}^*$ through the identification $\widetilde{\alpha_p(1)a}\mapsto a^*\alpha_p(1)$. Observe that  $\tensor*[_{\alpha}]{A}{}$ is always compactly aligned. It is simplifiable if and only if~$P$ is right reversible. It is not clear when~$A_{\alpha}$ is simplifiable (or even compactly aligned).

Let us now relate constructions of semigroup crossed products with relative Cuntz--Pimsner algebras: the ideal $I_p\idealin A$ given by the left inner product of~$A_{\alpha_p}$ is precisely $\overline{A\alpha_p(1)A}$. Given a nondegenerate representation $\psi=\{\psi_p\}_{p\in P}$ of~$A_{\alpha}$ in a $\Cst$\nb-algebra~$B$, we obtain a strictly continuous unital \Star homomorphism~$\bar{\psi}_e\colon M(A)\to M(B)$ by nondegeneracy of~$\psi_e$. In addition, we define a semigroup homomorphism from~$P$ to the semigroup of isometries in~$M(B)$ by setting~$$v_p\coloneqq\lim_{\substack{\lambda}}\psi_p(u_{\lambda}\alpha_p(1)).$$ Here the limit is taken in the strict topology of~$M(B)$. It indeed exists because~$\|\psi_p(u_{\lambda}\alpha_p(1))\|\leq 1$ for each~$\lambda$ and, for~$a\in A$ and~$b\in B$, $$\lim_{\substack{\lambda}}\psi_p(u_{\lambda}\alpha_p(1))(\psi_e(a)b)=\lim_{\substack{\lambda}}\psi_p(u_{\lambda}\alpha_p(a)) b=\psi_p(\alpha_p(a))b$$ and $$\lim_{\substack{\lambda}}(b\psi_e(a))\psi_p(u_{\lambda}\alpha_p(1))=\lim_{\substack{\lambda}}b\psi_p(au_{\lambda}\alpha_p(1))=b\psi_p(a\alpha_p(1)).$$ To see that $v_p^*v_p=1$, observe that \begin{equation*}\begin{aligned}v_p^*v_p(\psi_e(a)b)&=\lim_{\substack{\lambda}}\psi_p(u_{\lambda}\alpha_p(1))^*\psi_p(\alpha_p(a))b\\&=\lim_{\substack{\lambda}}\psi_e\big(\alpha_p^{-1}(\alpha_p(1)u_{\lambda}\alpha_p(a))\big)b=\psi_e(a)b.\end{aligned}\end{equation*}

 The semigroup of isometries~$\left\{v_p\middle|\, p\in P\right\}$ together with the \Star homomorphism $\bar{\psi}_e\colon M(A)\to M(B)$ satisfy the relation 
$$v_p\cdot \bar{\psi}_e(c)=\bar{\psi}_e(\alpha_p(c))v_p$$ for all~$c\in M(A)$ and~$p\in P$. Hence \begin{equation}\label{eq:covariant_pair} \bar{\psi}_e(\alpha_p(c))v_pv_p^*=v_p \bar{\psi}_e(c)v_p^*.\end{equation} In addition, $\psi_p(a\alpha_p(1))=\psi_e(a)v_p$ for all~$a\in A$ and~$p\in P$. If~$\psi$ is Cuntz--Pimsner covariant on
$\Hilm[I]_{A_{\alpha}}=\{I_p\}_{p\in P},$ it follows that for all $c\in M(A)$ and $p\in P$, $$\bar{\psi}_e(\alpha_p(c))=\bar{\psi}_e(\alpha_p(c))v_pv_p^*.$$ Indeed, for~$c$ in~$A$ and~$a\alpha_p(1)$ in~$A_{\alpha_p}$, we compute \begin{align*}\alpha_p(c^*c)a\alpha_p(1)&=\alpha_p(c^*)\alpha_p(\alpha_{p^{-1}}(\alpha_p(c)a\alpha_p(1)))\\&=\alpha_p(c^*)\alpha_p\left(\braket{\alpha_p(c^*)\alpha_p(1)}{a\alpha_p(1)}\right)\\&=\alpha_p(c^*)\cdot \braket{\alpha_p(c^*)\alpha_p(1)}{a\alpha_p(1)}\\&=\ket{\alpha_p(c^*)}\bra{\alpha_p(c^*)}(a\alpha_p(1)).\end{align*} Hence Cuntz--Pimsner covariance gives us \begin{align*}\psi_e(\alpha_p(c^*c))&=\psi_p(\alpha_p(c^*))\psi_p(\alpha_p(c^*))^*=\psi_e(\alpha_p(c^*))v_pv_p^*\psi_e(\alpha_p(c))\\&=\psi_e(\alpha_p(c^*))v_p\psi_e(c)v_p^*=\psi_e(\alpha_p(c^*c))v_pv_p^*.\end{align*} Since~$A$ is spanned by positive elements, the same relation holds for all~$c\in A$ and thus for all~$c\in M(A)$ if we replace~$\psi_e$ by its extension~$\bar{\psi}_e$. So combining this with~\eqref{eq:covariant_pair}, we deduce the relation $$\psi_e(\alpha_p(c))=v_p\psi_e(c)v_p^*$$ for all~$c\in A$. The same holds for~$\bar{\psi}_e$ and~$c$ in $M(A)$. 

Conversely, we claim that a nondegenerate \Star homomorphism $\pi\colon A\to B$ together with a semigroup of isometries~$\left\{v_p\middle| p\in P\right\}$ satisfying the relation \begin{equation}\label{eq:covariance_cp}\pi(\alpha_p(a))=v_p\pi(a)v_p^*\end{equation} yields a representation of~$A_{\alpha}$ that is Cuntz--Pimsner covariant on~$\Hilm[I]_{A_{\alpha}}$. First, notice that the projection~$v_pv_p^*$ coincides with~$\bar{\pi}(\alpha_p(1))$, where~$\bar{\pi}$ is the strictly continuous \Star homomorphism $M(A)\to M(B)$ extending~$\pi$. For each~$p\in P$ and $a\in A$, we set~$\psi_p(a \alpha_p(1))\coloneqq \pi(a)v_p$. Put $\psi=\{\psi_p\}_{p\in P}$. Then $\bar{\pi}(\alpha_p(1))=v_pv_p^*$ implies that~$\psi$ is Cuntz--Pimsner covariant on~$I_p=\overline{A\alpha_p(1)A}$ for all~$p\in P$, since \begin{align*}\psi_e\left(a\alpha_p(1)b\right)&=\pi\left(a\alpha_p(1)b\right)=\pi(a)\bar{\pi}(\alpha_p(1))\pi(b)\\&=\pi(a)v_pv_p^*\pi(b)=\pi(a)v_p(\pi(b^*)v_p)^*\\&=\psi_p(a\alpha_p(1))\psi_p(b^*\alpha_p(1))^*=\psi^{(p)}\left(\ket{a\alpha_p(1)}\bra{b^*\alpha_p(1)}\right)\end{align*} for all~$a$ and~$b$ in~$A$. Moreover, \eqref{eq:covariance_cp} tells us that $\psi_e(\alpha_p(a))v_p=v_p\psi_e(a)$ for all~$a\in A$ and~$p\in P$. This also gives $$\psi_p(a\alpha_p(1))\psi_q(b\alpha_q(1))=\psi_{pq}(a\alpha_p(b)\alpha_{pq}(1))=\psi_{pq}\big(\mu_{p,q}(a\alpha_p(1)\otimes b\alpha_q(1))\big).$$ Again by \eqref{eq:covariance_cp},\begin{align*}\psi_e\left(\alpha_{p^{-1}}(\alpha_p(1)a^*b\alpha_p(1))\right)&=v_p^*v_p\psi_e\left(\alpha_{p^{-1}}(\alpha_p(1)a^*b\alpha_p(1))\right)v_p^*v_p\\&=v_p^*\psi_e\left(\alpha_p(1)a^*b\alpha_p(1)\right)v_p\\&=v_p^*\psi_e(a^*b)v_p.\end{align*} This shows that~$\psi$ is a representation of~$A_{\alpha}$ that is Cuntz--Pimsner covariant on~$\Hilm[I]_{A_{\alpha}}$. 

As a result, the crossed product $A\rtimes_{\alpha}P$ of~$A$ by the semigroup of endomorphisms provided by~$\alpha$ has a description as the universal $\Cst$\nb-algebra of representations of~$A_{\alpha}$ that are Cuntz--Pimsner covariant on~$\Hilm[I]_{A_{\alpha}}$. By Lemma~\ref{lem:adjoint_product_system}, $A\rtimes_{\alpha}P$ may also be described as the universal $\Cst$\nb-algebra for representations of~$\tensor*[_{\alpha}]{A}{}$ that are Cuntz--Pimsner covariant on $\Hilm[I]_{\tensor*[_{\alpha}]{A}{}}$. If~$P^{\mathrm{op}}$ is the positive cone of a quasi-lattice order and~$P$ is right reversible (so~$P^{\mathrm{op}}$ is directed), a representation of $\tensor*[_{\alpha}]{A}{}$ that is Cuntz--Pimsner covariant on $\Hilm[I]_{\tensor*[_{\alpha}]{A}{}}$ is also Nica covariant by \cite[Proposition 5.4]{Fowler:Product_systems}. In this case $\CP_{\Hilm[I]_{\tensor*[_{\alpha}]{A}{} },\tensor*[_{\alpha}]{A}{}}\cong A\rtimes_{\alpha}P.$ In general, $A\rtimes_{\alpha}P$ is the Cuntz--Pimsner algebra of $\tensor*[_{\alpha}]{A}{}$ as defined by Fowler~\cite[Proposition~3.4]{Fowler:Product_systems}. See, for instance, \cite{LACA1996415} and \cite{MR1955617} for constructions of crossed products by semigroups of endomorphisms. We also refer the reader to~\cite{Larsen:Crossed_abelian} for this and further constructions of crossed products out of product systems.
\end{example}

\subsection{From Fell bundles to product systems}

Let~$(B_g)_{g\in G}$ be a semi-saturated Fell bundle with respect to~$(G,P)$. There is a canonical product system associated to~$(B_g)_{g\in G}$: for each~$p\in P$, view~$B_p$ as a Hilbert $B_e$\nb-bimodule with left and right actions inherited from the multiplication in~$(B_g)_{g\in G}$. The left inner product is given by~$\BRAKET{\xi}{\eta}\coloneqq\xi\eta^*$, while the right inner product is~$\braket{\xi}{\eta}\coloneqq\xi^*\eta$. The property (S1) of Definition~\ref{def:semi-saturatedness} says that~$\Hilm[B]=(B_p)_{p\in P}$ is a product system with isomorphisms~$B_p\otimes_{B_e} B_q\cong B_{pq}$ coming from the multiplication in~$(B_g)_{g\in G}$. If~$(B_g)_{g\in G}$ is also orthogonal, the next result states that the cross-sectional $\Cst$\nb-algebra of~$(B_g)_{g\in G}$ can be recovered from~$\Hilm[B].$

\begin{prop}\label{prop:associated_product_system} Let $(B_g)_{g\in G}$ be a Fell bundle that is semi-saturated and orthogonal with respect to~$(G,P)$. Then~$\Hilm[B]=(B_p)_{p\in P}$ is a simplifiable product system of Hilbert bimodules. Its relative Cuntz--Pimsner algebra~$\CP_{\Hilm[I],\Hilm[B]}$ is naturally isomorphic to the cross-sectional $\Cst$\nb-algebra of~$(B_g)_{g\in G}$. Such an isomorphism comes from the canonical isomorphism between the Fell bundles $(\CP^g_{\Hilm[I],\Hilm[B]})_{g\in G}$ and $(B_g)_{g\in G}$.
\end{prop}
\begin{proof} Let $p,q\in P$ and set $g=p^{-1}q$. Notice that~$p\vee q=\infty$ if and only if~$g\vee e=\infty$ and hence $\BRAKET{B_p}{B_p}\BRAKET{B_q}{B_q}=B_pB_p^*B_qB_q^*=\{0\}$ provided $p\vee q=\infty$. Suppose that $p\vee q<\infty$. Then~$g\vee e=p^{-1}(p\vee q)$ and $g^{-1}\vee e=q^{-1}(p\vee q)$ so that $$
\begin{aligned}\BRAKET{B_p}{B_p}\BRAKET{B_q}{B_q}&\subseteq B_pB_{p^{-1}q}B_q^*\\&=B_pB_{p^{-1}(p\vee q)}B^*_{q^{-1}(p\vee q)}B^*_q\\&=B_{p\vee q}B^*_{p\vee q}=\BRAKET{B_{p\vee q}}{B_{p\vee q}}.
\end{aligned}
$$

The representation of~$(B_g)_{g\in G}$ in $\Cst((B_g)_{g\in G})$ restricted to the fibres over~$P$ is Cuntz--Pimsner covariant on~$\Hilm[I]$. This gives us a \Star homomorphism $\psi\colon\CP_{\Hilm[I],\Hilm[B]}\rightarrow\Cst((B_g)_{g\in G})$ that is an isomorphism between the Fell bundles $(\CP^g_{\Hilm[I],\Hilm[B]})_{g\in G}$ and $(B_g)_{g\in G}$ by the uniqueness property established in Theorem~\ref{thm:uniqueness_and_extension}. Such an isomorphism yields a representation of~$(B_g)_{g\in G}$ in~$\CP_{\Hilm[I],\Hilm[B]}$ and hence $\CP_{\Hilm[I],\Hilm[B]}$ and $\Cst((B_g)_{g\in G})$ are canonically isomorphic to each other. 
\end{proof}

Combining Example~\ref{ex:semisaturatedness_for_CP} with the previous proposition, we obtain the following:

\begin{cor}\label{cor:CP_simplifiable_psystem} Let~$\Hilm=(\Hilm_p)_{p\in P}$ be a compactly aligned product system and $\CP_{\Hilm[J],\Hilm}$ a relative Cuntz--Pimsner algebra associated to~$\Hilm$. Then $(\CP^p_{\Hilm[J],\Hilm})_{p\in P}$ is simplifiable.
\end{cor}

\begin{defn} Let $(G,P)$ be a quasi-lattice order. A Fell bundle over~$G$ is said to be \emph{extended from P} if it is semi-saturated and orthogonal with respect to the quasi-lattice ordered group structure of~$(G,P)$.
\end{defn}

\subsection{Amenability for Fell bundles extended from free semigroups}\label{sec:amenability_for_Fellbundles}

A quasi-lattice ordered group~$(G,P)$ is called~\emph{amenable} if the Fock representation of~$P$ in~$\Bound(\ell_2(P))$ induces an injective \Star homomorphism $\psi^+\colon\Cst(G,P)\to \Bound(\ell_2(P))$ (see~\cite[Section 4.2]{Nica:Wiener--hopf_operators}). Examples of amenable quasi-lattice orders are free groups~\cite{Exel:Partial_amenable_free,Nica:Wiener--hopf_operators}, Baumslag--Solitar groups~$\mathrm{BS}(c,d)$ with~$c,d$ positive integers~\cite{MR3595491} and, of course, $(G,P)$ for an amenable group~$G$. Counterexamples are, for instance, nonabelian Artin groups of finite type~\cite{crisp_laca_2002}. In~\cite{Exel:Partial_amenable_free}, Exel proved that Fell bundles extended from a free semigroup~$\FF^+$ are amenable, under a separability hypothesis. In this section, we follow the ideas of~\cite{MR3595491} to show that any Fell bundle extended from~$\FF^+$ is amenable, with no extra assumptions. But here we deduce faithfulness of the regular representation from gauge-invariant uniqueness theorems for relative Cuntz--Pimsner algebras of single correspondences.  The same techniques are employed to show that a Fell bundle extended from~$\mathrm{BS}(c,d)^+$ is always amenable. This suggests that amenability for Fell bundles extended from a positive cone is connected with amenability of the underlying quasi-lattice ordered group. We have been unable to establish a general result in this direction.

\begin{lem}\label{lem:cp_picture} Let~$(B_g)_{g\in\FF}$ be a Fell bundle extended from~$\FF^+$. View~$\Hilm[G]\coloneqq {\bigoplus}_{a\in S}B_{a}$ as a correspondence over~$B_e$. Then~$\Cst\left((B_g)_{g\in \FF}\right)\cong\CP_{I_{\Hilm[G]},\Hilm[G]}$. 
\begin{proof} Let~$\theta\colon\FF\to \ZZ$ be the group homomorphism defined on the generators by~$a\mapsto 1$, for all~$a\in S$. So for~$b\in\FF^+$, $\theta(b)=|b|$ is the length of~$b$ in its reduced form. This induces a coaction of~$\ZZ$ on~$(B_g)_{g\in G}$ by~\cite[Example A.28]{Echterhoff-Kaliszewski-Quigg-Raeburn:Categorical}. Hence it provides~$\Cst((B_g)_{g\in\FF})$ with a topological $\ZZ$\nb-grading, for which the corresponding spectral subspace at~$m\in \ZZ$ is the closure of $$\mathrm{span}\{\xi_p\cdot\eta_q^*\vert\,p,q\in\FF^+\,\text{ and }\theta(p)-\theta(q)=m\}.$$

Now let~$\Hilm[G]$ be the direct sum ${\bigoplus}_{a\in S}B_{a}$ viewed as a correspondence over~$B_e$ in the usual way. Let~$I_{\Hilm[G]}$ be Katsura's ideal for~$\Hilm[G]$. That is, $$I_{\Hilm[G]}=\varphi^{-1}_{\Hilm[G]}(\Comp(\Hilm[G]))\cap(\ker\varphi_{\Hilm[G]})^\perp=\bigoplus_{\substack{a\in S}}B_{a}B_{a}^*.$$ This sum is indeed orthogonal because $B_{a}B_{a}^*B_{b}B^*_{b}=\{0\}$ for~$a\neq b$. It follows that $$\left(\bigoplus_{\substack{a\in S}}\xi_{a}\right)^*\left(\bigoplus_{\substack{a\in S}}\eta_{a}\right)=\bigoplus_{\substack{a\in S}}\xi_{a}^*\eta_{a}$$ in~$\Cst((B_g)_{g\in \FF})$, where~$\xi_{a},\eta_{a}\in B_{a}$ for all~$a\in S$. Thus we get a representation of~$\Hilm[G]$ in~$\Cst((B_g)_{g\in\FF})$ obtained by restricting the representation of~$(B_g)_{g\in G}$ to the~$B_{a}$'s. This is a gauge-compatible injective representation of~$\Hilm[G]$ that is covariant on~$I_{\Hilm[G]}$, and whose image generates~$\Cst((B_g)_{g\in \FF})$ as a $\Cst$\nb-algebra. Hence it induces an isomorphism~$\CP_{I_{\Hilm[G]},\Hilm[G]}\to\Cst((B_g)_{g\in \FF})$ by~\cite[Theorem 6.4]{Katsura:Cstar_correspondences}.
\end{proof}
\end{lem}

The next corollary is a consequence of the previous lemma and nuclearity and exactness results for Cuntz--Pimsner algebras.

\begin{cor} Let~$(B_g)_{g\in\FF}$ be a Fell bundle extended from~$\FF^+$. If~$B_e$ is nuclear (resp. exact), then~$\Cst\left((B_g)_{g\in \FF}\right)$ is nuclear (resp. exact).
\begin{proof} The nuclearity (resp. exactness) of~$\Cst((B_g)_{g\in\FF})$ follows from the fact that the Cuntz--Pimsner algebra of a correspondence is nuclear (resp. exact) whenever the coefficient algebra is nuclear (resp. exact) (see~\cite{Dykema-shlyakhtenko, Katsura:Cstar_correspondences}).
\end{proof}
\end{cor}

Now Lemma~\ref{lem:cp_picture} combined with gauge-invariant uniqueness theorems for Cuntz--Pimsner algebras implies the following:

\begin{prop}\label{prop:amenability_free_groups} A Fell bundle~$(B_g)_{g\in \FF}$ extended from~$\FF^+$ is amenable, where~$\FF$ denotes the free group on a set of generators~$S$.
\begin{proof} We begin by proving that, as  $\Cst((B_g)_{g\in \FF})$, the reduced cross-sectional $\Cst$\nb-algebra $\Cst_r((B_g)_{g\in \FF})$ also carries a topological $\ZZ$\nb-grading, for which the regular representation $\Lambda\colon\Cst((B_g)_{g\in \FF})\to\Cst_r((B_g)_{g\in \FF})$ is a grading-preserving \Star homomor\-phism. Indeed, for each~$z\in \TT$, define a unitary $U_z\in\Bound(\ell_2((B_g)_{g\in \FF}))$ by setting $$\eta^+=\bigoplus_{\substack{g\in \FF}}\eta_g\mapsto U_z(\eta^+)=\bigoplus_{\substack{g\in \FF}} z^{\theta(g)}\eta_g.$$ Then $\Lambda(b)\mapsto U_{z}\Lambda(b)U_z^*$ is a continuous action of~$\TT$ on the reduced cross-sectional $\Cst$\nb-algebra of~$(B_g)_{g\in\FF}$. Hence~$\Cst_r((B_g)_{g\in \FF})$ is a topologically $\ZZ$\nb-graded $\Cst$\nb-algebra (see~\cite{Exel:Circle_actions}). 

Thus the composition of the regular representation~$\Lambda$ with the isomorphism $\CP_{I_{\Hilm[G]},\Hilm[G]}\cong \Cst((B_g)_{g\in \FF})$ from Lemma~\ref{lem:cp_picture} gives a gauge-compatible injective representation of~$\Hilm[G]$ that is covariant on~$I_{\Hilm[G]}$. So we invoke again the gauge-invariant uniqueness theorem for Katsura's relative Cuntz--Pimsner algebra of a single correspondence, namely~\cite[Theorem 6.4]{Katsura:Cstar_correspondences}, to derive faithfulness of~$\Lambda$. This shows that~$(B_g)_{g\in \FF}$  is amenable. 
\end{proof}
\end{prop}

Let~$c$ and~$d$ be positive integers. Recall from \cite{Spielberg:Baumslag-Solitar} that the Baumslag--Solitar group~$\mathrm{BS}(c,d)$ is the universal group on two generators~$a$ and~$b$ subject to the relation~$ab^c=b^da$ and~$(\mathrm{BS}(c,d),\mathrm{BS}(c,d)^+)$ is a quasi-lattice ordered group, where~$\mathrm{BS}(c,d)^+$ is the unital subsemigroup generated by~$a$ and~$b$. As for free groups, there is a group homomorphism~$\theta\colon\mathrm{BS}(c,d)\to\ZZ$ which is given on generators by~$a\mapsto 1$ and~$b\mapsto 0$. We follow~\cite{MR3595491} and~\cite{Spielberg:Baumslag-Solitar} and call~$\theta(g)$ for~$g\in\mathrm{BS}(c,d)$ the~\emph{height} of~$g$. 

Each~$p\in \mathrm{BS}(c,d)^+$ has a reduced form $$p=b^{s_0}ab^{s_1}\ldots b^{s_{k-1}}ab^{s_k},$$ with $0\leq s_i< d$ for all $i\in\{1,\ldots,k-1\}$ and~$\theta(p)=k$. As in~\cite{MR3595491}, we set $$\mathrm{stem}(p)\coloneqq b^{s_0}ab^{s_1}\ldots b^{s_{k-1}}a.$$  Given a Fell bundle extended from~$\mathrm{BS}(c,d)^+$, we will again construct a correspondence~$\Hilm[G]$ over a $\Cst$\nb-algebra~$B$ so that~$\CP_{\Hilm[I]_{\Hilm[G]},\Hilm[G]}$ is $\TT$\nb-equivariantly isomorphic to~$\Cst((C_g)_{g\in{BS}(c,d)})$. 

We need the following lemma:
\begin{lem}[\cite{MR3595491}*{Lemma 3.4}] Let $p,q\in  \mathrm{BS}(c,d)^+$~be such that $p\vee q<\infty$. Then,
\begin{enumerate}
\item[\rm (i)] if $\theta(p)>\theta(q)$, there is $m\in\NN$ with $p\vee q=pb^m$;
\item[\rm (ii)] if $\theta(p)=\theta(q)$, there is $m\in\NN$ with either~$$p\vee q=pb^m=q,\quad\text{ or }\quad p\vee q=qb^m=p.$$
\end{enumerate}
\end{lem}

In particular, by the previous lemma, $p\vee q=\infty$ and hence $C_p^*C_q=\{0\}$ whenever~$p$ and~$q$ have reduced forms~$b^{s_0}a$ and~$b^{t_0}a$ with~$s_0\neq t_0$.

\begin{lem}\label{lem:pic_forBaumslag-Sol} Let~$(C_g)_{g\in \mathrm{BS}(c,d)}$ be a Fell bundle extended from~$\mathrm{BS}(c,d)^+$. Let~$B$ be the $\Cst$\nb-subalgebra of~$\Cst((C_g)_{g\in \mathrm{BS}(c,d)})$ generated by the fibre $C_b$ and the unit fibre~$C_e$. For each~$0\leq i< d$, let $\Hilm[G]_i=C_{b^{i}a}\otimes_{C_e}B$ and set $$\Hilm[G]\coloneqq\bigoplus^{d-1}_{\substack{i=0}}\Hilm[G]_i.$$ Then the multiplication on $(C_g)_{g\in \mathrm{BS}(c,d)}$ provides~$\Hilm[G]$ with a structure of correspondence over~$B$. Moreover, $\Cst((B_g)_{g\in \mathrm{BS}(c,d)})\cong  \CP_{I_{\Hilm[G]},\Hilm[G]}$.
\begin{proof} Observe that the $\Cst$\nb-subalgebra~$B$ of~$\Cst((C_g)_{g\in \mathrm{BS}(c,d)})$ generated by the fibre~$C_b$ and the unit fibre~$C_e$ is a topologically $\ZZ$\nb-graded $\Cst$\nb-algebra, where the conditional expectation onto~$C_e$ coincides with that of~$\Cst((C_g)_{g\in \mathrm{BS}(c,d)})$. The corresponding spectral subspace at~$m\in \ZZ$ is~$C_{b^{m}}$. In particular, $B$ is (isomorphic to) the cross-sectional $\Cst$\nb-algebra of the Fell bundle~$(C_{b^m})_{m\in\ZZ}$. 

Let $\Hilm[G]_i=C_{b^{i}a}\otimes_{C_e}B$, for~$0\leq i< d$, and set $$\Hilm[G]=\bigoplus^{d-1}_{\substack{i=0}}\Hilm[G]_i.$$ This is a correspondence~$C_e\leadsto B$. We extend the left action of~$C_e$ to~$B$ by using the multiplication on~$(C_g)_{g\in \mathrm{BS}(c,d)}$. By the above discussion, it suffices to find a representation of the Hilbert $C_e$\nb-bimodule~$C_b$ in~$\Bound(\Hilm[G])$ that is Cuntz--Pimsner covariant on~$C_bC_b^*$  by \cite[Proposition~2.19]{Meyer-Sehnem:Bicategorical_Pimsner} (see also~\cite{Abadie-Eilers-Exel:Morita_bimodules}). Thus for~$\xi\in C_b$ and~$i+1<d$, take an elementary tensor~$\eta\otimes \zeta\in\Hilm[G]_i$. We define $$\varphi_{\Hilm[G]_i}(\xi)(\eta\otimes \zeta)\coloneqq(\xi\cdot\eta) \otimes \zeta\in\Hilm[G]_{i+1}.$$ If~$i+1=d$, we use the relation~$b^da=ab^c$ and that~$(C_g)_{g\in \mathrm{BS}(c,d)}$ is extended from~$\mathrm{BS}(c,d)^+$ to identify the multiplication $\xi\cdot\eta\otimes c$ with an element of~$\Hilm[G]_0$. Notice that $a\vee b=b^{d}a=ab^c$ and hence $C_{b}^*C_a\subseteq C_{b^{d-1}a}C^*_{b^c}$. This guarantees that~$\varphi_{\Hilm[G]}(\xi)$ is adjointable for all~$\xi\in C_b$ and~$\varphi_{\Hilm[G]}(\xi)^*$ is given in a similar way by multiplication with~$\xi^*$. This produces a \Star homomorphism~$\varphi_{\Hilm[G]}\colon B\to\Bound(\Hilm[G])$, which turns~$\Hilm[G]$ into a correspondence over~$B$. Using the relation~$b^{d}a=ab^c$ and also $ab^{-c}=b^{-d}a$, we deduce that~$C_{b^m}C_aC_a^*C_{b^n}$ is contained in~$C_{b^{i}a}\cdot B\cdot C_{b^{j}a}^*$ in~$\Cst\big((C_g)_{g\in F}\big)$, where~$0\leq i,j<d$ are uniquely determined by~$m$ and~$n$, respectively, and $m,n\in \ZZ$. From this we see that Katsura's ideal for~$\Hilm[G]$ is $$I_{\Hilm[G]}=\overline{\mathrm{span}}\{C_{b^m}C_aC_a^*C_{b^n}\vert\, m,n\in\ZZ\}\idealin B,$$ since the left action of~$B$ on~$\Hilm[G]$ involves the multiplication on~$(C_g)_{g\in\mathrm{BS}(c,d)}$. 

Because~$C_p^*C_q=\{0\}$ whenever~$p$ and~$q$ have reduced forms~$b^{s_0}a$ and~$b^{t_0}a$ with~$s_0\neq t_0$, we have a canonical representation of~$\Hilm[G]$ in~$\Cst((C_g)_{g\in \mathrm{BS}(c,d)})$ coming from the identification~$\Hilm[G]_i\cong C_{b^ia}B$. If $g\in  \mathrm{BS}(c,d)$ has normal form $b^{s_0}ab^m$ with $m\leq 0$, we have~$g\vee e=b^{s_0}a$ and $g^{-1}\vee e= b^{-m}$. This implies that the representation of~$\Hilm[G]$ in~$\Cst((C_g)_{g\in \mathrm{BS}(c,d)})$ is Cuntz--Pimsner covariant on~$I_{\Hilm[G]}$, Furthermore, it is injective and gauge-compatible. This gives a surjective \Star homomorphism $\phi\colon\CP_{I_{\Hilm[G]},\Hilm[G]}\to\Cst((C_g)_{g\in \mathrm{BS}(c,d)})$ because~$(C_g)_{g\in \mathrm{BS}(c,d)}$ is extended from the positive cone~$\mathrm{BS}(c,d)^+$. Now \cite[Theorem 6.4]{Katsura:Cstar_correspondences} shows that~$\phi$ is an isomorphism.
\end{proof}
\end{lem}

As for free groups, we obtain nuclearity and exactness results for Fell bundles extended from~$\mathrm{BS}(c,d)^+$.
\begin{cor} Let~$(C_g)_{g\in \mathrm{BS}(c,d)}$ be a Fell bundle extended from~$\mathrm{BS}(c,d)^+$.  If~$C_e$ is nuclear (resp. exact), then $\Cst((B_g)_{g\in \mathrm{BS}(c,d)})$ is nuclear (resp. exact).
\begin{proof} If~$C_e$ is nuclear (resp. exact), then~$B$ is nuclear (resp. exact). Using the description of $\Cst((C_g)_{g\in \mathrm{BS}(c,d)})$ as a Katsura's relative Cuntz--Pimsner algebra of a correspondence over~$B$ from Lemma~\ref{lem:pic_forBaumslag-Sol}, we then conclude that $\Cst((C_g)_{g\in \mathrm{BS}(c,d)})$ is nuclear (resp. exact) whenever~$C_e$ is.
\end{proof}
\end{cor}

Again using gauge-invariant uniqueness theorems for Cuntz--Pimsner algebras, we have the following.
\begin{prop} A Fell bundle $(C_g)_{g\in \mathrm{BS}(c,d)}$ extended from~$\mathrm{BS}(c,d)^+$ is amenable.
\begin{proof} Employing the same argument used in Proposition~\ref{prop:amenability_free_groups}, we deduce that $\Cst_r((C_g)_{g\in \mathrm{BS}(c,d)})$ also carries a topological $\ZZ$\nb-grading, for which the regular representation is compatible. Thus~$\Lambda\colon \Cst((C_g)_{g\in \mathrm{BS}(c,d)})\to\Cst_r((C_g)_{g\in \mathrm{BS}(c,d)})$ produces a gauge-compatible representation of~$\CP_{I_{\Hilm[G]},\Hilm[G]}$ that is faithful on~$B$, so that the gauge-invariant uniqueness theorem for~$\CP_{I_{\Hilm[G]},\Hilm[G]}$ implies the desired isomorphism. 
\end{proof}
\end{prop}

 \section{Functoriality for relative Cuntz--Pimsner algebras}\label{sec:bicategorical-part}
In this section we introduce the bicategory of compactly aligned product systems~$\Corr^P_{\proper}$ and its sub-bicategory of simplifiable product systems of Hilbert bimodules~$\Corr^P_{\Bim}$. We upgrade the main results from Section~\ref{sec:Fell_bundles_positive} to an equivalence between~$\Corr^P_{\Bim}$  and a bicategory of Fell bundles over~$G$ that are extended from the positive cone~$P$. We also prove that the construction of a relative Cuntz--Pimsner algebra is part of a functor from~$\Corr^P_{\proper}$ into a bicategory of~$\Cst$\nb-algebras.

\subsection{Bicategories of compactly aligned product systems} We define covariant correspondences between compactly aligned product systems as in \cite[Definition~2.20]{Meyer-Sehnem:Bicategorical_Pimsner}, also following the ideas of Schweizer~\cite{Schweizer:Crossed_Cuntz-Pimsner}. Let $(G,P)$ be a quasi-lattice ordered group. Let $\Hilm=(\Hilm_p)_{p\in P}$ and $\Hilm[G]=(\Hilm[G]_p)_{p\in P}$ be compactly aligned product systems of correspondences over $\Cst$\nb-algebras~$A$ and~$B$, respectively. Let~$\mathcal{J}_A=\{J^A_p\}_{p\in P}$ and~$\mathcal{J}_{B}=\{J^B_p\}_{p\in P}$ be families of ideals in~$A$ and~$B$, with~$\varphi^A_p(J^A_p)\subseteq\Comp(\Hilm_p)$ and~$\varphi^B_p(J^B_p)\subseteq\Comp(\Hilm[G]_p)$ for all~$p\in P$.

\begin{defn}\label{defn:cov-correspondence} A \emph{covariant correspondence} from 
$(A,\Hilm,\mathcal{J}_A)$ to~$(B,\Hilm[G],\mathcal{J}_B)$ is a pair~$(\Hilm[F],V)$, where $\Hilm[F]\colon A\leadsto B$ is a correspondence such that~$J_p^A\Hilm[F]\subseteq\Hilm[F]J^B_p$ for all $p\in P$ and $V=\{V_p\}_{p\in P}$ is a family of correspondence isomorphisms 
$V_p\colon\Hilm_p\otimes_A\Hilm[F]\cong\Hilm[F]\otimes_B\Hilm[G]_p$, where~$V_e\colon A\otimes_A\Hilm[F]\cong\Hilm[F]\otimes_BB$ is the isomorphism which sends~$a\otimes(\xi b)$ to $\psi(a)\xi\otimes b$. These must make the following diagram commute for all~$p,q\in P$:
\begin{equation}\label{eq:coherence_covariant_corres}\begin{gathered}{\xymatrix@C12pt{(\Hilm_p\otimes_A\Hilm_q)\otimes_A\Hilm[F]  \ar@{<->}[d]_{} \ar@{->}[rr]^{\ \ \ \ \mu^{1}_{p,q}\otimes1}& & \Hilm_{pq}\otimes_A\Hilm[F]\ar@{->}[rr]^{V_{pq}}&&\Hilm[F]\otimes_B\Hilm[G]_{pq} \ar@{<-}[d]^{1\otimes \mu^2_{p,q}} \\ \Hilm_p\otimes_A(\Hilm_q\otimes_A\Hilm[F]) && &&  \Hilm[F]\otimes_B(\Hilm[G]_p\otimes_B\Hilm[G]_q)  \\  \Hilm_p\otimes_A(\Hilm[F]\otimes_B\Hilm[G]_q)
\ar@{<->}[rr]^{}\ar@{<-}[u]^{1\otimes V_q} & &(\Hilm_p\otimes_A\Hilm[F])\otimes_B\Hilm[G]_q \ar@{->}[rr]^{V_p\otimes 1}&& (\Hilm[F]\otimes_B\Hilm[G]_p)\otimes_B\Hilm[G]_q \ar@{<->}[u]^{} .}}
\end{gathered}\end{equation}
A covariant correspondence~$(\Hilm[F], V)$ is called \emph{proper} if~$\Hilm[F]$ is a proper correspondence. 
\end{defn}

\begin{defn}
  \label{def:P_bicategory}
  The bicategory \(\Corr^P\)
  has the following data:
  \begin{itemize}
  \item Objects are triples \((A, \Hilm,\Hilm[J])\),
    where \(A\)
    is a \(\Cst\)\nb-algebra,~$\Hilm=(\Hilm_p)_{p\in P}$
    is a compactly aligned product system over~$P$ of $A$\nb-correspondences,
    and \(\Hilm[J]=\{J_p\}_{p\in P}\) is a family of ideals in~$A$ with $J_p\subseteq\varphi_p^{-1}(\Comp(\Hilm_p))$ for all~$p\in P$.
  \item Arrows \((A, \Hilm, \Hilm[J])\to (A_1, \Hilm_1,\Hilm[J]_1)\)
    are covariant correspondences \((\Hilm[F],V)\)
    from \((A,\Hilm,\Hilm[J])\)
    to~\((A_1, \Hilm_1,\Hilm[J]_1)\).

  \item \(2\)\nb-Arrows
    \((\Hilm[F]_0,V_0) \Rightarrow (\Hilm[F]_1,V_1)\)
    are isomorphisms of covariant correspondences, that is,
    correspondence isomorphisms
    \(w\colon\Hilm[F]_0\rightarrow \Hilm[F]_1\)
    for which the following diagram commutes for all~$p\in P$:
    \[
    \xymatrix{
      \Hilm_p\otimes_A\Hilm[F]_0\ar@{->}[r]^{V_{0,p}} \ar@{->}[d]_{1_{\Hilm_p}\otimes w}&
      \Hilm[F]_0\otimes_{A_1}\Hilm_{1,p}\ar@{->}[d]^{w\otimes 1_{\Hilm_{1,p}}} \\
      \Hilm_p\otimes_A\Hilm[F]_1 \ar@{->}[r]^{V_{1,p}} &
      \Hilm[F]_1\otimes_{A_1}\Hilm_{1,p}.
    }
    \]

  \item The vertical product of \(2\)\nb-arrows
    \[
    w_0\colon(\Hilm[F]_0,V_0)\Rightarrow(\Hilm[F]_1,V_1),\qquad
    w_1\colon(\Hilm[F]_1,V_1)\Rightarrow(\Hilm[F]_2,V_2)
    \]
    is the usual product
    \(w_1\cdot w_0\colon\Hilm[F]_0\to\Hilm[F]_2\).  
    The
    arrows \[(A, \Hilm, \Hilm[J])\to (A_1, \Hilm_1,\Hilm[J]_1)\]
    and the \(2\)\nb-arrows
    between them form a groupoid
    \(\Corr^P\!\big(\!(\!A, \Hilm, \!\Hilm[J]),\! (\!A_1, \Hilm_1,\!\Hilm[J]_1)\!\big)\).

  \item Let \((\Hilm[F],V)\colon(A, \Hilm, \Hilm[J])\to(A_1, \Hilm_1, \Hilm[J]_1)\) and \(
   (\Hilm[F]_1,V_1)\colon(A_1, \Hilm_1, \Hilm[J]_1)\to(A_2,\allowbreak \Hilm_2,
    \Hilm[J]_2)\) be arrows. For each~$p\in P$, let $V_p\bullet V_{1,p}$ be the composite correspondence isomorphism
    \[\ \ \ \ \quad \ \ \ \ \Hilm_p\!\otimes_A\Hilm[F]\!\otimes_{A_1}\Hilm[F]_1   \overset{V_p\otimes 1_{\Hilm[F]_1}}{\longrightarrow}
    \Hilm[F]\!\otimes_{A_1}\Hilm_{1,p}\!\otimes_{A_1}\!\Hilm[F]_1   \overset{1_{\Hilm[F]}\otimes V_{1,p}}{\longrightarrow}
    \Hilm[F]\!\otimes_{A_1}\!\Hilm[F]_1\!\otimes_{A_2}\!\Hilm_{2,p}.
    \] We define the product  \((\Hilm[F]_1,V_1)\circ(\Hilm[F],V)\) by 
    \[(\Hilm[F]_1,V_1)\circ(\Hilm[F],V)\defeq
    (\Hilm[F]\otimes_{A_1}\Hilm[F]_1, V\bullet V_1),\]
    where~\(V\bullet V_1=\{V_p\bullet V_{1,p}\}_{p\in P}.\)

  \item The horizontal product for a diagram of arrows and
    \(2\)\nb-arrows
    \[
    \xymatrix{(A, \Hilm, \Hilm[J])\ar@/^1pc/[rr]^{(\Hilm[F],V)}="a"
      \ar@/^-1pc/[rr]_{(\widetilde{\Hilm[F]},\widetilde{V})}="b"&
      &(A_1, \Hilm_1, J_1) \ar@/^1pc/[rr]^{(\Hilm[F]_1,V_1)}="c"
      \ar@/^-1pc/[rr]_{(\widetilde{\Hilm[F]_1},\widetilde{V_1})}="d"
      & &(A_2, \Hilm_2, J_2) \ar@{=>}^{w}"a";"b" \ar@{=>}^{w_1}"c";"d"
    }
    \]
    is the \(2\)\nb-arrow
    \[
    \xymatrix{
      (A, \Hilm, \Hilm[J]) \ar@/^1pc/[rrrr]^{(\Hilm[F]\otimes_{A_1}\Hilm[F]_1,V\bullet V_1)}="a"
      \ar@/^-1pc/[rrrr]_{(\widetilde{\Hilm[F]}\otimes_{A_1}\widetilde{\Hilm[F]_1},
        \widetilde{V}\bullet\widetilde{V_1})}="b" & & &&
      ( A_2, \Hilm_2, \Hilm[J]_2).
      \ar@{=>}^{w\otimes w_1 }"a";"b"}
    \]
    This horizontal product and the product of arrows produce
    composition bifunctors
    \begin{multline*}
      \quad\Corr^P((A, \Hilm, \Hilm[J]), (A_1, \Hilm_1,\Hilm[J]_1))
      \times \Corr^P((A_1, \Hilm_1, \Hilm[J]_1), (A_2, \Hilm_2,\Hilm[J]_2))\\
      \to \Corr^P((A, \Hilm, \Hilm[J]), (A_2, \Hilm_2,\Hilm[J]_2)).
    \end{multline*}

  \item The unit arrow on the object~\((A,\Hilm,\Hilm[J])\)
    is the proper covariant correspondence~\((A,\iota_{\Hilm})\),
    where~\(A\)
    is the identity correspondence
    and~\(\iota_{\Hilm}=\{\iota_{\Hilm_p}\}_{p\in P}\) is the family of canonical isomorphisms 
    \[
    \Hilm_p\otimes_A A \cong  \Hilm_p \cong  A\otimes_A \Hilm_p
    \]
    obtained from the right and left actions of~\(A\) on~\(\Hilm\).

  \item The associators and unitors are the same as in the
    correspondence bicategory~\cite[Section~2]{Buss-Meyer-Zhu:Higher_twisted}. That is, $$a\colon\Hilm[F]_0\otimes_{A}(\Hilm[F]_1\otimes_{A_1}\Hilm[F]_{2})\xRightarrow{\cong} (\Hilm[F]_0\otimes_{A}\Hilm[F]_1)\otimes_{A_1}\Hilm[F]_{2}$$ is the obvious isomorphism, and the isomorphisms $A\otimes_A\Hilm[F]\xRightarrow{\cong} \Hilm[F]$ and $\Hilm[F]\otimes_{A_1}A_1\xRightarrow{\cong}\Hilm[F]$ implement the left and right actions of~$A$ and~$A_1$, respectively.  
     \end{itemize} 
    
      We let~$\Corr^P_{\proper}$ be the sub-bicategory of~$\Corr^P$ whose arrows are~\emph{proper} covariant correspondences. 
\end{defn}

\begin{defn} We denote by $\Corr^P_{\proper,\ast}$ the full sub-bicategory of~$\Corr^P_{\proper}$ whose objects are triples $(A,\Hilm, \Hilm[I])$, where~$\Hilm$ is a simplifiable product system of Hilbert bimodules and~$\Hilm[I]=\{I_p\}_{p\in P}$ is the family of Katsura's ideals for~$\Hilm$, that is, $I_p=\BRAKET{\Hilm_p}{\Hilm_p}$ for all~$p\in P$. 
\end{defn}

\begin{example}\label{ex:canonical_proper_ccorres} View~$\CP^e_{\Hilm[J],\Hilm}$ as a correspondence~$A\leadsto \CP^e_{\Hilm[J],\Hilm}$. For each~$p\in P$, let $\hat{\iota}_{\Hilm_p}$ be the isomorphism $$\Hilm_p\otimes_A \CP^e_{\Hilm[J],\Hilm}\cong\CP^p_{\Hilm[J],\Hilm}\cong \CP^e_{\Hilm[J],\Hilm}\otimes_{\CP^e_{\Hilm[J],\Hilm}}\CP^p_{\Hilm[J],\Hilm},$$ where the first isomorphism is that of Corollary~\ref{cor:positive_fibres_isomorphism}. Cuntz--Pimsner covariance on~$\Hilm[J]=\{J_p\}_{p\in P}$ implies that $$j_e(J_p)\subseteq j_p(\Hilm_p)j_p(\Hilm_p)^*\subseteq \CP^p_{\Hilm[J],\Hilm}\CP^{p\,\,*}_{\Hilm[J],\Hilm}$$
for all $p\in P$. So $(\CP^e_{\Hilm[J],\Hilm}, \hat{\iota}_{\Hilm})$ is a proper covariant correspondence $$(A,\Hilm,\Hilm[J])\to(\CP^e_{\Hilm[J],\Hilm}, (\CP^p_{\Hilm[J],\Hilm})_{p\in P}, \Hilm[I]_{\CP_{\Hilm[J],\Hilm}}),$$ where~$\hat{\iota}_{\Hilm}=\{\hat{\iota}_{\Hilm_p}\}_{p\in P}$ and $\Hilm[I]_{\CP_{\Hilm[J],\Hilm}}=\{I_p^{\CP_{\Hilm[J],\Hilm}}\}_{p\in P}$ with $$I_p^{\CP_{\Hilm[J],\Hilm}}=\CP^p_{\Hilm[J],\Hilm}\CP^{p\,\,*}_{\Hilm[J],\Hilm}=\BRAKET{\CP^p_{\Hilm[J],\Hilm}}{\CP^p_{\Hilm[J],\Hilm}}.$$

\end{example}

In order to prove that the construction of a relative Cuntz--Pimsner algebra is functorial, we begin by building correspondences between the underlying relative Cuntz--Pimsner algebras out of morphisms in~$\Corr^P_{\proper}$. 
\begin{prop}\label{prop:induced_correspondence} Let~$(\Hilm[F],V)\colon(A,\Hilm, \mathcal{J}_A)\to(B,\Hilm[G], \mathcal{J}_B)$ be a proper covariant correspondence. It induces a proper correspondence~$\CP_{\Hilm[F],V}\colon\CP_{\mathcal{J}_A,\Hilm}\leadsto\CP_{\mathcal{J}_B,\Hilm[G]}.$ In particular, a morphism in~$\Corr^P_{\Bim}$ between two simplifiable product systems of Hilbert bimodules produces a proper correspondence between the cross-sectional $\Cst$\nb-algebras of the associated Fell bundles.
\end{prop} 
\begin{proof}Let~$\Hilm[F]_{\CP}\coloneqq\Hilm[F]\otimes_B\CP_{\mathcal{J}_{B},\Hilm[G]}$. This is a proper correspondence~$A\leadsto\CP_{\mathcal{J}_{B},\Hilm[G]}$. We define a family of isometries $V^!=\{V_p^!\}_{p\in P}$ by setting, for all~$p\in P$, \[V_p^!\colon\Hilm_p\otimes_A\Hilm[F]_{\CP}=\Hilm_p\otimes_A\Hilm[F]\otimes_B\CP_{\mathcal{J}_B,\Hilm[G]}\xRightarrow{V_p\otimes\id}\Hilm[F]\otimes_B\Hilm[G]_p\otimes_B\CP_{\mathcal{J}_{B},\Hilm[G]}\xRightarrow{\id\otimes\mu_{\Hilm[G]_p}}\Hilm[F]_{\CP},\] where $\mu_{\Hilm[G]_p}$ is the isometry $\Hilm[G]_p\otimes_B\CP_{\mathcal{J}_B,\Hilm[G]}\Rightarrow\CP_{\mathcal{J}_B,\Hilm[G]}$ obtained from the representation of~$\Hilm[G]_p$ in $\CP_{\mathcal{J}_B,\Hilm[G]}$. For each~$\xi\in\Hilm_p$, we set $$\psi_p(\xi)(\eta)\coloneqq V_p^!(\xi\otimes_A\eta),\quad\eta\in\Hilm[F]_\CP.$$  Because~$\Hilm[F]_{\CP}$ is a proper correspondence, the map~$\eta\mapsto\xi\otimes_A\eta$ gives a compact operator from $\Hilm[F]$ to $\Hilm_p\otimes_A\Hilm[F]_{\CP}$. This is mapped to~$\Comp(\Hilm[F]_{\CP})$ when composed with~$V_p^!$ by \cite[Lemma~2.1]{Meyer-Sehnem:Bicategorical_Pimsner}. In particular, $\psi_p(\xi)$ is adjointable. The coherence axiom~\eqref{eq:coherence_covariant_corres} for~$(\Hilm[F],V)$ implies that~$\psi=\{\psi_p\}_{p\in P}$ preserves the multiplication on~$\Hilm$. In addition, for all~$\xi,\eta\in \Hilm_p$ and~$\zeta,\zeta'\in\Hilm[F]_{\CP}$, we have that $$\braket{\psi_p(\xi)^*\psi_p(\eta)\zeta}{\zeta'}=\braket{\psi_p(\eta)\zeta}{\psi_p(\xi)\zeta'}=\braket{\zeta}{\psi_e(\braket{\eta}{\xi})\zeta'}=\braket{\psi_e(\braket{\xi}{\eta})\zeta}{\zeta'}$$ provided~$V_p^!$ is an isometry. Therefore, $\psi=\{\psi_p\}_{p\in P}$ is a representation of~$\Hilm$ by compact operators on~$\Hilm[F]_\CP$.

We are left with the task of proving that~$\psi$ factors through~$\CP_{\mathcal{J}_A,\Hilm}$. To do so, we will first prove that it is Cuntz--Pimsner covariant on~$\mathcal{J}_A=\{J_p^A\}_{p\in P}$. The Nica covariance condition will then follow from the fact that the $G$\nb-grading of~$\CP_{\mathcal{J}_B,\Hilm[G]}$ is a Fell bundle extended from~$P$. The representation of~$\Hilm[G]$ in~$\CP_{\mathcal{J}_B,\Hilm[G]}$ is covariant on~$\mathcal{J}_B$. Hence the \Star homomorphism $j_{J_B}\colon B\to\CP_{\mathcal{J}_B,\Hilm[G]}$ satisfies  $$j_{\mathcal{J}_B}(J_p^B)\CP_{\mathcal{J}_B,\Hilm[G]}\subseteq \mu_{\Hilm[G]_p}(\Hilm[G]_p\otimes_B\CP_{\mathcal{J}_B,\Hilm[G]})$$ for all~$p\in P$. It follows that~$\psi_e(J_p^A)$ maps~$\Hilm[F]_{\CP}$ into~$\Hilm[F]\otimes_B\mu_{\Hilm[G]_p}(\Hilm[G]_p\otimes_B\CP_{\mathcal{J}_B,\Hilm[G]})$, provided $J_p^A\Hilm[F]\subseteq\Hilm[F]J_p^B$. Using that~$V_p$ is unitary, we see that this coincides with~$\psi_p(\Hilm_p)\Hilm[F]_{\CP}$. Proposition~\ref{prop:equivalent_covariant_cond} ensures that~$\psi$ is covariant on~$\Hilm[J]_A$. 

To see that~$\psi$ is also Nica covariant, let~$p,q\in P$, $T\in\Comp(\Hilm_p)$ and $S\in\Comp(\Hilm_q)$. Since~$$\psi_p(\Hilm_p)(\Hilm[F]_{\CP})\subseteq \Hilm[F]\otimes_B\mu_{\Hilm[G]_p}(\Hilm[G]_p\otimes_B\CP_{\mathcal{J}_B,\Hilm[G]})$$ and the representation of~$\Hilm[G]$ in $\CP_{\mathcal{J}_B,\Hilm[G]}$ is Nica covariant, it follows that~$$\psi^{(p)}(T)(\Hilm[F]\otimes_B\CP^{e}_{\mathcal{J}_B,\Hilm[G]})\subseteq\Hilm[F]\otimes_B\CP^p_{\Hilm[J]_B,\Hilm[G]}\CP^{p\,\,*}_{\Hilm[J]_B,\Hilm[G]}.$$ The same reasoning shows that~$$\psi^{(q)}(S)(\Hilm[F]\otimes_B\CP^{e}_{\mathcal{J}_B,\Hilm[G]})\subseteq\Hilm[F]\otimes_B\CP^q_{\Hilm[J]_B,\Hilm[G]}\CP^{q\,\,*}_{\Hilm[J]_B,\Hilm[G]}.$$ Now $\CP_{\mathcal{J}_B,\Hilm[G]}=\CP^e_{\mathcal{J}_B,\Hilm[G]}\CP_{\mathcal{J}_B,\Hilm[G]}$ and so we deduce that~$\psi^{(p)}(T)\psi^{(q)}(S)=0$ if $p\vee q=\infty$ because $(\CP^p_{\mathcal{J}_B,\Hilm[G]})_{g\in G}$ is simplifiable by Corollary~\ref{cor:CP_simplifiable_psystem}. In case~$p\vee q<\infty$, we then have $$\psi^{(p)}(T)\psi^{(q)}(S)(\Hilm[F]_{\CP})\subseteq \Hilm[F]\otimes_B\mu_{\Hilm[G]_{p\vee q}}(\Hilm[G]_{p\vee q}\otimes_B\CP_{\mathcal{J}_B,\Hilm[G]}).$$ The right-hand side above is contained in 
$\psi_{p\vee q}(\Hilm_{p\vee q})\Hilm[F]_{\CP}$, provided~$V_{p\vee q}$ is unitary. So we may argue as in Proposition~\ref{prop:equivalent_covariant_cond} to deduce that~$\psi$ is Nica covariant and therefore descends to a \Star homomorphism $\CP_{\mathcal{J}_{A},\Hilm}\to\Comp(\Hilm[F]_{\CP}),$ as desired. The last assertion in the statement follows from the fact that~$\Cst((\hat{\Hilm}_g)_{g\in G})$ is canonically isomorphic to~$\CP_{\mathcal{I}_{\Hilm},\Hilm}$ whenever~$\Hilm$ is a simplifiable product system of Hilbert bimodules (see Proposition~\ref{prop:associated_product_system}).
\end{proof}

Our next goal is to enrich the correspondence found in Section~\ref{sec:Fell_bundles_positive} between simplifiable product systems of Hilbert bimodules and Fell bundles extended from positive cones to an equivalence of bicategories, using the $\Cst$\nb-correspondence built in the previous proposition.

\begin{defn} Let $(B_g)_{g\in G}$ and~$(C_g)_{g\in G}$ be Fell bundles extended from~$P$. A correspondence $(\Hilm[F], U)\colon(B_g)_{g\in G}\to(C_g)_{g\in G}$ consists of a $\Cst$\nb-correspondence $\Hilm[F]\colon B_e\leadsto C_e$ and a family of isometries~$U=\{U_g\}_{g\in G}$, where~$U_g\colon B_g\otimes_{B_e}\Hilm[F]\to\Hilm[F]\otimes_{C_e}C_g$, such that $U_e\colon B_e\otimes_{B_e}\Hilm[F]\cong\Hilm[F]\otimes_{C_e}C_e$ is the isomorphism which sends $b\otimes(\xi c)$ to $\psi(b)\xi\otimes c$ and, for all~$p\in P$, $U_p$ is unitary. Here we are regarding the $B_g$'s as correspondences over~$B_e$. We also require the following diagram to commute for all~$g,h\in G$:  
\begin{equation}\label{def:bundle_correspondence}\begin{gathered}\xymatrix@C14pt{(B_g\!\otimes_{B_e}\!B_h)\!\otimes_{B_e}\!\Hilm[F]  \ar@{<->}[d]_{} \ar@{->}[rr]^{\ \ \ \ \hat{\mu}_{g,h}\otimes1}& & B_{gh}\!\otimes_{B_e}\!\Hilm[F]\ar@{->}[rr]^{U_{gh}}&&\Hilm[F]\otimes_{C_e}\!C_{gh} \ar@{<-}[d]^{1\otimes \hat{\mu}^{1}_{g,h}} \\ \!B_g\!\otimes_{B_e}\!(B_h\!\otimes_{B_e}\!\Hilm[F]) && &&  \Hilm[F]\!\otimes_{C_e}\!(C_g\!\otimes_{C_e}\!C_h)  \\  B_g\!\otimes_{B_e}\!(\Hilm[F]\!\otimes_{C_e}\!C_h)
\ar@{<->}[rr]^{}\ar@{<-}[u]^{1\otimes U_h} & &(B_g\!\otimes_{B_e}\!\Hilm[F])\!\otimes_{C_e}\!C_h \ar@{->}[rr]^{U_g\otimes 1}&& (\Hilm[F]\!\otimes_{C_e}\!C_g)\!\otimes_{C_e}\!C_h \ar@{<->}[u]^{} .}\end{gathered}\end{equation}
A correspondence~$(\Hilm[F], U)$ is \emph{proper} if~$\Hilm[F]$ is a proper correspondence. 
\end{defn} 

It is not clear to us whether all of the~$U_g$'s in the above definition are unitary whenever the~$U_p$'s are so.

\begin{defn} We will denote by \(\Corr^{(G,P)}\) the bicategory whose objects are Fell bundles over~$G$ extended from~$P$ and arrows $(B_g)_{g\in G}\rightarrow(C_g)_{g\in G}$ are correspondences as above. A $2$\nb-morphism $w\colon(\Hilm[F]_0,U_0)\Rightarrow(\Hilm[F]_1,U_1)$ is a correspondence isomorphism~$w\colon\Hilm[F]_0\to\Hilm[F]_1$ making the following diagram commute for all~$g\in G$:
  \[
    \xymatrix{
      B_g\otimes_{B_e}\Hilm[F]_0\ar@{->}[r]^{U_{0,g}} \ar@{->}[d]_{1_{B_g}\otimes w}&
      \Hilm[F]_0\otimes_{C_e}C_g\ar@{->}[d]^{w\otimes 1_{C_{g}}} \\
      B_g\otimes_{B_e}\Hilm[F]_1 \ar@{->}[r]^{U_{1,g}} &
      \Hilm[F]_1\otimes_{C_e}C_{g}.
    }
    \]The unit arrow on an object~$(B_g)_{g\in G}$ is the identity correspondence $B_e\colon B_e\to B_e$ with the family of isomorphisms~$\hat{\iota}_{G}=\{\iota_{B_g}\}_{g\in G}$, where~$\iota_{B_g}$ is the isomorphism~$B_e\otimes_{B_e}B_g\cong B_g\otimes_{B_e}B_e$ obtained as in Definition~\ref{def:P_bicategory}. The further data needed for a bicategory is also defined as in Definition~\ref{def:P_bicategory}. We let \(\Corr^{(G,P)}_{\proper}\) be the sub-bicategory of~\(\Corr^{(G,P)}\) whose arrows are proper correspondences.
\end{defn}

\begin{lem}\label{lem:restricted_corres} Let~$(\Hilm[F],U)\colon (B_g)_{g\in G}\to(C_g)_{g\in G}$ be a morphism in~\(\Corr^{(G,P)}_{\proper}\). Then its restriction to the positive fibres is a proper covariant correspondence~$$(B_e, \Hilm[B], \Hilm[I]_{\Hilm[B]})\to(C_e, \Hilm[C], \Hilm[I]_{\Hilm[C]}),$$ where~$\Hilm[I]_{\Hilm[B]}$ and $\Hilm[I]_{\Hilm[C]}$ denote the families of Katsura's ideals for~$\Hilm[B]$ and~$\Hilm[C]$, respectively. Moreover, if $U'=\{U_g'\}_{g\in G}$ is another family of isometries turning~$\Hilm[F]$ into a correspondence from~$(B_g)_{g\in G}$ to~$(C_g)_{g\in G}$ and such that~$U'_p=U_p$ for all~$p\in P$, then~$U'_g=U_g$ for all~$g\in G$.
\end{lem}
\begin{proof} Let~$(\Hilm[F],U)$ be a proper correspondence from~$(B_g)_{g\in G}$ to~$(C_g)_{g\in G}$. By definition, $U_p\colon B_p\otimes_{B_e}\Hilm[F]\to\Hilm[F]\otimes_{C_e}C_p$  is unitary whenever~$p$ belongs to the positive cone~$P$. Thus, all we need to prove is that the ideal~$\BRAKET{B_p}{B_p}$ maps~$\Hilm[F]$ into~$\Hilm[F]\BRAKET{C_p}{C_p}$. This follows from~\eqref{def:bundle_correspondence}. We let~$p^{-1}$ play the role of~$q$ and obtain the commutative diagram $$\xymatrix{B_p\!\otimes_{B_e}\!B^*_{p}\!\otimes_{B_e}\!\Hilm[F]  \ar@{->}[rr]^{\hat{\mu}_{p,p^{-1}}\otimes1}& & B_pB_p^*\!\otimes_{B_e}\!\Hilm[F]\ar@{->}[rr]^{U_{e}}&&\Hilm[F]\otimes_{C_e}\!C_e \ar@{<-}[d]^{1\otimes \hat{\mu}^{1}_{p,{p^{-1}}}} \\  B_p\!\otimes_{B_e}\!\Hilm[F]\!\otimes_{C_e}\!C^*_{p} \ar@{<-}[u]^{1\otimes U_{p^{-1}}} \ar@{->}[rrrr]^{U_p\otimes 1} &&  &&\Hilm[F]\!\otimes_{C_e}\!C_p\!\otimes_{C_e}\!C^*_{p}  .}$$ The image of the top map is~$\BRAKET{B_p}{B_p}\Hilm[F]$ and the image of the right map is~$\Hilm[F]\BRAKET{C_p}{C_p}$. Hence $\BRAKET{B_p}{B_p}\Hilm[F]\subseteq\Hilm[F]\BRAKET{C_p}{C_p}$.

We are left with the task of proving that~$U=\{U_g\}_{g\in G}$ is completely determined by the unitaries~$\{U_p\}_{p\in P}$. We begin by using the~$U_g$'s to make $\Hilm[F]_{\CP}=\Hilm[F]\otimes_{C_e}\Cst((C_g)_{g\in G})$ into a proper correspondence~$\Cst((B_g)_{g\in G})\leadsto\Cst((C_g)_{g\in G})$, essentialy repeating the argument employed in the proof of Proposition~\ref{prop:induced_correspondence}.

For each~$g\in G$, we let~$U_g^!$ be the isometry given by the composite \[B_g\otimes_{B_{\!e}}\Hilm[F]_{\CP}\!=\!B_g\otimes_{B_{\!e}}\Hilm[F]\otimes_{C_{\!e}}\!\Cst((C_g\!)_{g\in G}\!)\overset{U_g\otimes\id}{\Longrightarrow}\Hilm[F]\otimes_{C_{\!e}} C_g\otimes_{C_{\!e}}\!\Cst((C_g\!)_{g\in G}\!)\overset{\id\otimes\mu_{C_g}}{\Longrightarrow}\Hilm[F]_{\CP},\] where $\mu_{C_g}$ is the isometry $C_g\otimes_{C_{e}}\Cst((C_g)_{g\in G})\Rightarrow\Cst((C_g)_{g\in G})$ obtained from the representation of~$C_g$ in $\Cst((C_g)_{g\in G})$. We set~$U^!=\{U_g^!\}_{g\in G}$. For each~$\xi\in B_g$, we let $$\psi_g(\xi)(\eta)\coloneqq U_p^!(\xi\otimes_{B_e}\eta),\quad\eta\in\Hilm[F]_\CP.$$  As in Proposition~\ref{prop:induced_correspondence}, $\eta\mapsto\xi\otimes_{B_e}\eta$ gives a compact operator from $\Hilm[F]$ to $B_g\otimes_A\Hilm[F]_{\CP}$. This is mapped to~$\Comp(\Hilm[F]_{\CP})$ when composed with~$U_g^!$ and so $\psi_g(\xi)$ is adjointable. Observe that this is just the representation built in Proposition~\ref{prop:induced_correspondence} when restricted to the positive fibres if we identify $\Cst((C_g)_{g\in G})$ with the relative Cuntz--Pimsner algebra~$\CP_{\Hilm[I]_{\Hilm[C]},\Hilm[C]}$ through Proposition~\ref{prop:associated_product_system}.

Clearly~$\psi=\{\psi_g\}_{g\in G}$ is compatible with the multiplication operation on~$(B_g)_{g\in G}$ because~$(\Hilm[F],U)$ satisfies the coherence axiom~\eqref{def:bundle_correspondence}. Thus all we need to prove to conclude that it is a representation of~$(B_g)_{g\in G}$ is that~$\psi$ also preserves the involution. Indeed, it suffices to show that, for all $p\in P$ and $\xi^*\in B_p^*= B_{p^{-1}}$, we have $\psi_{p^{-1}}(\xi^*)=\psi_p(\xi)^*$. We may suppose that $\xi=\xi_1\braket{\xi_2}{\xi_3}$ using the isomorphism~$B_p\cong B_pB_p^*B_p$. Let~$\eta_1,\eta_2\in \Hilm[F]_\CP$. We have 
\begin{align*}\braket{\psi_{p^{-1}}(\xi^*)\eta_1}{\eta_2}&=\braket{\psi_e(\braket{\xi_3}{\xi_2})\psi_{p^{-1}}(\xi_1^*)\eta_1}{\eta_2}\\&=\braket{\psi_p(\xi_2)\psi_{p^{-1}}(\xi_1^*)\eta_1}{\psi_p(\xi_3)\eta_2}\\&=\braket{\psi_e(\BRAKET{\xi_2}{\xi_1})\eta_1}{\psi_p(\xi_3)\eta_2}\\&=\braket{\eta_1}{\psi_e(\BRAKET{\xi_1}{\xi_2})\psi_p(\xi_3)\eta_2}\\&=\braket{\eta_1}{\psi_p(\BRAKET{\xi_1}{\xi_2}\xi_3)\eta_2}\\&=\braket{\eta_1}{\psi_p(\xi_1\braket{\xi_2}{\xi_3})\eta_2}\\&=\braket{\eta_1}{\psi_p(\xi)\eta_2}.\end{align*}
This shows that~$\psi_{p^{-1}}(\xi^*)=\psi_p(\xi)^*$, Therefore, $\psi=\{\psi_g\}_{g\in G}$ is a representation of~$(B_g)_{g\in G}$ in~$\Comp(\Hilm[F]_\CP)$. It induces a \Star homomorphism~$\bar{\psi}\colon \Cst((B_g)_{g\in G})\to\Comp(\Hilm[F]_\CP)$ by the universal property of $\Cst((B_g)_{g\in G})$.

Now if $U'=\{U'_g\}_{g\in G}$ is another family of isometries making~$\Hilm[F]$ into a correspondence~$$(\Hilm[F],U')\colon (B_g)_{g\in G}\to(C_g)_{g\in G}$$ and such that~$U'_p=U_p$ for all~$p\in P$, we have a \Star homomorphism $$\bar{\psi}_{U'}\colon \Cst((B_g)_{g\in G})\to\Comp(\Hilm[F]_\CP)$$ that coincides with~$\bar{\psi}$ on the positive fibres~$(B_p)_{p\in P}$. But $\bar{\psi}=\bar{\psi}_{U'}$ if and only if~$U_g=U'_{g}$ for all~$g\in G$. Since $(B_p)_{p\in P}$ completely determines a representation of~ $\Cst((B_g)_{g\in G}$ by Proposition~\ref{prop:associated_product_system}, it follows that~$U=U'$ as claimed.
\end{proof}

\begin{thm}\label{thm:induced_equivalence} There is an equivalence of bicategories~$\Corr^P_{\Bim}\rightarrow\mathfrak{C}^{(G,P)}_{\proper}$ which sends an object~$(A,\Hilm, \Hilm[I])$ to the associated Fell bundle~$(\hat{\Hilm}_g)_{g\in G}$ extended from~$P$. 
\begin{proof} Let $(\Hilm[F],V)\colon (A,\Hilm, \Hilm[I]_{\Hilm})\to(B,\Hilm[G], \Hilm[I]_{\Hilm[G]})$ be a proper covariant correspondence. This induces a proper correspondence $\CP_{\Hilm[F],V}\colon\Cst((\hat{\Hilm}_g)_{g\in G})\!\leadsto\! \Cst((\hat{\Hilm[G]}_g)_{g\in G})$ by Proposition~\ref{prop:induced_correspondence}. By Proposition~\ref{pro:relative_embedding}, $\CP^p_{\mathcal{I}_{\Hilm[G]},\Hilm[G]}$ is isomorphic to~$\Hilm[G]_p$ and hence $\CP^e_{\Hilm[F],V}=\Hilm[F]\otimes_B\CP^e_{\mathcal{I}_{\Hilm[G]},\Hilm[G]}\cong\Hilm[F]$. Combining this with the left action of $\Cst((\hat{\Hilm}_g)_{g\in G})$ on $\CP_{\Hilm[F],V},$ we obtain isometries $$U_g\colon \hat{\Hilm}_g\otimes_A\Hilm[F]\to\Hilm[F]\otimes_B\hat{\Hilm[G]}_g$$ through the identification $\hat{\Hilm}_g\otimes_A\Hilm[F]\cong\bar{\psi}(\hat{\Hilm}_g)\Hilm[F]$. Thus $U_p=V_p$ for all $p\in P$ so that~$U=\{U_g\}_{g\in G}$ extends~$V$. 

We let~$\Hilm[F]^\sharp$ be $\Hilm[F]$ viewed as a correspondence from the unit fibre of $\hat{\Hilm}$ to that of $\hat{\Hilm[G]}$. We set~$V^\sharp_g\coloneqq U_g$ and let~$V^\sharp=\{V_g^\sharp\}_{g\in G}$. The pair~$(\Hilm[F]^\sharp,V^\sharp)$ satisfies the coherence condition~\eqref{def:bundle_correspondence} because~$\Hilm[F]^\sharp$ and~$V^\sharp$ come from a correspondence $\Cst((\hat{\Hilm}_g)_{g\in G})\leadsto\Cst((\hat{\Hilm[G]}_g)_{g\in G}).$ 

In order to show that a $2$\nb-morphism $w\colon(\Hilm[F]_0,V_0)\Rightarrow(\Hilm[F]_1,V_1)$ produces a $2$\nb-arrow $w^\sharp\colon(\Hilm[F]_0^\sharp,V^\sharp_0)\Rightarrow(\Hilm[F]_1^\sharp,V^\sharp_1)$ such that $w^\sharp=w$ as a correspondence isomorphism $\Hilm[F]_0\cong\Hilm[F]_1$, we observe that a \Star homomorphism from $\Cst((\hat{\Hilm}_g)_{g\in G})$ in a $\Cst$\nb-algebra~$B$ is completely determined by its restriction to~$(A,\Hilm)$. We may use a $2$\nb-morphism $w\colon(\Hilm[F]_0,V_0)\Rightarrow(\Hilm[F]_1,V_1)$ to define a left action of $\CP_{\mathcal{I}_{\Hilm},\Hilm}$ on~$\Hilm[F]_1\otimes_B \CP_{\mathcal{I}_{\Hilm[G]},\Hilm[G]}$ by setting~$$\bar{\psi}_1'(\,\cdot\,)\coloneqq (w\otimes 1)\bar{\psi}_0(\,\cdot\,)(w^{-1}\otimes 1).$$ This is equal to~$\bar{\psi}_1$ on~$(A,\Hilm)$. By the above observation, $\bar{\psi}'_1=\bar{\psi}_1$ on $\CP_{\mathcal{I}_{\Hilm},\Hilm}$. As a consequence, the diagram \[
    \xymatrix{
      \hat{\Hilm}_g\otimes_{A}\Hilm[F]^\sharp_0\ar@{->}[r]^{V^\sharp_{0,g}} \ar@{->}[d]_{1_{\hat{\Hilm}_g}\otimes w}&
      \Hilm[F]^\sharp_0\otimes_{B}\hat{\Hilm[G]}_g\ar@{->}[d]^{w\otimes 1_{\hat{\Hilm[G]}_{g}}} \\
      \hat{\Hilm}_g\otimes_{A}\Hilm[F]^\sharp_1 \ar@{->}[r]^{V^\sharp_{1,g}} &
      \Hilm[F]^\sharp_1\otimes_{B}\hat{\Hilm[G]}_{g}
    }
    \] commutes for all~$g\in G$. So $w$ gives a $2$\nb-arrow $w^\sharp\colon(\Hilm[F]_0^\sharp,V^\sharp_0)\Rightarrow(\Hilm[F]_1^\sharp,V^\sharp_1)$. The map~$(\Hilm[F],V)\mapsto (\Hilm[F]^\sharp,V^\sharp)$, $w\mapsto w^\sharp$ clearly gives a functor $$\Corr^P_{\Bim}\big((A,\Hilm,\Hilm[I]_{\Hilm}),(B,\Hilm[G],\Hilm[I]_{\Hilm[G]})\big)\to\mathfrak{C}^{(G,P)}_{\proper}\big((\hat{\Hilm}_g)_{g\in G},(\hat{\Hilm[G]}_g)_{g\in G}\big)$$ between the groupoids of arrows associated to the objects $(A,\Hilm,\Hilm[I]_{\Hilm})$ and $(B,\Hilm[G],\Hilm[I]_{\Hilm[G]})$. Furthermore, it follows from Lemma~\ref{lem:restricted_corres} that such a functor is an equivalence between categories since it is fully faithful and essentially surjective. 

Now we define a homomorphism of bicategories $L^*\colon\Corr^P_{\Bim}\rightarrow\mathfrak{C}^{(G,P)}_{\proper}$ by sending a simplifiable product system~$\Hilm=(\Hilm_p)_{p\in P}$ to its associated Fell bundle~$(\hat{\Hilm}_g)_{g\in G}$ and a morphism~$(\Hilm[F],V)\colon(A,\Hilm,\Hilm[I]_{\Hilm})\to(B,\Hilm[G],\Hilm[I]_{\Hilm[G]})$ to the arrow~$(\Hilm[F]^\sharp,V^\sharp)\colon(\hat{\Hilm}_g)_{g\in G}\to(\hat{\Hilm[G]}_g)_{g\in G}$ built out of $(\Hilm[F],V)$ as above. A $2$\nb-arrow~$w$ is mapped to~$w^\sharp$. Given arrows $$(\Hilm[F],V)\colon(A,\Hilm,\Hilm[I]_{\Hilm})\to(A_1,\Hilm_1,\Hilm[I]_{\Hilm_1}),\quad(\Hilm[F]_1,V_1)\colon(A_1,\Hilm_1,\Hilm[I]_{\Hilm_1})\to(A_2,\Hilm_2,\Hilm[I]_{\Hilm_2}),$$ we have that $$(\Hilm[F]\otimes_{A_1}\Hilm[F]_1)^\sharp=\Hilm[F]^\sharp\otimes_{A_1}\Hilm[F]_1^\sharp=\Hilm[F]\otimes_{A_1}\Hilm[F]_1$$ as correspondences~$A\leadsto A_2$. Moreover, the product of arrows in~$\mathfrak{C}^{(G,P)}$ is defined as in~$\Corr^P_{\,\ast}$ and Lemma~\ref{lem:restricted_corres} tells us that $(\Hilm[F]\otimes_{A_1}\Hilm[F]_1, V\bullet V_1)$ extends uniquely to a correspondence~$(\hat{\Hilm}_g)_{g\in G}\to(\hat{\Hilm}_{2,g})_{g\in G}.$ This guarantees that~$L^*$ preserves the product of arrows. Thus, this is indeed a homomorphism of bicategories. 

We have proven that $L^*$ is locally an equivalence. That it is also essentially surjective follows from the fact that every Fell bundle extended from~$P$ is isomorphic to the Fell bundle constructed out of the simplifiable product system of Hilbert bimodules determined by its positive fibres (see Proposition~\ref{prop:associated_product_system}). So by~\cite[Lemma 3.1]{Gurski:Biequivalences_tricategories}, $L^*$ is an equivalence of bicategories. 
\end{proof}
\end{thm}

\subsection{Relative Cuntz--Pimsner algebras as universal arrows} Let $\mathcal{B}$ and~$\mathcal{C}$ be bicategories and let $R\colon\mathcal{C}\to\mathcal{B}$ be a functor. Let $b\in\obj\mathcal{B}$ and~$c\in \obj\mathcal{C}$. An arrow~$g\colon b\to R(c)$ induces a functor~$$g^*\colon \mathcal{C}(R(c),x)\to\mathcal{B}(b, R(x))$$ defined by $f\mapsto R(f)\cdot g$, $w\mapsto R(w)\bullet 1_g$. It is a \emph{universal arrow} from $b$ to~$R$ if $g^*$ is an equivalence of categories \cite[Definition~9.4]{Fiore:Pseudo_biadjoints}. To each $(A,\Hilm, \mathcal{J})$ in $\obj\Corr^P_\proper$ we may associate an object $(\CP^e_{\Hilm[J],\Hilm}, (\CP^p_{\Hilm[J],\Hilm})_{p\in P}, \mathcal{I}_{\CP_{\Hilm[J],\Hilm}})$ of~$\Corr^P_\Bim$ by Corollary~\ref{cor:CP_simplifiable_psystem}. In what follows, we let $$\upsilon_{(A,\Hilm, \mathcal{J})}\colon(A,\Hilm, \mathcal{J})\rightarrow\left(\CP^e_{\Hilm[J],\Hilm}, (\CP^p_{\Hilm[J],\Hilm})_{p\in P}, \mathcal{I}_{\CP_{\Hilm[J],\Hilm}}\right)$$ be the canonical proper covariant correspondence from Example~\ref{ex:canonical_proper_ccorres}. That is, $\upsilon_{(A,\Hilm, \mathcal{J})}\coloneqq(\CP^e_{\Hilm[J],\Hilm}, \hat{\iota}_{\Hilm})$. It gives rise to a functor
 \[  \Corr^P_\Bim\bigl((\CP^e_{\Hilm[J],\Hilm},(\CP^p_{\Hilm[J],\Hilm})_{p\in P}, \mathcal{I}_{\CP_{\Hilm[J],\Hilm}}), (B,\Hilm[G],\mathcal{I}_{\Hilm[G]})\bigr)
  \to\Corr^P_\proper\bigl((A,\Hilm, \mathcal{J}), (B,\Hilm[G],\mathcal{I}_{\Hilm[G]})\bigr) \] as above. This is given by $$(\Hilm[F],V)\mapsto  (\Hilm[F],V)\circ\upsilon_{(A,\Hilm, \mathcal{J})},\;\;\qquad w\mapsto w\bullet1_{\CP^e_{\Hilm[J],\Hilm}}.$$

Our next result is a generalisation of \cite[Proposition~3.4]{Meyer-Sehnem:Bicategorical_Pimsner} to the context of compactly aligned product systems. It says that the above functor is an equivalence of categories. Hence $\upsilon_{(A,\Hilm, \mathcal{J})}$ is a universal arrow from $(A,\Hilm, \mathcal{J})$ to the inclusion functor $\Corr^P_\Bim\hookrightarrow\Corr^P_\proper$. Our argument to prove this fact is essentially that of~\cite[Proposition~3.4]{Meyer-Sehnem:Bicategorical_Pimsner} given in the context of single correspondences and so we limit ourselves to sketching some ideas. Among them, we will prove that a proper covariant correspondence $(\Hilm[F],V)\colon(A,\Hilm, \mathcal{J})\to(B,\Hilm[G],\mathcal{I}_{\Hilm[G]})$ yields a proper covariant correspondence $$(\Hilm[F]^\sharp,V^\sharp)\colon(\CP^e_{\Hilm[J],\Hilm},(\CP^p_{\Hilm[J],\Hilm})_{p\in P}, \mathcal{I}_{\CP_{\Hilm[J],\Hilm}})\to(B,\Hilm[G],\mathcal{I}_{\Hilm[G]})$$ using Proposition~\ref{prop:induced_correspondence}. This is a crucial part of the proof that the functor defined by composing objects with~$\upsilon_{(A,\Hilm, \mathcal{J})}$ is indeed essentially surjective.

 \begin{prop}\label{prop:P_universal_arrow} Let $(A,\Hilm, \mathcal{J})$ and $(B,\Hilm[G],\mathcal{I}_{\Hilm[G]})$ be objects of $\Corr^P_\proper$ and $\Corr^P_\Bim$, respectively. There is a groupoid equivalence  
 \[  \Corr^P_\Bim\bigl((\CP^e_{\Hilm[J],\Hilm},(\CP^p_{\Hilm[J],\Hilm})_{p\in P}, \mathcal{I}_{\CP_{\Hilm[J],\Hilm}}), (B,\Hilm[G],\mathcal{I}_{\Hilm[G]})\bigr)
  \cong\Corr^P_\proper\bigl((A,\Hilm, \mathcal{J}), (B,\Hilm[G],\mathcal{I}_{\Hilm[G]})\bigr), \]
 which is defined by composing objects with~$\upsilon_{(A,\Hilm, \mathcal{J})}$.
 \end{prop}
 \begin{proof} Let $(\Hilm[F],V)\colon(A,\Hilm, \mathcal{J})\to (B,\Hilm[G],\mathcal{I}_{\Hilm[G]})$ be a morphism in~$\Corr^P_\proper$. Let~$\CP_{\Hilm[F],V}$ be the correspondence~$\CP_{\Hilm[J],\Hilm}\leadsto\CP_{\mathcal{I}_{\Hilm[G]},\Hilm[G]}$ induced by~$(\Hilm[F],V)$ built in Proposition~\ref{prop:induced_correspondence}. By Proposition~\ref{pro:relative_embedding},  $\CP^e_{\mathcal{I}_{\Hilm[G]},\Hilm[G]}$ is isomorphic to~$B$ and hence $\CP^e_{\Hilm[F],V}=\Hilm[F]\otimes_B\CP^e_{\mathcal{I}_{\Hilm[G]},\Hilm[G]}\cong\Hilm[F]$. Since~$\CP^e_{\Hilm[J],\Hilm}$ acts by $G$\nb-grading-preserving operators on~$\CP_{\Hilm[F],V}$,  we then obtain a nondegenerate \Star homomorphism from~$\CP^e_{\Hilm[J],\Hilm}$ to $\Comp(\Hilm[F])$ by restricting its left action to~$\CP^{e}_{\Hilm[F],V}$. This makes~$\Hilm[F]$ into a proper correspondence~$\CP^e_{\Hilm[J],\Hilm}\leadsto B$, which we denote by $\Hilm[F]^\sharp$. We also have correspondence isomorphisms $$V_p^{\sharp}\colon\CP^p_{\Hilm[J],\Hilm}\otimes_{\CP^e_{\Hilm[J],\Hilm}}\Hilm[F]^\sharp\cong\Hilm[F]^\sharp\otimes_B\Hilm[G]_p$$ obtained from $\CP_{\Hilm[F],V}$ because $\CP^p_{\Hilm[F],V}=\Hilm[F]\otimes_B\CP^p_{\mathcal{I}_{\Hilm[G]},\Hilm[G]}\cong\Hilm[F]\otimes_B\Hilm[G]_p$ for all~$p\in P$ and $$\bar{\psi}(\CP^p_{\Hilm[J],\Hilm})\CP^e_{\Hilm[F],V}=\psi_p(\Hilm_p)\bar{\psi}(\CP^e_{\Hilm[J],\Hilm})\CP^e_{\Hilm[F],V}=\psi_p(\Hilm_p)\CP^e_{\Hilm[F],V}\cong\CP^p_{\Hilm[F],V}.$$ Thus we let $V^\sharp=\{V^\sharp_p\}_{p\in P}$. The coherence axiom \eqref{eq:coherence_covariant_corres} holds because~$V^\sharp$ comes from the \Star homomorphism~$\bar{\psi}\colon\CP_{\Hilm[J],\Hilm}\to\Comp(\CP_{\Hilm[F],V})$.

In order to see that the pair $(\Hilm[F]^\sharp,V^\sharp)$ is a proper covariant correspondence $$(\CP^e_{\Hilm[J],\Hilm},(\CP^p_{\Hilm[J],\Hilm})_{p\in P}, \mathcal{I}_{\CP_{\Hilm[J],\Hilm}})\to(B,\Hilm[G],\mathcal{I}_{\Hilm[G]}),$$ it remains to prove that~$I_p^{\CP_{\Hilm[J],\Hilm}}\Hilm[F]\subseteq\Hilm[F]I_p^{\Hilm[G]}$ for all~$p\in P$. The ideal $I_p^{\CP_{\Hilm[J],\Hilm}}$ is determined by the left inner product, so that $I_p^{\CP_{\Hilm[J],\Hilm}}=\CP^p_{\Hilm[J],\Hilm}\CP^{p\,\,*}_{\mathcal{J},\Hilm}$. Now $\CP^{p\,\,*}_{\mathcal{J},\Hilm}$ sends~$\CP^e_{\Hilm[F],V}$ to $\CP^{p^{-1}}_{\Hilm[F],V}=\Hilm[F]\otimes_B\CP^{p\,\,*}_{{\mathcal{I}_{\Hilm[G]},\Hilm[G]}}\cong\Hilm[F]\otimes_B\Hilm[G]_p^*,$ while~$\CP^p_{\Hilm[J],\Hilm}$ maps $\Hilm[F]\otimes_B\CP^{p\,\,*}_{\Hilm[G]}$ into $\Hilm[F]\otimes_B\CP^{p}_{{\mathcal{I}_{\Hilm[G]},\Hilm[G]}}\CP^{p\,\,*}_{{\mathcal{I}_{\Hilm[G]},\Hilm[G]}}$. The isomorphism $\CP^{p}_{{\mathcal{I}_{\Hilm[G]},\Hilm[G]}}\CP^{p\,\,*}_{{\mathcal{I}_{\Hilm[G]},\Hilm[G]}}\cong\BRAKET{\Hilm[G]_p}{\Hilm[G]_p}$ implies that~$(\Hilm[F]^\sharp,V^\sharp)$ is indeed a proper covariant correspondence. 

Now let $(\Hilm[F],V)\colon(\CP^e_{\Hilm[J],\Hilm},(\CP^p_{\Hilm[J],\Hilm})_{p\in P}, \mathcal{I}_{\CP_{\Hilm[J],\Hilm}})\to (B,\Hilm[G],\mathcal{I}_{\Hilm[G]})$ be a proper covariant correspondence. Composition with~$\upsilon_{(A,\Hilm, \mathcal{J})}$ yields a proper correspondence~$\CP_{\Hilm[J],\Hilm}\otimes_{\CP_{\Hilm[J],\Hilm}}\Hilm[F]\colon A\leadsto B$, which we naturally identify with the correspondence~$\Hilm[F]\colon A\leadsto B$ with left action induced by the \Star homomorphism $j_{\Hilm[J]_A}\colon A\to\CP_{\Hilm[J],\Hilm}$. We denote this correspondence~$A\leadsto B$ by~$\Hilm[F]^\flat$. We let~$V^\flat=\{V_p^\flat\}_{p\in P}$ be the associated family of isomorphisms $$V_p^\flat\colon\Hilm_p\otimes_A\Hilm[F]^\flat\cong\Hilm_p\otimes_A\CP^e_{\Hilm[J],\Hilm}\otimes_{\CP^e_{\Hilm[J],\Hilm}}\Hilm[F]\cong\CP^p_{\Hilm[J],\Hilm}\otimes_{\CP^e_{\Hilm[J],\Hilm}}\Hilm[F]\xRightarrow{V_p}\Hilm[F]^\flat\otimes_B\Hilm[G]_p.$$ This gives a proper covariant correspondence $(\Hilm[F]^\flat,V^\flat)\colon(A,\Hilm, \mathcal{J})\to (B,\Hilm[G],\mathcal{I}_{\Hilm[G]}).$    

The proof that~$(\Hilm[F]^{\sharp\flat},V^{\sharp\flat})=(\Hilm[F],V)$ and that $(\Hilm[F]^{\flat\sharp},V^{\flat\sharp})=(\Hilm[F],V)$ follows as for single correspondences. So we refer the reader to \cite[Proposition~3.4]{Meyer-Sehnem:Bicategorical_Pimsner} for further details. Thus $(\Hilm[F],V)\mapsto(\Hilm[F]^\flat,V^\flat)$, $w\mapsto w^\flat$ is an equivalence of categories. This functor is naturally equivalent to the one defined by composition with~$\upsilon_{(A,\Hilm, \mathcal{J})}$. Such an equivalence has component at $(\Hilm[F],V)$ determined by the canonical correspondence isomorphism $\CP^e_{\Hilm[J],\Hilm}\otimes_{\CP^e_{\Hilm[J],\Hilm}}\Hilm[F]\cong\Hilm[F]^\flat.$ Therefore, composition with~$\upsilon_{(A,\Hilm, \mathcal{J})}$ establishes a groupoid equivalence.
\end{proof}

Let~$\mathcal{B}$ and~$\mathcal{C}$ be bicategories. By \cite[Theorem~9.17]{Fiore:Pseudo_biadjoints}, a functor $R\colon\mathcal{C}\to\mathcal{B}$ has a left adjoint if and only if there exists a universal arrow from~$b$ to~$R$ for every $b\in\obj\mathcal{B}$. As a consequence, we have the following:

\begin{cor}[see \cite{Meyer-Sehnem:Bicategorical_Pimsner}*{Corollary~4.7}]\label{cor:P_adjunction} The sub-bicategory~$\Corr^P_\Bim\subseteq\Corr^P_\proper$ is reflective. That is, the inclusion homomorphism $R\colon\Corr^P_\Bim\hookrightarrow\Corr^P_\proper$ has a left adjoint $L\colon\Corr^P_\proper\to\Corr^P_\Bim$ (reflector). This is defined on objects by 
$$(A,\Hilm,\Hilm[J])\mapsto (\CP^e_{\Hilm[J],\Hilm}, (\CP^p_{\Hilm[J],\Hilm})_{p\in P},\Hilm[I]_{\CP_{\Hilm[J],\Hilm}}).$$
\end{cor}

 The proof of \cite[Theorem~9.17]{Fiore:Pseudo_biadjoints} builds the reflector~$L$: it maps an arrow $(\Hilm[F],V)\colon(A,\Hilm,\Hilm[J])\to(A_1,\Hilm_1,\Hilm[J]_1)$ to $$L(\Hilm[F],V)=(\upsilon_{(A_1,\Hilm_1, \mathcal{J}_1)}\circ(\Hilm[F],V))^\sharp=\big((\Hilm[F]\otimes_{A_1}\CP^e_{\Hilm[J]_1,\Hilm_1})^\sharp,(V\bullet\bar{\iota}_{\Hilm_1})^\sharp\big).$$ It is defined on a $2$\nb-arrow $w\colon(\Hilm[F]_0,V_0)\Rightarrow(\Hilm[F]_1,V_1)$ by $L(w)=(w\otimes1_{\CP_{\Hilm[J]_1,\Hilm_1}})^\sharp.$ 

Let $(\Hilm[F],V)\colon\!(A,\Hilm,\Hilm[J])\to(A_1,\Hilm_1,\Hilm[J]_1)$ and $(\Hilm[F]_1,V_1)\colon\!(A_1,\Hilm_1,\Hilm[J]_1)\to(A_2,\Hilm_2,\Hilm[J]_2)$ be proper covariant correspondences. The isomorphism $$\lambda((\Hilm[F],V),(\Hilm[F]_1,V_1))\colon L(\Hilm[F]_1,V_1)\circ L(\Hilm[F],V)\cong L((\Hilm[F]_1,V_1)\circ(\Hilm[F],V))$$ is built out of the left action of~$\CP^e_{\Hilm[J]_1,\Hilm_1}$ on $(\Hilm[F]_1\otimes_{A_2}\CP^e_{\Hilm[J]_2,\Hilm_2})^\sharp$ constructed in Proposition~\ref{prop:P_universal_arrow}. That is, it is given by the canonical isomorphism $$(\Hilm[F]\otimes_{A_1}\CP^e_{\Hilm[J]_1,\Hilm_1})^\sharp\otimes_{\CP^e_{\Hilm[J]_1,\Hilm_1}}(\Hilm[F]_1\otimes_{A_2}\CP^e_{\Hilm[J]_2,\Hilm_2})^\sharp\cong(\Hilm[F]\otimes_{A_1}\Hilm[F]_1\otimes_{A_2}\CP^e_{\Hilm[J]_2,\Hilm_2})^\sharp.$$ The compatibility isomorphism for units is obtained from the nondegenerate \Star homomorphism~$j_{\Hilm[J]}\colon A\to\CP^e_{\Hilm[J],\Hilm}$.

\subsection{Morita equivalence for relative Cuntz--Pimsner algebras}\label{subsec: Morita}

Let~$\Corr^G$ denote the bicategory whose objects are $\Cst$\nb-algebras carrying a nondegenerate full coaction of~$G$. Arrows are correspondences with a coaction of~$G$ compatible with those on the underlying $\Cst$\nb-algebras and $2$\nb-arrows are correspondence isomorphisms that intertwine the left and right actions of the $\Cst$\nb-algebras. We refer to \cite[Definition 2.10]{Echterhoff-Kaliszewski-Quigg-Raeburn:Categorical} for a precise definition. See also \cite[Theorem 2.15]{Echterhoff-Kaliszewski-Quigg-Raeburn:Categorical} for~$\Corr^G$. Let~$\Corr^G_{\proper}$ be the sub-bicategory of~$\Corr^G$ whose arrows are proper correspondences.
\begin{cor}\label{cor:P_functoriality} The construction of relative Cuntz--Pimsner algebras is functorial. There is a homomorphism of bicategories~$\Corr^P_{\proper}\to\Corr^G_{\proper}$ which is defined on objects by $$(A,\Hilm, \Hilm[J])\mapsto\CP_{\Hilm[J],\Hilm}.$$ This functor is naturally isomorphic to the composite $$\Corr^P_{\proper}\xrightarrow{L}\Corr^P_{\Bim}\rightarrow \Corr^G_{\proper},$$ where the functor on the right-hand side sends a simplifiable product system of Hilbert bimodules to its relative Cuntz--Pimsner algebra for the family of Katsura's ideals.
\end{cor}
\begin{proof} It follows from Proposition~\ref{prop:induced_correspondence} that a proper covariant correspondence $(\Hilm[F],V)\colon(A,\Hilm,\Hilm[J])\to(A_1,\Hilm_1,\Hilm[J]_1)$ yields a proper $\Cst$\nb-correspondence $$\CP_{\Hilm[F],V}\colon\CP_{\Hilm[J],\Hilm}\leadsto\CP_{\Hilm[J]_1,\Hilm_1}.$$ This carries a coaction of~$G$ that is compatible with the coactions on~$\CP_{\Hilm[J],\Hilm}$ and~$\CP_{\Hilm[J]_1,\Hilm_1}$. Repeating the arguments in the proof of Theorem~\ref{thm:induced_equivalence}, we deduce that a $2$\nb-arrow $w\colon (\Hilm[F]_0,V_0)\Rightarrow(\Hilm[F]_1,V_1)$ produces an isomorphism $$w\otimes 1_{\CP_{\Hilm[J]_1,\Hilm_1}}\colon \CP_{\Hilm[F]_0,V_0}\Rightarrow\CP_{\Hilm[F]_1,V_1}$$ that intertwines the left and right actions of $\CP_{\Hilm[J],\Hilm}$ and~$\CP_{\Hilm[J]_1,\Hilm_1}$, respectively. The remaining data for composition of arrows and compatibility of units comes from the canonical isomorphisms $$\CP_{\Hilm[J],\Hilm}\otimes_{\CP_{\Hilm[J],\Hilm}} \CP_{\Hilm[F],V}\cong  \CP_{\Hilm[F],V},\quad A\otimes_A \CP_{\Hilm[J],\Hilm}\cong \CP_{\Hilm[J],\Hilm}.$$ So we obtain a homomorphism~$\Corr^P_\proper\to \Corr^G_{\proper}$ that sends $(A,\Hilm,\Hilm[J])$ to $\CP_{\Hilm[J],\Hilm}$, $(\Hilm[F],V)$ to $\CP_{\Hilm[F],V}$ and maps a $2$\nb-morphism~$w$ to $w\otimes 1$. 

The above functor restricts to a homomorphism~$\Corr^P_{\Bim}\rightarrow \Corr^G_{\proper}$ that sends a simplifiable product system of Hilbert bimodules to its relative Cuntz--Pimsner algebra for the family of Katsura's ideals. We know that $\CP_{\Hilm[J],\Hilm}$ is isomorphic to the cross-sectional $\Cst$\nb-algebra of the Fell bundle~$(\CP^g_{\Hilm[J],\Hilm})_{g\in G}$ associated to the coaction of~$G$. This establishes a canonical isomorphism $$\CP_{\Hilm[I]_{\CP_{\Hilm[J],\Hilm}},\CP_{\Hilm[J],\Hilm}}\cong\CP_{\Hilm[J],\Hilm}$$  by Proposition~\ref{prop:associated_product_system} because $(\CP^g_{\Hilm[J],\Hilm})_{g\in G}$ is extended from~$P$. Here $\CP_{\Hilm[I]_{\CP_{\Hilm[J],\Hilm}},\CP_{\Hilm[J],\Hilm}}$ is the relative Cuntz--Pimsner algebra for Katsura's ideals of~$(\CP^p_{\Hilm[J],\Hilm})_{p\in P}$. This produces a natural isomorphism between the homomorphism $\Corr^P_{\proper}\rightarrow \Corr^G_{\proper}$ built above and the composite of the reflector~$L\colon \Corr^P_{\proper}\rightarrow\Corr^P_{\Bim}$ from Corollary~\ref{cor:P_adjunction} with the homomorphism~$\Corr^P_{\Bim}\to\Corr^G_{\proper}$ obtained by restriction.
\end{proof}

\begin{cor}\label{cor:Morita_equivalence} Let~$(A,\Hilm,\Hilm[J])$ and~$(B,\Hilm[G],\Hilm[J]_B)$ be objects of~$\Corr^P_{\proper}$. Then~$\CP_{\Hilm[J],\Hilm}$ and~$\CP_{\Hilm[J]_{\!B},\Hilm[G]}$ are Morita equivalent if there is a covariant correspondence $$(\Hilm[F],V)\colon(A,\Hilm,\Hilm[J])\to(B,\Hilm[G],\Hilm[J]_B)$$ so that $J_p^A\Hilm[F]=\Hilm[F]J_p^B$ for all~$p\in P$ and $\Hilm[F]\colon A\leadsto B$ establishes a Morita equivalence. For objects in~$\Corr^P_{\Bim}$, this equivalence preserves amenability of Fell bundles.
\end{cor}
\begin{proof} First, notice that $\Hilm[F]$ is automatically a proper correspondence. By \cite[Lemma~2.4]{Echterhoff-Kaliszewski-Quigg-Raeburn:Categorical}, a correspondence~$\Hilm[F]\colon A\leadsto B$ is an imprimitivity $A,B$\nb-bimodule if there exists a correspondence~$\widetilde{\Hilm[F]}\colon A\leadsto B$ with correspondence isomorphisms $$\Hilm[F]\otimes_B\widetilde{\Hilm[F]}\cong A,\quad \widetilde{\Hilm[F]}\otimes_A\Hilm[F]\cong B.$$ So by functoriality for relative Cuntz--Pimsner algebras established in Corollary~\ref{cor:P_functoriality}, it suffices to show that~$\Hilm[F]$ is an invertible arrow in~$\Corr^P_{\proper}$. That is, there is a proper covariant correspondence $(\Hilm[F]^*,\widetilde{V})\colon(B,\Hilm[G],\Hilm[J]_B)\to(A,\Hilm,\Hilm[J])$ with (invertible) $2$\nb-arrows $$w\colon(\Hilm[F]\otimes_B\Hilm[F]^*,V\bullet\widetilde{V})\Rightarrow(A,\iota_{\Hilm}),\quad\widetilde{w}\colon(\Hilm[F]^*\otimes_A\Hilm[F],\widetilde{V}\bullet V)\Rightarrow(B,\iota_{\Hilm[G]}).$$

Let $\Hilm[F]^*$ be the Hilbert $B,A$\nb-bimodule adjoint to $\Hilm[F]$. For each~$p\in P$, we use the identifications $\Hilm[F]\otimes_B\Hilm[F]^*\cong A$ and $\Hilm[F]^*\otimes_A\Hilm[F]\cong B$ to define a correspondence isomorphism $\widetilde{V}_p\colon\Hilm[G]_p\otimes_B\Hilm[F]^*\cong\Hilm[F]^*\otimes_A\Hilm_p$ as the composite
\begin{align*}\Hilm[G]_p\otimes_B\Hilm[F]^*&\cong\Hilm[F]^*\otimes_A\Hilm[F]\otimes_B\Hilm[G]_p\otimes_B\Hilm[F]^*&(1_{\Hilm[F]^*}\otimes V_p^{-1}\otimes1_{\Hilm[F]^*})\\&\cong\Hilm[F]^*\otimes_A\Hilm_p\otimes_A\Hilm[F]\otimes_B\Hilm[F]^*&\\ &\cong\Hilm[F]^*\otimes_A\Hilm_p.&\end{align*} We set $\widetilde{V}=\{\widetilde{V}_p\}_{p\in P}$. Observe that~$J_p^A\Hilm[F]=\Hilm[F]J_p^B$ implies $J_p^B\Hilm[F]^*=\Hilm[F]^*J_p^A$. So in order to conclude that~$(\Hilm[F]^*,\widetilde{V})$ is a covariant correspondence from $(B,\Hilm[G],\Hilm[J]_B)$ to $(A,\Hilm,\Hilm[J])$, all we need to prove is that it satisfies the coherence axiom~\eqref{eq:coherence_covariant_corres}. To do so, let~$p,q\in P$. Tensoring the coherence diagram \eqref{eq:coherence_covariant_corres} for~$(\Hilm[F],V)$ with $\Hilm[F]^*$ on the left and on the right, we obtain the following commutative diagram:
   \begin{equation}\label{eq:coherence_for_inverse}\begin{gathered}
    \xymatrix{
    \Hilm[F]^*\otimes_A\Hilm_p\otimes_A\Hilm_q\otimes_A\Hilm[F]\otimes_B\Hilm[F]^* \ar@{->}[r]^{} \ar@{<-}[d]_{1\otimes V_q^{-1}\otimes1_{\Hilm[F]^*}}&   \Hilm[F]^*\otimes_A\Hilm_{pq}\otimes_A\Hilm[F]\otimes_B\Hilm[F]^* 
\ar@{}[d]^{} \\
    \Hilm[F]^*\otimes_A\Hilm_p\otimes_A\Hilm[F]\otimes_B\Hilm[G]_q\otimes_B\Hilm[F]^* \ar@{<-}[d]_{1_{\Hilm[F]^*}\otimes V_p^{-1}\otimes1} &  \\
    \Hilm[F]^*\otimes_A\Hilm[F]\otimes_B\Hilm[G]_p\otimes_B\Hilm[G]_q\otimes_B\Hilm[F]^* \ar@{->}[r]^{} & \Hilm[F]^*\otimes_A\Hilm[F]\otimes_B\Hilm[G]_{pq}\otimes_B\Hilm[F]^*  \ar@{->}[uu]_{1_{\Hilm[F]^*}\otimes V_{pq}^{-1}\otimes1_{\Hilm[F]^*}}.
    }\end{gathered}
\end{equation}

The following diagram also commutes:
  \begin{equation}\label{eq:coherence_for_left-side}\begin{gathered}
    \xymatrix{
    \Hilm[F]\otimes_B\Hilm[G]_p\otimes_B\Hilm[F]^*\otimes_A\Hilm[F]\otimes_B\Hilm[G]_q  \ar@{->}[r]^{1\otimes V_q^{-1}} \ar@{->}[d]_{ V_p^{-1}\otimes1}&
  \Hilm[F]\otimes_B\Hilm[G]_p\otimes_B\Hilm[F]^*\otimes_A\Hilm_q\otimes_A\Hilm[F]\ar@{->}[d]^{V_p^{-1}\otimes1} \\ 
\Hilm_p\otimes_A\Hilm[F]\otimes_B\Hilm[F]^*\otimes_A\Hilm[F]\otimes_B\Hilm[G]_q  \ar@{->}[r]^{1\otimes V_q^{-1}} &
\Hilm_p\otimes_A\Hilm[F]\otimes_B\Hilm[F]^*\otimes_A\Hilm_q\otimes_A\Hilm[F].
    }
   \end{gathered}
\end{equation}
  
Since all the maps involved in the above diagrams are $A,B$\nb-bimodule maps, commutativity of \eqref{eq:coherence_for_left-side} implies that the top-left composite of \eqref{eq:coherence_for_inverse} is precisely the bottom-left composite of \begin{align*}\xymatrix@C=50pt{\Hilm[G]_p\otimes_B\Hilm[G]_q\otimes_B\Hilm[F]^* \ar@{->}[d]_{1\otimes \widetilde{V}_q} \ar@{->}[r]^{\mu^{2}_{p,q}\otimes1}&  \Hilm[G]_{pq}\otimes_B\Hilm[F]^*\ar@{->}[r]^{\widetilde{V}_{pq}}&\Hilm[F]^*\otimes_A\Hilm_{pq} \ar@{=}[d]^{} \\ \Hilm[G]_p\otimes_B\Hilm[F]^*\otimes_A\Hilm_q  \ar@{->}[r]^{ \widetilde{V}_p\otimes1} &  \Hilm[F]^* \otimes_A\Hilm_p\otimes_A\Hilm_q  \ar@{->}[r]^{1\otimes \mu^1_{p,q}} &\Hilm[F]^* \otimes_A\Hilm_{pq}}\end{align*}up to the usual identifications. Hence the above diagram commutes and so~$(\Hilm[F]^*,\bar{V})$ is a proper covariant correspondence as desired. The canonical isomorphisms $w\colon\Hilm[F]\otimes_B\Hilm[F]^*\cong A$ and $\widetilde{w}\colon\Hilm[F]^*\otimes_A\Hilm[F]\cong B$ are the required $2$\nb-arrows.

  If~$(\Hilm[F], V)\colon(A,\Hilm,\Hilm[I]_{\Hilm})\to(B,\Hilm[G],\Hilm[I]_{\Hilm[G]})$ is an equivalence in~$\Corr^P_\Bim$ and~$(\hat{\Hilm[G]}_g)_{g\in G}$ is amenable, then $(\hat{\Hilm}_g)_{g\in G}$ is also amenable. Indeed, by functoriality, \allowbreak$\CP_{\Hilm[F], V}\colon \CP_{\Hilm[I]_{\Hilm},\Hilm}\allowbreak\leadsto  \CP_{\Hilm[I]_{\Hilm[G]},\Hilm[G]}$ is an imprimitivity bimodule. In particular, $\CP_{\Hilm[I]_{\Hilm},\Hilm}\cong\Cst\big((\hat{\Hilm}_g)_{g\in G}\big)$ acts faithfully on~$\CP_{\Hilm[F], V}$. Since~$\Cst\big((\hat{\Hilm[G]}_g)_{g\in G}\big)\cong\Cst_r\big((\hat{\Hilm[G]}_g)_{g\in G}\big)$ through the regular representation, we can turn $\Cst\big((\hat{\Hilm[G]}_g)_{g\in G}\big)$ into a faithful correspondence $$\Cst\big((\hat{\Hilm[G]}_g)_{g\in G}\big)_{E}\colon\Cst\big((\hat{\Hilm[G]}_g)_{g\in G}\big)\leadsto B$$ using the (faithful) conditional expectation $E\colon\Cst\big((\hat{\Hilm[G]}_g)_{g\in G}\big)\to B$ to define the right $B$\nb-valued inner product (see \cite[Proposition~19.7]{Exel:Partial_dynamical}). Now composing with $\Hilm[F]^*$ we obtain a faithful correspondence from $\Cst\big((\hat{\Hilm}_g)_{g\in G}\big)$ to~$A$ which is given by $$\CP_{\Hilm[F], V}\otimes_{\CP_{\Hilm[I]_{\Hilm[G]},\Hilm[G]}}\Cst\big((\hat{\Hilm[G]}_g)_{g\in G}\big)_E\otimes_B\Hilm[F]^* \colon \Cst\big((\hat{\Hilm}_g)_{g\in G}\big)\leadsto A.$$ We use the canonical correspondence isomorphism $$ \CP_{\Hilm[I]_{\Hilm[G]},\Hilm[G]}\otimes_{ \CP_{\Hilm[I]_{\Hilm[G]},\Hilm[G]}}\Cst\big((\hat{\Hilm[G]}_g)_{g\in G}\big)_E\cong \Cst\big((\hat{\Hilm[G]}_g)_{g\in G}\big)_E$$ coming from the nondegenerate left action of $ \CP_{\Hilm[I]_{\Hilm[G]},\Hilm[G]}$ on $\Cst\big((\hat{\Hilm[G]}_g)_{g\in G}\big)_E$ to identify $\CP_{\Hilm[F], V}\otimes_{\CP_{\Hilm[I]_{\Hilm[G]},\Hilm[G]}}\Cst\big((\hat{\Hilm[G]}_g)_{g\in G}\big)_E$ with the direct sum of correspondences $\bigoplus_{\substack{g\in G}}(\Hilm[F]\otimes_B \hat{\Hilm[G]}_g)$. This is isomorphic to $(\bigoplus_{\substack{g\in G}}\hat{\Hilm}_g)\otimes_A\Hilm[F]$ because the $V_p$'s are unitary and $\widetilde{V}_p\colon \Hilm[G]_p\otimes_B\Hilm[F]^*\cong\Hilm[F]^*\otimes_A\Hilm_p$ induces an isomorphism $\Hilm_p^*\otimes_A\Hilm[F]\cong \Hilm[F]\otimes_B\Hilm[G]_p^*$ for each $p\in P$. Thus the isomorphism $\Hilm[F]\otimes_A\Hilm[F]^*\cong A$ provides the direct sum $\bigoplus_{\substack{g\in G}}\hat{\Hilm}_g$ with a structure of faithful correspondence $\Cst\big((\hat{\Hilm}_g)_{g\in G}\big)\leadsto A$ which coincides with the structure coming from the regular representation of $\hat{\Hilm}=(\hat{\Hilm}_g)_{g\in G}$ in $\Bound\big(\bigoplus_{\substack{g\in G}}\hat{\Hilm}_g\big)$. Therefore, $\hat{\Hilm}$ must be amenable.
\end{proof}

\begin{rem} That an equivalence between objects in~$\Corr^P_{\Bim}$ preserves amenability also follows from~\cite{Abadie-Ferraro} and Theorem~\ref{thm:induced_equivalence}.
\end{rem}

\begin{example} Let~$A$ and~$B$ be $\Cst$\nb-algebras and let $\Hilm[F]\colon A\to B$ be an imprimitivity $A,B$\nb-bimodule. A compactly aligned product system $\Hilm=(\Hilm_p)_{p\in P}$ over $A$ induces a compactly aligned product system $\Hilm[G]=(\Hilm[G]_p)_{p\in P}$ over~$B$ as follows. We set $\Hilm[G]_p\coloneqq\Hilm[F]^*\otimes_A\Hilm_p\otimes_A\Hilm[F]$. The multiplication map~$\widetilde{\mu}_{p,q}\colon\Hilm[G]_p\otimes_A\Hilm[G]_q\cong\Hilm[G]_{pq}$ is defined using the isomorphism~$\Hilm[F]\otimes_A\Hilm[F]^*\cong A$. More explicitly, it is given by \begin{align*}\Hilm[G]_p\otimes_B\Hilm[G]_q& =\Hilm[F]^*\otimes_A\Hilm_p\otimes_A\Hilm[F]\otimes_B\Hilm[F]^*\otimes_A\Hilm_q\otimes_A\Hilm[F] &\\&\cong\Hilm[F]^*\otimes_A\Hilm_p\otimes_A\Hilm_q\otimes_A\Hilm[F]&(1_{\Hilm[F]^*}\otimes\mu_{p,q}\otimes1_{\Hilm[F]})\\&\cong\Hilm[F]^*\otimes_A\Hilm_{pq}\otimes_A\Hilm[F]=\Hilm[G]_{pq}.&\end{align*} The multiplication maps $\{\widetilde{\mu}_{p,q}\}_{p,q\in P}$ satisfy the coherence axiom required for product systems because~$\{\mu_{p,q}\}_{p,q\in P}$ do so.

We claim that~$\Hilm[G]$ is compactly aligned. Indeed, let~$p,q\in P$ with $p\vee q<\infty$. Notice that $\Comp(\Hilm[G]_p)$ is canonically isomorphic to~$\Hilm[F]^*\otimes_A\Comp(\Hilm_p)\otimes_A\Hilm[F]$ through the $B$\nb-bimodule map
\begin{align*}\Hilm[F]^*\otimes_A\Hilm_p\otimes_A\Hilm[F]\otimes_B(\Hilm[F]^*\otimes_A\Hilm_p\otimes_A\Hilm[F])^*&\cong\Hilm[F]^*\otimes_A\Hilm_p\otimes_A\Hilm[F]\otimes_B\Hilm[F]^*\otimes_A\Hilm_p^*\otimes_A\Hilm[F]\\&\cong \Hilm[F]^*\otimes_A\Hilm_p\otimes_A\Hilm_p^*\otimes_A\Hilm[F]\\ &\cong\Hilm[F]^*\otimes_A\Comp(\Hilm_p)\otimes_A\Hilm[F].\end{align*} So take $T\in\Comp(\Hilm_p)$ and $S\in\Comp(\Hilm_q)$. Let $\zeta_1,\zeta_2,\eta_1,\eta_2\in\Hilm[F]$ and let $\eta^*\otimes\xi\otimes\zeta$ be an elementary tensor of $\Hilm[F]^*\otimes_A\Hilm_{p\vee q}\otimes_A\Hilm[F]$. We have that \begin{align*}\iota_{q}^{p\vee q}(\eta_2^*\otimes S\otimes\zeta_2)(\eta^*\otimes\xi\otimes\zeta)&=\eta_2^*\otimes\iota_q^{p\vee q}(S)\big(\varphi_{p\vee q}(\BRAKET{\zeta_2}{\eta})(\xi)\big)\otimes\zeta.
\end{align*} Applying~$\iota_{p}^{p\vee q}(\eta_1^*\otimes T\otimes\zeta_1)$ to both sides of the above equality, we deduce that \begin{align*}\iota_{p}^{p\vee q}(\eta_1^*\otimes T\otimes\zeta_1)&\iota_{q}^{p\vee q}(\eta_2^*\otimes S\otimes\zeta_2)(\eta^*\otimes\xi\otimes\zeta)\\=&\,\,\eta_1^*\!\otimes\iota_{p}^{p\vee q}(T)\big(\!\varphi_{p\vee q}(\BRAKET{\zeta_1}{\eta_2}\!)\iota_q^{p\vee q}(S)\big(\!\varphi_{p\vee q}(\BRAKET{\zeta_2}{\eta})(\xi)\big)\!\big)\!\otimes\zeta.\end{align*} Define $T'\in\Comp(\Hilm_{p\vee q})$ by $T'=\iota_{p}^{p\vee q}(T)\varphi_{p\vee q}(\BRAKET{\zeta_1}{\eta_2})\iota_q^{p\vee q}(S)$. Then $$\eta_1^*\otimes T'\big(\varphi_{p\vee q}(\BRAKET{\zeta_2}{\eta})(\xi)\big)\otimes\zeta=(\eta_1^*\otimes T'\otimes\zeta_2)(\eta^*\otimes\xi\otimes\zeta).$$ So~$\Hilm[G]$ is also compactly aligned, as claimed.

Given $p\in P$, an element $b\in B$ is compact on~$\Hilm[G]_p$ if and only if $b\Hilm[F]^*\subseteq\Hilm[F]^*\varphi^{-1}_p(\Comp(\Hilm_p))$, provided~$\Hilm[F]^*$ is an equivalence (see also~\cite[Corollary~3.7]{Pimsner:Generalizing_Cuntz-Krieger}). The bijection between the lattices of ideals of~$A$ and~$B$, respectively, obtained from the Rieffel correspondence, yields a one-to-one correspondence between ideals in~$A$ acting by compact operators on~$\Hilm_p$ and ideals in~$B$ mapped to compact operators on~$\Hilm[G]_p$. Precisely, this sends $J^A_p\idealin \varphi_p^{-1}(\Comp(\Hilm_p))$ to $J_p^B=\braket{J_p^A\Hilm[F]}{\Hilm[F]}.$ Its inverse maps an ideal $J_p^B\idealin\widetilde{\varphi}_p^{-1}(\Comp(\Hilm[G]_p))$ to $J_p^A=\BRAKET{\Hilm[F]J_p^B}{\Hilm[F]}$.

The equivalence $\Hilm[F]$ may be turned into a proper covariant correspondence $(\Hilm[F],V)\colon(A,\Hilm, \Hilm[J]_A)\to(B,\Hilm[G],\Hilm[J]_B)$, where~$V=\{V_p\}_{p\in P}$ and~$V_p\colon \Hilm_p\otimes_A\Hilm[F]\cong\Hilm[F]\otimes_B\Hilm[G]_p$ arises from the canonical isomorphism $$\Hilm_p\otimes_A\Hilm[F]\cong\Hilm[F]\otimes_B\Hilm[F]^*\otimes_A\Hilm_p\otimes_A\Hilm[F]=\Hilm[F]\otimes_B\Hilm[G]_p.$$ Here~$\Hilm[J]_A$ and~$\Hilm[J]_B$ are related by the bijection described above. 

It follows from Corollary~\ref{cor:Morita_equivalence} that $(\Hilm[F],V)$ is invertible in $\Corr^P_\proper$ and produces a Morita equivalence between~$\CP_{\Hilm[J]_A,\Hilm}$ and~$\CP_{\Hilm[J]_B,\Hilm[G]}$. Therefore, up to equivariant Morita equivalence, the relative Cuntz--Pimsner algebras associated to~$\Hilm$ correspond bijectively to those associated to~$\Hilm[G]$. In particular, if~$\Hilm$ is a simplifiable product system of Hilbert bimodules, the cross-sectional $\Cst$\nb-algebra of the Fell bundle associated to~$\Hilm$ is Morita equivalent to that of~$\Hilm[G]$. This is so because the family of Katsura's ideals~$\Hilm[I]_{\Hilm}$ corresponds to~$\Hilm[I]_{\Hilm[G]}$ under the Rieffel correspondence. 
\end{example}

The next proposition characterises equivalences between product systems built out of semigroups of injective endomorphisms with hereditary range as in Example~\ref{ex:semigroup_of_endomorphisms}. This generalises~\cite[Proposition 2.4]{Muhly-Solel:Morita_equivalence_of_tensor_algebras}. The idea of the proof is also taken from there, but notice that we do not require that the actions be by automorphisms. We will see below that our notion of equivalence between product systems associated to these semigroup actions is an analogue of Morita equivalence for actions of groups as studied in \cite{10.1112/plms/s3-49.2.289, CMW}.

\begin{prop}\label{prop:charact_extendible} Let $\alpha\colon P\to\mathrm{End}(A)$ and $\beta\colon P\to\mathrm{End}(B)$ be actions by extendible injective endomorphisms with hereditary range. Let~$\tensor*[_\alpha]{A}{}$ and~$\tensor*[_\beta]{B}{}$ be the associated product systems of Hilbert bimodules over~$P^{\mathrm{op}}$. There is an equivalence~$(\Hilm[F],V)\colon(A,\tensor*[_\alpha]{A}{},\mathcal{I}_{\tensor*[_\alpha]{A}{}})\to(B, \tensor*[_\beta]{B}{},\mathcal{I}_{\tensor*[_\beta]{B}{}})$ if and only if there are an imprimitivity $A,B$\nb-bimodule~$\Hilm[F]$ and a semigroup homomorphism~$p\mapsto U_p$ from~$P$ to the semigroup of $\CC$\nb-linear isometries on~$\Hilm[F]$ such that, for all~$p\in P$ and $\xi,\eta\in \Hilm[F],$ \begin{equation}\label{eq:scaling_inner_product}\begin{gathered}\BRAKET{U_p(\xi)}{U_p(\eta)}=\alpha_p(\BRAKET{\xi}{\eta}),\qquad\braket{U_p(\xi)}{U_p(\eta)}=\beta_p(\braket{\xi}{\eta}).\end{gathered}\end{equation}
\end{prop}
\begin{proof} Let $(F,V)\colon(A,\tensor*[_\alpha]{A}{},\mathcal{I}_{\tensor*[_\alpha]{A}{}})\to(B, \tensor*[_\beta]{B}{},\mathcal{I}_{\tensor*[_\beta]{B}{}})$ be an equivalence. Then~$\Hilm[F]$ is an imprimitivity $A,B$\nb-bimodule. Observing that, for all $p\in P$, $$\BRAKET{\tensor*[_{\beta_p}]{\!B}{}}{\tensor*[_{\beta_p}]{\!B}{}}=\beta_p^{-1}(\beta_p(B))=B,$$ we define a correspondence isomorphism~$U'_p\colon\Hilm[F]\to\tensor*[_{\alpha_p}]{\!A}{}\otimes_A\Hilm[F]\otimes_B\tensor*[_{\beta_p}]{\!B}{}^*$ by 
$$\Hilm[F]\cong\Hilm[F]\otimes_B\tensor*[_{\beta_p}]{\!B}{}\otimes_B\tensor*[_{\beta_p}]{\!B}{}^*\xRightarrow{V^{-1}_p\otimes1}\tensor*[_{\alpha_p}]{\!A}{}\otimes_A\Hilm[F]\otimes_B\tensor*[_{\beta_p}]{\!B}{}^*.$$ We identify $\tensor*[_{\beta_p}]{\!B}{}^*$ with $B_{\beta_p}=B\beta_p(1)$ via $\widetilde{\beta_p(1)b}\mapsto b^*\beta_p(1)$ to obtain a linear map $$\tensor*[_{\alpha_p}]{\!A}{}\otimes_A\Hilm[F]\otimes_B\tensor*[_{\beta_p}]{\!B}{}^*\to\Hilm[F]$$ defined on an elementary tensor $\alpha_p(1)a\otimes_A\xi\otimes_Bb\beta_p(1)$ by 
$(\alpha_p(1)a)\xi (b\beta_p(1)).$ This is isometric because~$\beta_p^{-1}$ is an injective \Star homomorphism between $\Cst$\nb-alge\-bras. Its composition with~$U_p'$ yields a linear map $\Hilm[F]\to\Hilm[F]$, which we denote by~$U_p$. Given $\xi,\eta\in\Hilm[F]$, we have that $\braket{\xi}{\eta}=\braket{U_p'(\xi)}{U_p'(\eta)},$ that is, $U_p'$ preserves inner products. From this we deduce $$\braket{U_p(\xi)}{U_p(\eta)}=\beta_p\big(\braket{U_p'(\xi)}{U_p'(\eta)}\big)=\beta_p(\braket{\xi}{\eta}).$$ Similarly, $\BRAKET{U_p'(\xi)}{U_p'(\eta)}\!=\!\BRAKET{\xi}{\eta}$ and we see that $\BRAKET{U_p(\xi)}{U_p(\eta)}\!=\!\alpha_p(\!\BRAKET{\xi}{\eta}\!)$.

It remains to verify that $p\mapsto U_p$ is a semigroup homomorphism from $P^{}$ to the semigroup of $\CC$\nb-linear isometries on~$\Hilm[F]$. First, let $\alpha_q(1)a\in\tensor*[_{\alpha_q}]{\!A}{}$ and notice that, given an elementary tensor $\xi\otimes\alpha_p(1)b$ of~$\Hilm[F]\otimes_B\tensor*[_{\beta_p}]{\!B}{}$, one has $$V_p^{-1}\big(\alpha_q(1)a\xi\otimes\beta_p(1)b\big)=\alpha_q(1)aV_p^{-1}\big(\xi\otimes\beta_p(1)b\big).$$ Since the left action of~$A$ on~$\tensor*[_{\alpha_p}]{\!A}{}$ is implemented by~$\alpha_p$,  it follows that the image of $V_p^{-1}\big(\alpha_q(1)a\xi\otimes\beta_p(1)b\big)$ in $\Hilm[F]$ under the map~$\tensor*[_{\alpha_p}]{\!A}{}\otimes_A\Hilm[F]\to\Hilm[F]$ determined by the left action of~$A$ on~$\Hilm[F]$ coincides with the image of $$(\mu^{\alpha}_{q,p}\otimes1)\big(\alpha_q(1)a\otimes_AV_p^{-1}(\xi\otimes\beta_p(1)b)\big)$$ under the corresponding map~$\tensor*[_{\alpha_{pq}}]{\!A}{}\otimes_A\Hilm[F]\to\Hilm[F]$. Here $\mu^{\alpha}_{q,p}$ is the correspondence isomorphism $\tensor*[_{\alpha_q}]{\!A}{}\otimes_A\tensor*[_{\alpha_p}]{\!A}{}\cong\tensor*[_{\alpha_{pq}}]{\!A}{}.$

Now let $p,q\in P$ and let $(u_\lambda)_{\lambda\in\Lambda}$ be an approximate identity for~$B$. Fix $\lambda\in\Lambda$ and let $\xi\in\Hilm[F]$ and $b\in B$. Then $$U'_q(\xi u_\lambda u_\lambda b)=V^{-1}_q(\xi\otimes\beta_q(u_\lambda))\otimes \beta_q(u_\lambda b).$$ From the above observation and from the fact that $V_p^{-1}$ and $V_q^{-1}$ intertwine the right actions of~$B$, we conclude that \begin{equation*}\begin{aligned}U_p'U_q(\xi u_\lambda u_\lambda b) &=(\mu^{\alpha}_{q,p}\!\otimes\!1)(1\!\otimes\! V_p^{-1}\!\otimes\!1_{B_{\beta_p}})\big(V^{-1}_q(\xi\!\otimes\!\beta_q(u_\lambda))\!\otimes\! \beta_{pq}(u_\lambda)\!\otimes\!\beta_{pq}( b)\big).\end{aligned}\end{equation*} Combining this with the coherence condition~\eqref{eq:coherence_covariant_corres} we may replace the right-hand side of the above equality by \begin{align*}(V^{-1}_{pq}\otimes1_{B_{\beta_p}})(1\otimes\mu_{q,p}^{\beta}\otimes1)\big(\xi\otimes(\beta_q(u_\lambda)\otimes \beta_{pq}(u_\lambda))\otimes\beta_{pq}( b)\big)=&\\=(V^{-1}_{pq}\otimes1_{B_{\beta_p}})(\xi\otimes\beta_{pq}(u_\lambda u_{\lambda}))\otimes\beta_{pq}(b).\end{align*}

This implies $U_pU_q(\xi u_\lambda u_\lambda b)=U_{pq}(\xi u_\lambda u_\lambda b)$. Using that all the~$U_p$'s are continuous and $$\xi b=\underset{\lambda}{\lim}(\xi u_\lambda u_\lambda b),$$ we obtain~$U_pU_q(\xi  b)=U_{pq}(\xi  b)$. This shows that $p\mapsto U_p$ is a semigroup homomorphism, as asserted.

Conversely, suppose that we are given an imprimitivity~$A,B$\nb-bimodule~$\Hilm[F]$ and a semigroup homomorphism~$p\mapsto U_p$ from~$P$ to the semigroup of $\CC$\nb-linear isometries on~$\Hilm[F]$ satisfying \eqref{eq:scaling_inner_product}. For each $p\in P$, $\xi\in \Hilm[F]$ and~$b\in B$, we have that~$U_p(\xi b)=U_p(\xi)\beta_p(b)$ because \begin{align*}\braket{U_{\!p}(\xi b)\!-\!U_{\!p}(\xi)\beta_p(b)\!}{\!U_{\!p}(\xi b)\!-\!U_{\!p}(\xi)\beta_p(b)}&\!=\!\braket{U_{\!p}(\!\xi b)}{U_{\!p}(\!\xi b)}\!-\!\braket{U_{\!p}(\!\xi b)}{U_{\!p}(\!\xi)\beta_p(b)}\\-\braket{U_{\!p}(\!\xi)\beta_p(b)}{U_{\!p}(\!\xi b)}+&\braket{U_{\!p}(\!\xi)\beta_p(b)}{U_{\!p}(\!\xi)\beta_p(b)}\\ &=\beta_p(\braket{\xi b}{\xi b})-\beta_p(\braket{\xi b}{\xi})\beta_p(b)\\-\beta_p(b)^*\beta_p(\braket{\xi}{\xi b})&+\beta_p(b^*)\beta_p(\braket{\xi}{\xi})\beta_p(b)=0.
\end{align*} The same reasoning shows that~$U_p(a\xi)=\alpha_p(a)U_p(\xi)$ for all~$a\in A$.

We then define a map~$V_p'\colon\tensor*[_{\alpha_p}]{\!A}{}^*\otimes_A\Hilm[F]\otimes_B \tensor*[_{\beta_p}]{\!B}{}\to\Hilm[F]$ on elementary tensors by $$a\alpha_p(1)\otimes\xi\otimes_B\beta_p(1)b\mapsto aU_p(\xi)b.$$ In order to verify that this preserves the $B$\nb-valued inner product, let $a,c\in A$, $b,d\in B$ and $\xi,\eta\in \Hilm[F]$. Let $(u_{\lambda})_{\lambda\in\Lambda}$ be an approximate identity for~$A$ and fix~$\lambda\in \Lambda$. Then
\begin{align*}\braket{aU_p(u_{\lambda}\xi)b}{cU_p(u_{\lambda}\eta)d}&=b^*\braket{a\alpha_p(u_{\lambda})U_p(\xi)}{c\alpha_p(u_{\lambda})U_p(\eta)}d\\&=b^*\braket{U_p(\xi)}{\alpha_p(u_{\lambda})a^*c\alpha_p(u_{\lambda})U_p(\eta)}d\\&=b^*\braket{U_p(\xi)}{U_p\big(\alpha_p^{-1}(\alpha_p(u_{\lambda})a^*c\alpha_p(u_{\lambda}))\eta\big)}d\\&=b^*\beta_p\big(\braket{\xi}{\alpha_p^{-1}(\alpha_p(u_{\lambda})a^*c\alpha_p(u_{\lambda}))\eta}\big)d\\&=\braket{a\alpha_p(u_{\lambda})\otimes\xi\otimes\beta_p(1)b}{c\alpha_p(u_{\lambda})\otimes\eta\otimes\beta_p(1)d}.\end{align*}Using that $U_p$ is continuous and $\xi=\lim_{\substack{\lambda}}u_{\lambda}\xi$, $\eta=\lim_{\substack{\lambda}}u_{\lambda}\eta$, we conclude that~$V_p'$ preserves the inner product. In addition, it intertwines the left and right actions of~$A$ and~$B$. 

Now we let~$\widetilde{V}_p\colon\Hilm[F]\otimes_B\tensor*[_{\beta_p}]{\!B}{}\Rightarrow\tensor*[_{\alpha_p}]{\!A}{}\otimes_A\Hilm[F]$ be the composite $$\Hilm[F]\otimes_B\tensor*[_{\beta_p}]{\!B}{}\cong \tensor*[_{\alpha_p}]{\!A}{}\otimes_A\tensor*[_{\alpha_p}]{\!A}{}^*\otimes_A\Hilm[F]\otimes_B \tensor*[_{\beta_p}]{\!B}{}\xRightarrow{1_{}\otimes V_p'}\tensor*[_{\alpha_p}]{\!A}{}\otimes_A\Hilm[F],$$ where the isomorphism on the left-hand side comes from the identification $$\tensor*[_{\alpha_p}]{\!A}{}\otimes_A\tensor*[_{\alpha_p}]{\!A}{}^*\cong \BRAKET{ \tensor*[_{\alpha_p}]{\!A}{}}{ \tensor*[_{\alpha_p}]{\!A}{}}=A.$$ Then~$\widetilde{V}_p$ is an isometry between correspondences $A\leadsto B$. To see that it is indeed unitary, we need to prove that it is also surjective. 

First, observe that $$\alpha_p(\BRAKET{\xi}{\eta})\zeta=\BRAKET{U_p(\xi)}{U_p(\eta)}\zeta=U_p(\xi)\braket{U_p(\eta)}{\zeta}.$$ This implies $\alpha_p(A)\Hilm[F]=U_p(\Hilm[F])\braket{U_p(\Hilm[F])}{\Hilm[F]}$, provided~$\BRAKET{\Hilm[F]}{\Hilm[F]}=A$. Again we let~$(u_{\lambda})_{\lambda\in\Lambda}$ be an approximate identity for~$A$ and fix~$\lambda\in\Lambda$. Let~$c\in A$ be such that~$u_\lambda=c^*c$. Take $a\in A$ and~$\xi\in\Hilm[F]$. Then $$\alpha_p(u_{\lambda})a\otimes_A\xi=\alpha_p(c^*)\otimes_A\alpha_p(c)(a\xi)\in\alpha_p(c^*)\otimes_A U_p(\Hilm[F])\braket{U_p(\Hilm[F])}{\Hilm[F]}.$$ Using that $U_p(\Hilm[F])=\alpha_p(A)U_p(\Hilm[F])$, we deduce that~$\alpha_p(u_{\lambda})a\otimes_A\xi$ belongs to the image of~$\widetilde{V}_p$. This has closed range and hence $\alpha_p(1)a\otimes\xi$ also lies in~$\widetilde{V}_p(\Hilm[F]\otimes_B\tensor*[_{\beta_p}]{\!B}{})$. Applying again the fact that~$\widetilde{V}_p$ has closed range, we conclude that it is indeed unitary. 

We let~$V_p=\widetilde{V}_p^*$ and $V=\{V_p\}_{p\in P}$. We shall now prove that $(\Hilm[F],V)$ is a proper covariant correspondence. In this case, it suffices to show that it satisfies the coherence axiom~\eqref{eq:coherence_covariant_corres} and that~$V_e$ is the canonical isomorphism obtained from the left and right actions of~$A$ and~$B$, respectively. This latter fact follows from the identities $$\BRAKET{U_e(\xi)}{\eta}=\BRAKET{\xi}{U_e(\eta)}=\BRAKET{\xi}{\eta}=\BRAKET{U_e(\xi)}{U_e(\eta)},$$ so that $U_e=\id_{\Hilm[F]}$. The above equalities may be derived from the computation
$$\BRAKET{U_e(\xi)}{\eta}\!=\!\alpha_e(\BRAKET{U_e(\xi)}{\eta})\!=\!\BRAKET{U_e(U_e(\xi))}{U_e(\eta)}\!=\!\BRAKET{U_e(\xi)}{U_e(\eta)}\!=\!\BRAKET{\xi}{\eta}.$$

Finally, given~$a,c\in A$, $b,d\in B$ and~$\xi\in\Hilm[F]$, we have $$cU_q(aU_p(\xi)b)d=c\alpha_q(a)U_q(U_p(\xi))\beta_q(b)d=c\alpha_q(a)U_{qp}(\xi)\beta_q(b)d.$$ This leads to a commutative diagram for~$\widetilde{V}_p$, $\widetilde{V}_q$ and $\widetilde{V}_{qp}$ as in \eqref{eq:coherence_covariant_corres}. By reversing arrows, we conclude that~$(\Hilm[F],V)$ also makes such a diagram commute. This completes the proof.\end{proof}

\end{document}